\documentclass[11pt,reqno]{amsart}
\usepackage{mathrsfs}
\usepackage{amsmath}
\usepackage{amssymb, color}
\usepackage{amsthm}
\usepackage{epsfig}
\usepackage{graphicx}
\usepackage{titletoc}
\usepackage{fullpage}
\usepackage{palatino,mathpazo}
\numberwithin{equation}{section}

\newtheorem{theorem}{Theorem}[section]
\newtheorem{proposition}[theorem]{Proposition}
\newtheorem{remark}[theorem]{Remark}
\newtheorem{lemma}[theorem]{Lemma}

\theoremstyle{definition}

\theoremstyle{remark}

\def\XXint#1#2#3{{\setbox0=\hbox{$#1{#2#3}{\int}$}
		\vcenter{\hbox{$#2#3$}}\kern-.5\wd0}}

\def\XXint#1#2#3{{\setbox0=\hbox{$#1{#2#3}{\int}$ }
		\vcenter{\hbox{$#2#3$ }}\kern-.6\wd0}}

\begin{document}
\title[Infinite time bubble towers in the fractional heat equation with critical exponent]{ Infinite time bubble towers in the fractional heat equation with critical exponent}

\author[L. Cai]{Li Cai}
\address{\noindent Li ~Cai,~School of Mathematics, Southeast University, Nanjing,
	210096, P. R. China.}
\email{230198817@seu.edu.cn}
	
\author[J. Wang]{Jun Wang}
\address{\noindent Jun ~Wang,~Institute of Applied System Analysis, Jiangsu University,
	Zhenjiang, Jiangsu, 212013, P.R. China.}
\email{wangmath2011@126.com }

\author[J.C. Wei]{Jun-cheng Wei}
\address{\noindent Jun-cheng ~Wei,~Department of Mathematics, University of British Columbia,
	Vancouver, B.C., V6T 1Z2, Canada}
\email{jcwei@math.ubc.ca}

\author[W. Yang]{Wen Yang}
\address{\noindent Wen ~Yang,~Wuhan Institute of Physics and Mathematics,  Innovation Academy for Precision Measurement Science and Technology, Chinese Academy of Sciences, Wuhan 430071, P. R. China}
\email{wyang@wipm.ac.cn}	

\thanks{The research of J. Wang is partially supported by NSFC No. 11971202 and Outstanding Young foundation of Jiangsu
	Province No. BK20200042. The research of J. Wei is partially supported by NSERC of Canada. The research of W. Yang is partially supported by NSFC No. 12171456, No. 12271369 and No. 11871470.}

\begin{abstract}
In this paper, we consider the fractional heat equation with critical exponent in $\mathbb{R}^n$ for $n>6s,s\in(0,1),$
\begin{equation*}
u_t=-(-\Delta)^su+|u|^{\frac{4s}{n-2s}}u,\quad (x,t)\in \mathbb{R}^n\times\mathbb{R}.
\end{equation*}
We construct a bubble tower type solution both for the forward and backward problem by establishing the existence of the sign-changing solution with multiple blow-up at a single point  with the form
\begin{equation*}
u(x,t)=(1+o(1))\sum_{j=1}^{k}(-1)^{j-1}\mu_j(t)^{-\frac{n-2s}{2}}U\left(\frac{x}{\mu_j(t)}\right)
\quad\mbox{as}\quad t\to+\infty,
\end{equation*}
and the positive solution with  multiple blow-up at a single point  with the form
\begin{equation*}
u(x,t)=(1+o(1))\sum_{j=1}^{k}\mu_j(t)^{-\frac{n-2s}{2}}U\left(\frac{x}{\mu_j(t)}\right)
\quad\mbox{as}\quad t\to-\infty,
\end{equation*}
respectively. Here $k\ge2$ is a positive integer,
$$U(y)=\alpha_{n,s}\left(\frac{1}{1+|y|^2}\right)^{\frac{n-2s}{2}},$$
and
\begin{equation*}
\mu_j(t)=\beta_j |t|^{-\alpha_j}(1+o(1))~\mathrm{as}~t\to\pm\infty, \quad	\alpha_j=\frac{1}{2s}\left(\frac{n-2s}{n-6s}\right)^{j-1}-\frac{1}{2s},
\end{equation*}
for some certain positive numbers $\beta_j,j=1,\cdots,k.$
\end{abstract}
\maketitle

\

\

\section{Introduction}
This paper deals with the analysis of solutions that exhibit infinite time blow-up in the fractional critical heat equation
\begin{equation}
\label{eq1.1}
u_t=-(-\Delta)^su+|u|^{p-1}u,\quad  (x,t)\in \mathbb{R}^n\times\mathbb{R},
\end{equation}
where $n>2s$, $s\in(0,1)$ and $p=\frac{n+2s}{n-2s}$ is the fractional critical Sobolev exponent. Here, for any point $x\in\mathbb{R}^n$, the fractional Laplace operator $(-\Delta)^s u(x)$ is defined as
\begin{equation*}
(-\Delta)^s u(x):=C(n,s)~\mbox{P.V.}\int_{\mathbb{R}^n}\frac{u(x)-u(y)}{|x-y|^{n+2s}}\,dy
\end{equation*}
with a suitable positive normalizing constant $C(n,s)=\frac{2^{2s}s\Gamma(\frac{n+2s}{2})}{\Gamma(1-s)\pi^{\frac{n}{2}}}.$
We refer to \cite{Nezza2012} for an introduction to the fractional Laplacian and to the appendix of \cite{Davila2015} for a heuristic physical motivation in nonlocal quantum mechanics of the fractional operator. Fractional parabolic problems and related ones have attracted a lot of attentions in recent years, we refer the readers to \cite{Banerjee2018,Barrios2014,Bogdan2010,Bonforte2017,Caffarelli2011,Caffarelli2013,Caffarelli2013a,Caffarelli2010,
	Caffarelli2011a,Chen2010,Silvestre2012} and the references therein. It is well-known that $\eqref{eq1.1}$ is the formal negative $L^2$-gradient flow of the functional
\begin{equation*}	J(u):=\frac{1}{2}\int_{\mathbb{R}^n}|(-\Delta)^{\frac{s}{2}}u|^2\,dx
-\frac{n-2s}{2n}\int_{\mathbb{R}^n}|u|^{\frac{2n}{n-2s}}\,dx,\quad \frac{d}{dt}J(u(\cdot,t))=-\int_{\mathbb{R}^n}|u_t|^2\,dx,
\end{equation*}
where the function space is
$$H^s(\mathbb{R}^n):=\left\{u\in L^2(\mathbb{R}^n):
\int_{\mathbb{R}^n}|(-\Delta)^{\frac{s}{2}}u|^2\,dx<\infty\right\}.$$
In particular, if the function $u(x,t)$ is independent of $t,$ then $\eqref{eq1.1}$ is the following semi-linear elliptic problem with fractional Laplacian,
\begin{equation}
\label{eq1.2}
(-\Delta)^sU=U^{\frac{n+2s}{n-2s}}\quad \mathrm{in}\quad\mathbb{R}^n.
\end{equation}
By the classical moving plane method, Chen-Li-Ou \cite{Chen2005} and Li \cite{Li2004} have shown that
$U(y)=\alpha_{n,s}\left(\frac{1}{1+|y|^2}\right)^{\frac{n-2s}{2}}$ is the bubble solution solving problem $\eqref{eq1.2},$ where $\alpha_{n,s}$ is constant depending only on $n$ and $s$.

When $s=1,$ $t>0,$ for general $p>0,$ $\eqref{eq1.1}$ is a special form of the Fujita equation, reads as
\begin{equation}
\label{eq1.3}
\left\{\begin{array}{ll}
u_t=\Delta u+|u|^{p-1}u,\quad & (x,t)\in \mathbb{R}^n\times(0,\infty),\\
u(x,0)=u_0,\quad& x\in\mathbb{R}^n.
\end{array}\right.
\end{equation}
Many works has been devoted to studying this problem about the blow-up rates, sets and profiles since Fujita's celebrated work \cite{fujita1966}. See for example, \cite{Pino2021,Pino2020,Pino2019,Giga1985,Giga2004,Merle1997,Matano2004,Suzuki2008} and the references therein. In \cite{cdm2020}, Cortazar-del Pino-Musso obtained the following result for \eqref{eq1.3} in bounded domain $\Omega$ equipped with the Dirichlet boundary condition. Let $G(x,y)$ be the Green's function of $-\Delta$ in $\Omega$ posed with the Dirichlet boundary condition and $H(x,y)$ be its regular part. For any $k$ distinct points in $q_1,\cdots,q_k$ in $\Omega$ and define
\begin{equation}
\label{eq1.5}
\mathcal{G}(q):=
\begin{pmatrix}
H(q_1,q_1)&-G(q_1,q_2)&\dots&-G(q_1,q_k)\\
-G(q_2,q_1)&H(q_2,q_2)&\dots&-G(q_2,q_k)\\
\vdots&\vdots&\ddots&\vdots\\
-G(q_k,q_1)&-G(q_k,q_2)&\dots&H(q_k,q_k) 
\end{pmatrix}.
\end{equation}
If $\mathcal{G}(q)$ is positive definite, they proved that there exists an initial datum $u_0$ and smooth parameter functions $\xi_j(t)\to q_j$, $0<\mu_j(t)\to0$ as $t\to+\infty,$ $j=1,\cdots,k$ such that there exists an infinite time blow up solution $u_q$ with the approximate form
$$u_q\approx  \sum_{j=1}^k\alpha_n\left(\frac{\mu_j(t)}{\mu_j(t)^2+
|x-\xi_j(t)|^2}\right)^{\frac{n-2}{2}},$$
where $\mu_j(t)=\beta_jt^{-\frac{1}{n-4}}(1+o(1))$ for certain positive constants $\beta_j$. Soon after this, Musso-Sire-Wei-Zheng-Zhou \cite{Musso2018} studied the fractional heat equation with critical exponent and obtained the counter-part result. In the above mentioned two works, the blow up happens at infinite time. Concerning the finite time blow-up solutions of the classical heat equation, there are two types of blow up in literature depending on the rate:
\begin{equation}
\label{1.blowup-type}
\begin{aligned}
&\mathrm{Type~I}:&~\lim\sup\limits_{t\to T}(T-t)^{\frac{1}{p-1}}
\|u(\cdot,t)\|_{L^\infty(\mathbb{R}^n)}<+\infty,\\
&\mathrm{Type~II}:&~\lim\sup\limits_{t\to T}(T-t)^{\frac{1}{p-1}}
\|u(\cdot,t)\|_{L^\infty(\mathbb{R}^n)}=+\infty.
\end{aligned}
\end{equation}
As the elliptic partial differential equation, the Sobolev critical exponent and the Joseph-Lundgren exponent play important roles in determining the existence of type I and type II blow up solutions. When the exponent $\frac{n+2}{n-2}$ is replaced by $p\in(1,\frac{n+2}{n-2}),$ i.e., in the subcritical case. The blow up is almost completely clear, for instance, we refer the readers to \cite{fk1992,Giga1985,Giga1987,Giga2004,quittner1999}. In this case, the solution always blows up in type I in this range. Beyond the Sobolev critical exponent, Matano and Merle classified the radial blow-up solution in \cite{Matano2011} and they found that when $p$ is between the Sobolev critical exponent and Joseph-Lundgren exponent, no type II blow up can occur for radially symmetric functions. In \cite{Pino2021}, del Pino-Musso-Wei constructed non-radial type II blow up solution for $p\in\left(\frac{n+2}{n-2},p_{JL}(n)\right)$ with $p_{JL}(n)$ denoting the Joseph-Lundgren exponent. For the Sobolev critical case $p=\frac{n+2}{n-2}$, by using the energy method Schweyer \cite{Schweyer2012} constructed the radial, type II finite blow up solution in $\mathbb{R}^4$. While for smooth bounded domain $\Omega\subset\mathbb{R}^5$,
del Pino-Musso-Wei in \cite{Pino2019} found the existence of finite time type II blow up solution.
The bubbling phenomena appears a lot in many other critical contexts, for example, harmonic map heat flow, Keller-Segel chemotaxis system, Shr\"odinger map and various geometric flows. We refer the readers for instance to \cite{Daskalopoulos2018,Davila2019,Davila2020,Giga1987,Kenig2008,Krieger2009,Merle2013} and references therein.

Consider the critical heat equation \eqref{eq1.3}, i.e., $p=\frac{n+2}{n-2}$, the behavior at infinite time for finite energy solutions is deeply related to the steady states of the Yamabe equation
\begin{equation}
\label{1.yamabe}
\Delta u+|u|^{\frac{4}{n-2}}u=0\quad \mbox{in}\quad \mathbb{R}^n.
\end{equation}
We have $u(x,t)$ along a sequences $t_n\to\infty$ of Palais-Smale type for the energy $J$. Struwe's famous profile of decomposition tells us that passing to a subsequence, there are finite energy solutions $U_1,\cdots,U_k$ of \eqref{1.yamabe}, positive scalars $\mu_j(t)$ and points $\xi_j(t)$ such that for $i\neq j$
\begin{equation*}
\left|\log\frac{\mu_i(t)}{\mu_j(t)}\right|
+\frac{\xi_i(t)-\xi_j(t)}{\mu_i(t)}\to+\infty\quad 
\mathrm{as}\quad  t=t_n\to+\infty
\end{equation*}
and
\begin{equation*}
u(x,t)=\sum_{j=1}^k\frac{1}{\mu_j(t)^{\frac{n-2}{2}}}
U_j\left(\frac{x-\xi_j(t)}{\mu_j(t)}\right)+o(1)\quad 
\mathrm{as}\quad  t=t_n\to+\infty,\quad o(1)\to0\quad
\mathrm{in}\quad L^{\frac{2n}{n-2}}(\mathbb{R}^n).
\end{equation*}
In order to investigate the precise way a profile decomposition like the above formula could take place, del Pino-Musso-Wei \cite{Pino2021a} establish a family of solutions whose soliton resolution is made out of least energy steady states, all centered at the origin point and present multiple blow up with different rates in the form of a bubble tower type. To be more precisely, they proved there exists a radially symmetric initial condition $u_0(x)$ such that the solution $u$ blows-up in infinite time exactly at $0$ with a profile of the form
\begin{equation*} u(x,t)=\sum_{j=1}^k(-1)^{j-1}\mu_j^{-\frac{n-2}{2}}
\left(\frac{\mu_j(t)}{\mu_j(t)^2+|x|^2}\right)^{\frac{n-2}{2}}+o(1)\quad \mbox{as}\quad t\to+\infty,
\end{equation*}
where $\mu_j(t)=\beta_jt^{-\alpha_j}(1+o(1)),$ $\alpha_j=\frac{1}{2}\left(\frac{n-2}{n-6}\right)^{j-1}-\frac{1}{2},$ $j=1,\cdots,k.$
In the backward direction, Sun-Wei-Zhang \cite{Liming2021} constructed a radial smooth positive ancient solution which blows up at the origin for energy critical semi-linear heat equation in $\mathbb{R}^n, n \ge7$. The profile reads as
\begin{equation*} u(x,t)=(1+O(|t|^{-\epsilon}))\sum_{j=1}^k\mu_j(t)^{-\frac{n-2}{2}}
\left(\frac{\mu_j(t)}{\mu_j(t)^2+|x|^2}\right)^{\frac{n-2}{2}}\quad \mbox{as}\quad t\to-\infty,
\end{equation*}
where $\epsilon>0$ small, $\mu_j(t)=\beta_j(-t)^{-\alpha_j}(1+o(1)),$ $\alpha_j=\frac{1}{2}\left(\frac{n-2}{n-6}\right)^{j-1}-\frac{1}{2},$ $j=1,\cdots,k.$ Other results about multiple blow-up phenomenon for ellptic and parabolic equations $\eqref{eq1.3}$ can be seen in \cite{Daskalopoulos2018,Pino2003,Ge2005,Iacopetti2016,Musso2010}.

A natural question is whether we can obtain the same conclusions for the fractional critical heat equation, exhibiting the bubble tower type solutions at infinite time both in forward and backward direction. In the current article, we shall answer both two questions affirmatively.  The first main result is stated as follows.
\begin{theorem}
\label{th1.1}
Let $n>6s,$ $k\ge 2,$ $s\in(0,1),t>0.$ There exists a initial condition $u_0(x)$ such that  the solution of $\eqref{eq1.1}$ that blows up  in infinite time exactly at $0$ with a profile of the form
\begin{equation}
\label{e1.6}
u(x,t)=(1+o(1))\sum_{j=1}^{k}(-1)^{j-1}\mu_j(t)^{-\frac{n-2s}{2}}U\left(\frac{x}{\mu_j(t)}\right)
\quad\mbox{as}\quad t\to+\infty
\end{equation}
for certain positive numbers $\beta_j,j=1,\cdots,k,$ we have $\mu_j(t)=\beta_j t^{-\alpha_j}(1+o(1)),$ where
\begin{equation*}
\alpha_j=\frac{1}{2s}\left(\frac{n-2s}{n-6s}\right)^{j-1}-\frac{1}{2s},\quad j=1,\cdots,k.
\end{equation*}
\end{theorem}

In the backward direction, we have

\begin{theorem}
\label{th1.2}
Let $n>6s,k\ge 2,$ $s\in(0,1),t<0.$ There exists a positive solution of $\eqref{eq1.1}$ that blows up backward in infinite time exactly at $0$ with a profile of the form
\begin{equation*}
u(x,t)=(1+o(1))\sum_{j=1}^{k}\mu_j(t)^{-\frac{n-2s}{2}}U\left(\frac{x}{\mu_j(t)}\right)
\quad\mbox{as}\quad t\to-\infty
\end{equation*}
with $\mu_j(t)=\beta_j (-t)^{-\alpha_j}(1+o(1)),$ where $\alpha_j,\beta_j$ are introduced by Theorem $\ref{th1.1}.$
\end{theorem}
\begin{remark}
\label{r1.3}
For Theorems \ref{th1.1} and \ref{th1.2}, we find out that the proof shares a lot of similarities. Therefore we shall focus on the proof of the forward bubble tower case and state the differences for the backward case if necessary.
\end{remark}

The proof of Theorems $\ref{th1.1}$ and \ref{th1.2} is mainly based on the inner-outer gluing method. It is well-known that this method has a lot of powerful applications in various elliptic problems, see for instance \cite{Davila2015a,Pino2006,Pino2011,Pino2010} and the references therein. In recent years, this method has also been successfully applied to many parabolic equations, see \cite{Davila2019,Pino2020,Pino2019} and the references therein.

Let us close the introduction by mentioning some of the main steps and new ingredients in this work. In Section $2,$ we provide the first approximation for the solution and give the estimation on the error.  In order to improve the approximation, solvability conditions are required for the fractional elliptic linearized operator around the bubble, which yields that scaling parameter functions $\mu_j$ at the leading order. In Section $3,$ we apply the inner-outer gluing method and decompose the small perturbation in the perturbation in the form $\sum_{j=1}^k\frac{(-1)^{j-1}}{\mu_j^{\frac{n-2s}{2}}}\phi_j\left(\frac{x}{\mu_j},t\right)\eta_j+\Psi,$ where $\eta_j$ is a smooth cut-off function. Then $(\phi_j,\Psi)$ will satisfy a coupled nonlinear system: the outer problem $\Psi$ and the inner problem $\phi_j.$ In Section $4,5,$ we shall solve these two types of problems respectively and give the a priori estimate for the linearized problem. Some important but tedious computations are put in Appendix A and B. In Section $6,$ we solve the problem by the means of the fixed point argument. As \cite{Pino2021a}, the main contribution of this paper is to construct multiple sign-changing blow up solutions for the fractional heat equation with critical exponent. In the process of applying the inner-outer gluing method, we mainly overcome the following difficulties:
\begin{itemize}
\item [(1)] The study of the outer problem is already a very delicate issue for the classical case. Due to the fractional Laplacian is a nonlocal operator, $(-\Delta)^s\eta$ is no longer compactly supported for the smooth cut-off function $\eta$. Therefore, it is more complicated to  get   a point-wise estimate for the outer problem. To solve this issue, we introduce two new types of weight function for controlling the error arising from the cut-off function. These are very different from the classical setting. We believe such process might be useful in treating other blow up solutions to the fractional critical parabolic differential equation.
		
\item[(2)] Compared with the heat kernel of the classical parabolic kernel, we see that the heat kernel of the fractional heat operator $\partial_t+(-\Delta)^s$ is only algebraically decay at infinity, which makes the problem much more intricate. Indeed, in dealing with integration by the Duhamel's formula, we need to carry out a more delicate discussion.

%\item[(3)] Instead of the linear theory used in \cite{Musso2018}, we apply the one in \cite{Liming2021} for both forward and backward cases. So we have to involve some new terms in the definition of the norm for the inner perturbation $\phi_j$.
	\end{itemize}

Next we give some notations. Throughout this paper, we denote $a\lesssim b$ if $a\le Cb$ for some positive constant $C.$ Denote $a\thickapprox b$ if $a\lesssim b\lesssim a.$ $\langle x\rangle=\sqrt{1+x^2},$ $\chi(s)$ denotes a smooth cut-off function such that $0\le\chi(s)\le 1,$
\begin{equation*}
\chi(s)=
\begin{cases}
1&\text{if $s\le 1,$}\\
0&\text{if $s\ge 2,$}
\end{cases}
\end{equation*}
and for a set $\Omega\subset\mathbb{R}^n,$ $1_{\Omega}$ will denote the characteristic function defined as
\begin{equation*}
1_{\Omega}(x)=
\begin{cases}
1&\text{if $x\in\Omega,$}\\
0&\text{if $x\in\mathbb{R}^n\backslash\Omega.$}
\end{cases}
\end{equation*}

\section{A first approximation and the ansatz}
Let
\begin{equation}
\label{eq2.1}
S[u]:=u_t+(-\Delta)^s u-f(u)=0\quad (x,t)\in \mathbb{R}^n\times(t_0,\infty),
\end{equation}
where
\begin{equation*}
f(u)=|u|^{p-1}u=|u|^{\frac{4s}{n-2s}}u
\end{equation*}
and the initial time $t_0>0$ is left as a parameter which will be  taken sufficiently large. For any integer $k\ge 2,$ let us consider $k$ positive functions
\begin{equation*}
\mu_k(t)<\mu_{k-1}(t)<\cdots<\mu_1(t),\quad \text{in}\quad(t_0,\infty),
\end{equation*}
which satisfy the relations below, and the explicit behavior will be determined later
\begin{equation}\label{eq2.2}
\mu_1\to 1,\quad\frac{\mu_{j+1}(t)}{\mu_j(t)}\to 0 ~\mathrm{as}~t\to+\infty,~j=1,\cdots,k-1.
\end{equation}
In the following we denote
\begin{equation*}
\vec{\mu}=(\mu_1,\cdots,\mu_k).
\end{equation*}
Moreover, we assume that for $j=1,\cdots,k,$ $\mu_{0,j}$ is the leading order of $\mu_j$ and has the similar property as $\mu_j$ above. We define the approximate solution by
\begin{equation}
\label{eq2.3}
\bar{U}=\sum_{j=1}^{k}U_j=\sum_{j=1}^{k}\frac{(-1)^{j-1}}{\mu_j^{\frac{n-2s}{2}}}U\left(\frac{x}{\mu_j}\right),
\end{equation}
where $U$ is given by $\eqref{eq1.2}.$ Next we will get a first approximation to a solution of $\eqref{eq2.1}$ of the form $\bar{U}+\varphi_0$, where $\varphi_0$ is introduced for reducing a part of the error $S[\bar{U}]$ which is created by the interaction of the bubbles $U_j$ and $U_{j-1},$ $j=2,\cdots,k.$ To get the correction $\varphi_0,$ we will need to fix the parameters $\mu_j$ at main order around certain explicit values.

Let us consider the geometric averages
\begin{equation*}
\left\{\begin{array}{lr}
\bar{\mu}_j:=\sqrt{\mu_j\mu_{j-1}},\quad\bar{\mu}_{0,j}:=\sqrt{\mu_{0,j}\mu_{0,j-1}}\quad  &\text{ for} \quad j=2,\cdots,k,\\
\bar{\mu}_1=\bar{\mu}_{0,1}=t^\delta,\quad\bar{\mu}_{k+1}=\bar{\mu}_{0,k+1}=0,
\end{array}\right.
\end{equation*}
where $\delta>0$ is a small constant. We introduce the cut-off functions
\begin{equation}
\label{eq2.4}
\chi_j(x,t)=
\begin{cases}	\chi\left(\frac{2|x|}{\bar{\mu}_{0,j}}\right)-\chi\left(\frac{2|x|}{\bar{\mu}_{0,j+1}}\right)
&\text{$j=2,\cdots,k-1,$}\\
\chi\left(\frac{2|x|}{\bar{\mu}_{0,k}}\right)&\text{$j=k.$}
\end{cases}
\end{equation}
These cut-off functions have the property
\begin{equation}
\label{eq2.5}
\chi_j(x,t)=
\begin{cases}
0&\text{if $|x|\le \frac{1}{2}\bar{\mu}_{0,j+1},$}\\
1&\text{if $\bar{\mu}_{0,j+1}\le|x|\le \frac{1}{2}\bar{\mu}_{0,j},$}\\
0&\text{if $|x|\ge\bar{\mu}_{0,j},$}
\end{cases}~j=1,\cdots,k-1,\quad \mbox{and}\quad
\chi_k(x,t)=
\begin{cases}
1&\text{if $|x|\le \frac12\bar{\mu}_{0,k},$}\\
0&\text{if $|x|\ge\bar{\mu}_{0,k}.$}
\end{cases}
\end{equation}
We look for a correction $\varphi_0$ of the form
\begin{equation}
\label{eq2.6}
\varphi_0=\sum_{j=2}^{k}\varphi_{0,j}\chi_j,
\end{equation}
where
\begin{equation*}
\varphi_{0,j}(x,t)=\frac{(-1)^{j-1}}{\mu_j(t)^{\frac{n-2s}{2}}}\phi_{0,j}\left(\frac{x}{\mu_j},t\right)
\end{equation*}
for certain functions $\phi_{0,j}(y,t)$ defined for $y\in\mathbb{R}^n$ which we will suitably determine. Let us write
\begin{equation}
\label{eq2.7}
S(\bar{U}+\varphi_0)=S(\bar{U})+\mathcal{L}_{\bar{U}}[\varphi_0]-N_{\bar{U}}[\varphi_0],
\end{equation}
where
\begin{equation*}
\mathcal{L}_{\bar{U}}[\varphi_0]=\partial_t\varphi_0+(-\Delta)_x^s\varphi_0-f'(\bar{U})\varphi_0,
\quad\mathrm{and}\quad
N_{\bar{U}}[\varphi_0]=f(\bar{U}+\varphi_0)-f'(\bar{U})\varphi_0-f(\bar{U}).
\end{equation*}
Using the homogeneity of the function $f,$ we observe that
\begin{equation*}
\begin{split}
S(\bar{U})=\left(\sum_{j=1}^k U_j\right)_t+(-\Delta)_x^s\left(\sum_{j=1}^k U_j\right)-f\left(\sum_{j=1}^kU_j\right)=\sum_{j=1}^k\partial_t U_j-\left[f\left(\sum_{j=1}^kU_j\right)-\sum_{j=1}^kf(U_j)\right].
\end{split}
\end{equation*}
Next we write $\mathcal{L}_{\bar{U}}[\varphi_0]$ using the form of $\varphi_0$ as follows
\begin{equation*}
\begin{split}	\mathcal{L}_{\bar{U}}[\varphi_0]&=\sum_{j=2}^k
\left[(-\Delta)_x^s\varphi_{0,j}-f'(U_j)\varphi_{0,j}\right]\chi_j
-\sum_{j=2}^k\left(f'(\bar{U})-f'(U_j)\right)\varphi_{0,j}\chi_j\\ &\quad+\sum_{j=2}^k\left[(-\Delta)_x^s\chi_j\varphi_{0,j}
-\left[-(-\Delta)_x^{\frac{s}{2}}\chi_j,-(-\Delta)_x^{\frac{s}{2}}\varphi_{0,j}\right]\right]
+\sum_{j=2}^k\partial_t(\varphi_{0,j}\chi_j).
\end{split}
\end{equation*}
Here
\begin{equation*} \left[-(-\Delta)_x^{\frac{s}{2}}f(x),-(-\Delta)_x^{\frac{s}{2}}g(x)\right]:
=C_{n,s}P.V.\int_{\mathbb{R}^n}\frac{[f(x)-f(y)][g(x)-g(y)]}{|x-y|^{n+2s}}\,dy.
\end{equation*}
%with $C_{n,s}=\frac{2^{2s}s\Gamma(\frac{n+2s}{2})}{\Gamma(1-s)\pi^{\frac{n}{2}}}.$
Substituting $\bar U+\varphi_0$ into the equation \eqref{eq2.1}, we have the error expansion
\begin{equation}
\label{eq2.8}
\begin{split}
S(\bar{U}+\varphi_0)=I_1+I_2+I_3,
\end{split}
\end{equation}
where
\begin{equation*}
I_1:=\partial_t U_1-\sum_{j=2}^k(-(-\Delta)_x^s\varphi_{0,j}+f'(U_j)\varphi_{0,j}
-\partial_t U_j+f'(U_j)U_{j-1}(0))\chi_j,
\end{equation*}
\begin{equation*}
I_2:=\bar{E}_{11}-\sum_{j=2}^k(f'(\bar{U})-f'(U_j))\varphi_{0,j}\chi_j,
\end{equation*}
\begin{equation*} I_3:=\sum_{j=2}^k\left((-\Delta)_x^s\chi_j\varphi_{0,j}
-\left[-(-\Delta)_x^{\frac{s}{2}}\chi_j,-(-\Delta)_x^{\frac{s}{2}}\varphi_{0,j}\right]\right)
+\sum_{j=2}^k\partial_t(\varphi_{0,j}\chi_j)-N_{\bar{U}}[\varphi_0],
\end{equation*}
and
\begin{equation}
\label{eq2.9}	\bar{E}_{11}=-\left[f(\bar{U})-\sum_{j=1}^kf(U_j)-\sum_{j=2}^kf'(U_j)U_{j-1}(0)\chi_j
-\sum_{j=2}^k(1-\chi_j)\partial_t U_j\right].
\end{equation}
The functions $\varphi_{0,j}$ will be chosen to eliminate the main order terms in $I_1$ after conveniently restricting the range of variation of $\vec{\mu}.$
\begin{equation}
\label{eq2.10}
\begin{split}
E_j[\varphi_{0,j},\vec{\mu}]:&=-\left[-(-\Delta)_x^s\varphi_{0,j}+f'(U_j)\varphi_{0,j}
-\partial_t U_j+f'(U_j)U_{j-1}(0)\right]\\
&=\frac{(-1)^{j}}{\mu_j^{\frac{n+2s}{2}}}
\left[-(-\Delta)_{y_j}^s\phi_{0,j}+pU^{p-1}\phi_{0,j}
+\mu_j^{2s-1}\dot{\mu}_jZ_{n+1}(y_j)\right]_{y_j=\frac{x}{\mu_j}}\\
&\quad+\frac{(-1)^{j}}{\mu_j^{\frac{n+2s}{2}}}\left[-pU^{p-1}
\left(\frac{\mu_j}{\mu_{j-1}}\right)^{\frac{n-2s}{2}}U(0)\right]_{y_j=\frac{x}{\mu_j}},
\end{split}
\end{equation}
where
\begin{equation*}
Z_{n+1}(y_j)=\frac{n-2s}{2}U(y_j)+y_j \cdot\nabla U(y_j).
\end{equation*}
\medskip

In order to reduce the error term $I_1$, we consider the existence of solutions to the following  elliptic equation
\begin{equation}
\label{eq2.11}
\left\{\begin{array}{ll}
-(-\Delta)_y^s\phi+pU(y)^{p-1}\phi=-h_j(y,\mu)\quad& \text{in}\quad\mathbb{R}^n,\\
\phi\to0\quad& \text{as}\quad|y|\to\infty,
\end{array}\right.
\end{equation}
where
\begin{equation*} h_j(y,\mu):=\mu_j^{2s-1}\dot{\mu}_jZ_{n+1}(y)-pU(y)^{p-1}
\left(\frac{\mu_j}{\mu_{j-1}}\right)^{\frac{n-2s}{2}}U(0).
\end{equation*}
Then we define
\begin{equation*}
L_0[\phi]:=-(-\Delta)_y^s\phi+pU(y)^{p-1}\phi=-h_j(y,\mu).
\end{equation*}
From \cite{Davila2013}, we know that every bounded solution  of $L_0[\phi]=0$ in $\mathbb{R}^n$ is the linear combination of the functions $Z_1,\cdots,Z_{n+1},$
where
\begin{equation*}
Z_i(y):=\frac{\partial U}{\partial y_i}(y),\quad i=1,\cdots,n.
\end{equation*}
Furthermore, $\eqref{eq2.11}$ is solvable for $h_j=O(|y|^{-m}),$ $m>2s,$ if it holds that for $i=1,\cdots,n+1,$
\begin{equation*}
\int_{\mathbb{R}^n}h_j(y)Z_i(y)\,dy=0.
\end{equation*}
Using the parity of functions $U(y),\frac{\partial U}{\partial y},$ we only need to show the solvability condition
\begin{equation*}
\int_{\mathbb{R}^n}h_j(y,\mu)Z_{n+1}(y)\,dy=0.
\end{equation*}
The latter conditions hold if the parameters $\mu_j(t)$ satisfy the following relations: for all $j=2,\cdots,k,$
\begin{equation}
\label{eq2.12} \mu_1\to1,\quad\mu_j^{2s-1}\dot{\mu}_j=-c\lambda_j^{\frac{n-2s}{2}},\quad\lambda_j=\frac{\mu_j}{\mu_{j-1}},
\end{equation}
where
\begin{equation*} c=-U(0)\frac{p\int_{\mathbb{R}^n}U^{p-1}Z_{n+1}\,dy}{\int_{\mathbb{R}^n}Z_{n+1}^2\,dy}
=U(0)\frac{n-2s}{2}\frac{\int_{\mathbb{R}^n}U^p\,dy}{\int_{\mathbb{R}^n}Z_{n+1}^2\,dy}>0.
\end{equation*}
We look for a specific solution $\vec{\mu}_0=(\mu_{0,1},\cdots,\mu_{0,k})$ of $\eqref{eq2.12}$ in $(t_0,\infty)$, i.e.,
\begin{equation}
\label{eq2.13}
\mu_{0,j}(t)=\beta_jt^{-\alpha_j},\quad t\in(t_0,\infty),
\end{equation}
where
\begin{equation*}
\alpha_j=\frac{1}{2s}\left(\frac{n-2s}{n-6s}\right)^{j-1}-\frac{1}{2s},\quad j=1,\cdots,k,
\end{equation*}
and the numbers $\beta_j$ are determined by the recursive relations
\begin{equation*}
\beta_1=1,\quad\beta_j=(\alpha_jc^{-1})^{\frac{2}{n-6s}}\beta_{j-1}^{\frac{n-2s}{n-6s}}.
\end{equation*}
From $\eqref{eq2.12},$ we set $\lambda_{0,j}(t)=\frac{\mu_{0,j}}{\mu_{0,j-1}}(t)$ and
\begin{equation*}
h_j(y,\mu_0)=\lambda_{0,j}^{\frac{n-2s}{2}}\bar{h}(y),\quad\bar{h}(y)=-pU(0)U(y)^{p-1}-cZ_{n+1}(y).
\end{equation*}
Since $\int_{\mathbb{R}^n}\bar{h}Z_{n+1}(y)\,dy=0,$ there exists a solution $\bar{\phi}$ to the equation
\begin{equation*}
-(-\Delta)_y^s\bar{\phi}+pU(y)^{p-1}\bar{\phi}+\bar{h}(y)=0\quad \text{in}\quad \mathbb{R}^n.
\end{equation*}
Then we define $\phi_{0,j}$ as
\begin{equation}
\label{eq2.14}
\phi_{0,j}=\lambda_{0,j}^{\frac{n-2s}{2}}\bar{\phi}(y).
\end{equation}
Thus $\phi_{0,j}$ solves $\eqref{eq2.11}.$

In what follows we let the parameters $\mu_j$ in $\eqref{eq2.2}$ have the form $\vec{\mu}=\vec{\mu}_0+\vec{\mu}_1,$ namely
\begin{equation}
\label{eq2.15}
\mu_j(t)=\mu_{0,j}(t)+\mu_{1,j}(t),
\end{equation}
where the parameters $\mu_{1,j}(t)$ {\em to be determined} satisfy
\begin{equation}
\label{eq2.16}
|\mu_{1,j}(t)|\lesssim\mu_{0,j}(t)t^{-\sigma}
\end{equation}
for some small and fixed constant $0<\sigma<1.$ In addition, we see that for some positive number $c_j$
\begin{equation*}
\lambda_{0,j}=c_j t^{-\frac{2}{n-6s}(\frac{n-2s}{n-6s})^{j-2}}.
\end{equation*}
With these choices, the expression $E_j[\varphi_{0,j},\vec{\mu}]$ in $\eqref{eq2.10}$ can be decomposed as
\begin{equation*}
\begin{split}	&E_j[\varphi_{0,j},\vec{\mu}_0+\vec{\mu}_1]\\
&=\frac{(-1)^{j}}{\mu_j^{\frac{n+2s}{2}}}
\left[(\mu_j^{2s-1}\dot{\mu}_j-\mu_{0,j}^{2s-1}\dot{\mu}_{0,j})Z_{n+1}(y_j)
-(\lambda_j^{\frac{n-2s}{2}}-\lambda_{0,j}^{\frac{n-2s}{2}})pU^{p-1}(y_j)U(0)\right]\\
&=\frac{(-1)^{j}}{\mu_j^{\frac{n+2s}{2}}}\left[(\dot{\mu}_{0,j}[(\mu_{0,j}+\mu_{1,j})^{2s-1}
-\mu_{0,j}^{2s-1}]+\dot{\mu}_{1,j}(\mu_{0,j}+\mu_{1,j})^{2s-1})Z_{n+1}(y_j)\right]\\
&\quad+\frac{(-1)^{j}}{\mu_j^{\frac{n+2s}{2}}}
\left[-\frac{n-2s}{2}pU^{p-1}(y_j)U(0)\lambda_{0,j}^{\frac{n-2s}{2}}
\left(\frac{\mu_{1,j}}{\mu_{0,j}}-\frac{\mu_{1,j-1}}{\mu_{0,j-1}}\right)\right]\\
&\quad+\frac{(-1)^{j}}{\mu_j^{\frac{n+2s}{2}}}\left[U^{p-1}(y_j)\lambda_{0,j}^{\frac{n-2s}{2}}
O\left(\frac{|\mu_{1,j}|}{\mu_{0,j}}+\frac{|\mu_{1,j-1}|}{\mu_{0,j-1}}\right)^2\right]\\
&=\frac{(-1)^{j}}{\mu_j^{\frac{n+2s}{2}}}D_j[\vec{\mu}_1](y_j,t)
+\frac{(-1)^{j}}{\mu_j^{\frac{n+2s}{2}}}\Theta_j[\vec{\mu}_1](y_j,t),\quad y_j
=\frac{x}{\mu_j(t)},
\end{split}
\end{equation*}
where $j=2,\cdots,k,$ $\vec{\mu}_1=(\mu_{1,1},\mu_{1,2},\cdots,\mu_{1,k}),$
\begin{equation}
\label{eq2.17}
\begin{split}	D_j[\vec{\mu}_1](y_j,t):&=(\dot{\mu}_{0,j}[(\mu_{0,j}+\mu_{1,j})^{2s-1}
-\mu_{0,j}^{2s-1}]+\dot{\mu}_{1,j}[(\mu_{0,j}+\mu_{1,j})^{2s-1}])Z_{n+1}(y_j)\\
&\quad-\frac{n-2s}{2}pU^{p-1}(y_j)U(0)\lambda_{0,j}^{\frac{n-2s}{2}}
\left(\frac{\mu_{1,j}}{\mu_{0,j}}-\frac{\mu_{1,j-1}}{\mu_{0,j-1}}\right),
\end{split}
\end{equation}
\begin{equation}
\label{eq2.18} \Theta_j[\vec{\mu}_1](y_j,t):=-pU^{p-1}(y_j)\lambda_{0,j}^{\frac{n-2s}{2}}
O\left(\frac{|\mu_{1,j}|}{\mu_{0,j}}+\frac{|\mu_{1,j-1}|}{\mu_{0,j-1}}\right)^2.
\end{equation}
By straightforward computation, we get
\begin{equation}
\label{eq2.20}
\partial_t U_1=-\mu_1^{-\frac{n+2s}{2}}D_1[\vec{\mu}_1].
\end{equation}
where
\begin{equation}
\label{eq2.19}
D_1[\vec{\mu}_1](y_1,t):=(1+\mu_{1,1})^{2s-1}\dot{\mu}_{1,1}Z_{n+1}(y_1),\quad y_1=\frac{x}{\mu_1}.
\end{equation}

Based on the above argument, we define the approximate solution $u_*=u_*[\vec{\mu}]$ by
\begin{equation}
\label{eq2.21}
u_*=\bar{U}+\varphi_0,
\end{equation}
where $\bar{U}$ is defined by $\eqref{eq2.3},$ $\varphi_0$ has the form $\eqref{eq2.6},$ with $\phi_{0,j}$ defined by $\eqref{eq2.14}$ and $\mu_j$ defined by $\eqref{eq2.15}$ and $\eqref{eq2.16}.$

\begin{remark}
\label{r1.2}
As \cite{Liming2021}, we can also construct a bubble-tower type ancient solution to the fraction energy critical heat equation
\begin{equation}
\label{eq2.22}
u_t=-(-\Delta)^su+|u|^{p-1}u,\quad  (x,t)\in \mathbb{R}^n\times(-\infty,0),
\end{equation}
where $n>6s,s\in(0,1)$ and $p$ is the  fractional critical exponent $p:=\frac{n+2s}{n-2s}.$ There exists a solution which blows up backward in infinite time exactly at $0$ with the ansatz $\bar U$ written as
$\bar{U}=\sum_{j=1}^kU_j,$ where
\begin{equation*}
U_j(x,t)=\frac{1}{\mu_j^{\frac{n-2s}{2}}}U\left(\frac{x}{\mu_j}\right).
\end{equation*}
Similar to the calculation in Section $2,$ we obtain for $j=2,\cdots,k,$
\begin{equation*}		\mu_1\to1,\quad\mu_j^{2s-1}\dot{\mu}_j=c\lambda_j^{\frac{n-2s}{2}},\quad\lambda_j:=\frac{\mu_j}{\mu_{j-1}},
\end{equation*}
where $c=-U(0)\frac{p\int_{\mathbb{R}^n}U^{p-1}Z_{n+1}\,dy}{\int_{\mathbb{R}^n}Z_{n+1}^2\,dy}
=U(0)\frac{n-2s}{2}\frac{\int_{\mathbb{R}^n}U^p\,dy}{\int_{\mathbb{R}^n}Z_{n+1}^2\,dy}>0.$ The leading order term $\mu_{0,j}$ in $\mu_j$ is $\mu_{0,j}(t)=\beta_j(-t)^{-\alpha_j}.$
\end{remark}

\section{The inner-outer gluing system}
Let $u_*=u_*[\vec{\mu}_1]$ be defined in $\eqref{eq2.21}$, we shall try to find a solution to equation \eqref{eq2.1} with the form $u=u_*+\varphi$. In this case, the original problem turns to be
\begin{equation}
\label{eq3.1}
\left\{\begin{array}{ll}
S[u_*+\varphi]=\varphi_t+(-\Delta)_x^s\varphi- f'(u_*)\varphi-N_{u_*}[\varphi]+S[u_*]=0\quad & (x,t)\in \mathbb{R}^n\times(t_0,\infty),\\
\varphi(\cdot,t_0)=\varphi_*,\quad& x\in\mathbb{R}^n,
\end{array}\right.
\end{equation}
where
\begin{equation*}
N_{u_*}[\varphi]=f(u_*+\varphi)-f'(u_*)\varphi-f(u_*).
\end{equation*}
The function $\varphi_*(x)$ is an initial condition to be determined, and  we have $u_0=u_*(\cdot,t_0)+\varphi_*.$

We consider the cut-off functions $\eta_j,\zeta_j,j=1,\cdots,k,$ defined by
\begin{equation}
\label{eq3.2}
\eta_j(x,t)=\chi\left(\frac{|x|}{2R\mu_{0,j}(t)}\right)
\end{equation}
and
\begin{equation}
\label{eq3.3}
\zeta_j(x,t)=\begin{cases}
\chi\left(\frac{|x|}{R\mu_{0,j}}\right)-\chi\left(\frac{R|x|}{\mu_{0,j}}\right)&\text{if $j=1,\cdots,k-1,$}\\
\chi\left(\frac{|x|}{R\mu_{0,k}}\right)&\text{if $j=k.$}\\
\end{cases}
\end{equation}
We find that $\eta_j\zeta_j=\zeta_j,$ because
\begin{equation*}
\eta_j(x,t)=
\begin{cases}
1&\text{if $|x|\le 2R\mu_{0,j}(t),$}\\
0&\text{if $|x|\ge 4R\mu_{0,j}(t),$}
\end{cases}
\end{equation*}
and for $j=1,\cdots,k-1,$
\begin{equation}
\label{eq3.4}
\zeta_j(x,t)=
\begin{cases}
1&\text{if $2R^{-1}\mu_{0,j}(t)\le|x|\le R\mu_{0,j}(t),$}\\
0&\text{if $|x|\ge 2R\mu_{0,j}(t)~\,\mbox{or}~\,|x|\le R^{-1}\mu_{0,j}(t),$}
\end{cases}
\quad
\zeta_k(x,t)=
\begin{cases}
1&\text{if $|x|\le R\mu_{0,k}(t),$}\\
0&\text{if $|x|\ge 2R\mu_{0,k}(t).$}
\end{cases}
\end{equation}
Here $R$ is chosen to be a $t$-dependent, slowly growing function, say
\begin{equation}
\label{eq3.6}
R(t)=t^\epsilon,\quad t>t_0,
\end{equation}
where $\epsilon>0$ will be later on fixed sufficiently small.

Let $\phi_j(y_j,t),j=1,\cdots,k$ be defined in $B_{8R}\times(t_0,\infty)$ and a function $\Psi(x,t)$ be defined in $\mathbb{R}^n\times(t_0,\infty).$ Then we look for a solution $\varphi$ of $\eqref{eq3.1}$ of the form
\begin{equation}
\label{eq3.7}
\varphi=\sum_{j=1}^k\varphi_j\eta_j+\Psi,
\end{equation}
where
\begin{equation*}
\varphi_j(x,t)=\frac{(-1)^{j-1}}{\mu_j^{\frac{n-2s}{2}}}\phi_j(y_j,t).
\end{equation*}
Let us substitute $\varphi$ given by $\eqref{eq3.7}$ into equation $\eqref{eq3.1}$
\begin{equation*}
\begin{split}	&S[u_*+\varphi]\\&=\varphi_t+(-\Delta)_x^s\varphi-f'(u_*)\varphi-N_{u_*}[\varphi]+S[u_*]\\
&=\sum_{j=1}^k\partial_t\varphi_j\eta_j+\sum_{j=1}^k\varphi_j\partial_t\eta_j+\Psi_t
+\sum_{j=1}^k(-\Delta)_x^s\varphi_j\eta_j+\sum_{j=1}^k\varphi_j\cdot[(-\Delta)_x^s\eta_j]\\
&\quad-\sum_{j=1}^k\left[-(-\Delta)_x^{\frac{s}{2}}\varphi_j,-(-\Delta)_x^{\frac{s}{2}}\eta_j\right]
+(-\Delta)_x^s\Psi\\&\quad-\sum_{j=1}^kf'(u_*)\varphi_j\eta_j-f'(u_*)\Psi-N_{u_*}
\left[\sum_{j=1}^k\varphi_j\eta_j+\Psi\right]+S[u_*]\\
&\quad+\sum_{j=1}^k\eta_j\zeta_jf'(U_j)\Psi-\sum_{j=1}^k\eta_j\zeta_jf'(U_j)\Psi
+\sum_{j=1}^k\eta_jf'(U_j)\varphi_j-\sum_{j=1}^k\eta_jf'(U_j)\varphi_j\\
&\quad+\sum_{j=1}^k\eta_j\frac{(-1)^{j}}{\mu_j^{\frac{n+2s}{2}}}D_j[\vec{\mu}_1]
-\sum_{j=1}^k\eta_j\frac{(-1)^{j}}{\mu_j^{\frac{n+2s}{2}}}D_j[\vec{\mu}_1]\\
&=\sum_{j=1}^k\eta_j\cdot\frac{(-1)^{j}}{\mu_j^{\frac{n+2s}{2}}}
\left(-\mu_j^{2s}\partial_t\phi_j-(-\Delta)_{y_j}^s\phi_j+pU^{p-1}\phi_j
+\zeta_jp(-1)^{j-1}U^{p-1}\mu_j^{\frac{n-2s}{2}}\Psi+D_j\right)\\
&\quad+\Psi_t+(-\Delta)_x^s\Psi-V\Psi-B[\vec{\phi}]-N[\vec{\phi},\Psi,\vec{\mu}_1]+E^{out}.
\end{split}
\end{equation*}
Here we denote for $\bar{\phi}=(\phi_1,\cdots,\phi_k),\vec{\mu}_1=(\mu_{1,1},\cdots,\mu_{1,k}),$
\begin{equation}
\label{eq3.8}
\begin{split}	B[\vec{\phi}]:&=\sum_{j=1}^k[-(-\Delta)_x^{\frac{s}{2}}\eta_j,-(-\Delta)_x^{\frac{s}{2}}\varphi_j]
+(-\partial_t\eta_j-(-\Delta)_x^s\eta_j)\varphi_j\\&\quad+\sum_{j=1}^k\eta_j(f'(u_*)-f'(U_j))\varphi_j
-\sum_{j=1}^k\dot{\mu}_j\frac{\partial\varphi_j}{\partial\mu_j}\eta_j,
\end{split}
\end{equation}
\begin{equation}
\label{eq3.9}
N[\vec{\phi},\Psi,\vec{\mu}_1]=N_{u_*}\left(\sum_{j=1}^k\varphi_j\eta_j+\Psi\right),
\end{equation}
\begin{equation}
\label{eq3.10}
V=f'(u_*)-\sum_{j=1}^k\zeta_jf'(U_j)
\end{equation}
and
\begin{equation}\label{eq3.11}
E^{out}:=S[u_*]-\sum_{j=1}^k\frac{(-1)^{j}}{\mu_j^{\frac{n+2s}{2}}}D_j[\vec{\mu}_1]\eta_j,
\end{equation}
where $D_j[\vec{\mu}_1]$ are defined in $\eqref{eq2.17}$ and $\eqref{eq2.19}.$ It is not difficult to see that $S[u_*+\varphi]=0$ if the following system of $k+1$ equations are satisfied, for $j=1,2,\cdots,k,$
\begin{equation}
\label{eq3.12}
\mu_j^{2s}\partial_t\phi_j=-(-\Delta)_{y_j}^s\phi_j+pU(y_j)^{p-1}\phi_j+H_j[\Psi,\vec{\mu}_1] \quad \text{in}\quad B_{8R}\times(t_0,\infty),
\end{equation}
\begin{equation}
\label{eq3.13}
\Psi_t=-(-\Delta)_x^s\Psi+\mathcal{G}\quad \text{in}\quad\mathbb{R}^n\times(t_0,\infty),
\end{equation}
where
\begin{equation}
\label{eq3.14}
H_j[\Psi,\vec{\mu}_1]:=\zeta_jp(-1)^{j-1}U^{p-1}\mu_j^{\frac{n-2s}{2}}\Psi+D_j[\vec{\mu}_1],
\end{equation}
\begin{equation}
\label{eq3.15}
\mathcal{G}:=V\Psi+B[\vec{\phi}]+N[\vec{\phi},\Psi,\vec{\mu}_1]-E^{out}.
\end{equation}
In the next sections we will find a solution to this system with suitable choice of parameters $\vec{\mu}_1.$

\section{The linear outer problem}
In this section, we shall get proper  priori estimates of the associated linear problem of the outer problem $\eqref{eq3.13}.$ 	We consider the solution of
\begin{equation}
\label{eq5.1}
\Psi_t=-(-\Delta)_x^s\Psi+\mathcal{G},\quad \text{in}\quad\mathbb{R}^n\times(t_0,\infty),
\end{equation}
where $\mathcal{G}$ is defined in $\eqref{eq3.15}.$
Recall that the heat kernel to the fractional heat operator $\partial_t+(-\Delta)^s$ is given by
\begin{equation}
\label{eq5.2}
K_s(x,t)=\frac{t}{{(t^{\frac{1}{s}}+|x|^2)^{\frac{n+2s}{2}}}}.
\end{equation}
Then from Lemma $\ref{lea.3}-\ref{lea.8},$ we can find a nonzero initial condition $\Psi_0$ such that the solution $\Psi$ decays at infinity with respect to  time. Moreover, by Duhamel's formal,
\begin{equation}
\label{eq5.3} |\Psi(x,t)|=|\mathcal{T}^{out,*}[\mathcal{G}]|\lesssim
\mathcal{T}^{out}[\mathcal{G}](x,t):=\int_{t}^{\infty}\int_{\mathbb{R}^n}K_s(x-y,l-t)|\mathcal{G}(y,l)|\,dydl,
\end{equation}
where 
\begin{align*}
\mathcal{T}^{out,*}[\mathcal{G}]:=~&\int_{\mathbb{R}^n}K_s(x-y,t_0)\Psi_0\,dy+\int_{t_0}^{t}\int_{\mathbb{R}^n}K_s(x-y,t-l)\mathcal{G}(y,l)\,dy\,dl\\
=~&\int_{t}^{\infty}\int_{\mathbb{R}^n}K_s(x-y,l-t)\mathcal{G}(y,l)\,dydl.
\end{align*}
\begin{remark}
\label{r1.3}
For the equation $\eqref{eq2.22},$ we see that the outer problem corresponds to the equation
\begin{equation*}
\Psi_t=-(-\Delta)_x^s\Psi+\mathcal{G}',\quad \text{in}\quad\mathbb{R}^n\times(-\infty,t_0'),
\end{equation*}
where $t_0'<0$ and $\mathcal{G}'$ is a suitable term in view of Section $3.$ Thus by Duhamel's formal, we know that
\begin{equation*}
\Psi:=\int_{-\infty}^{t}\int_{\mathbb{R}^n}K_s(x-y,t-l)\mathcal{G}'(y,l)\,dydl.
\end{equation*}
\end{remark}

From section $A,$ we define the following types of weights to set up a topology for solving the outer problem $\eqref{eq3.13}$
\begin{equation*}
\begin{split}
\omega_{1,1}(x,t)=\frac{t^{-1-\sigma}}{1+|x|^{2s+\alpha}}1_{\{|x|\le2\bar{\mu}_{0,1}\}}\approx t^{\gamma_1}1_{\{|x|\le 1\}}+t^{\gamma_1}|x|^{-2s-\alpha}1_{\{1\le|x|\le2\bar{\mu}_{0,1}\}},
\end{split}
\end{equation*}
\begin{equation*}
\omega_{1,1}'(x,t)=t^{\gamma_1}\bar{\mu}_{0,1}^{n-2s-\alpha}|x|^{-s-n}1_{\{\bar{\mu}_{0,1}\le|x|\le t^{\frac{1}{2s}}\}}+t^{\gamma_1}|x|^{-n-2s}1_{\{|x|\ge t^{\frac{1}{2s}}\}},
\end{equation*}
\begin{equation*}
\begin{split}	\omega_{1,j}=\frac{t^{-\sigma}}{\mu_{0,j}^{\frac{n+2s}{2}}}\frac{\lambda_{0,j}^{\frac{n-2s}{2}}}
{1+\left(\frac{|x|}{\mu_{0,j}}\right)^{2s+\alpha}}1_{\{|x|\le2\bar{\mu}_{0,j}\}}
\approx\mu_{0,j}^{-2s}t^{\gamma_j}1_{\{|x|\le\mu_{0,j}\}}+\mu_{0,j}^\alpha t^{\gamma_j}|x|^{-2s-\alpha}1_{\{\mu_{0,j}\le|x|\le2\bar{\mu}_{0,j}\}},
\end{split}
\end{equation*}
\begin{equation*}
\omega_{1,j}'=\mu_{0,j}^n t^{\gamma_j}|x|^{-n-2s}1_{\{|x|\ge\bar{\mu}_{0,j}\}},\quad
\omega_{1,j}''=\bar{\mu}_{0,j}^{n}t^{\gamma_j}|x|^{-n-2s}1_{\{|x|\ge\bar{\mu}_{0,j}\}},
\end{equation*}
where $0<\alpha<s,$ $\sigma$ is small constant, $\gamma_1=-1-\sigma$ and $\gamma_j=\frac{n-2s}{2}\alpha_{j-1}-\sigma$ for $j=2,\cdots,k.$
\begin{equation*} \omega_{2,1}(x,t)=t^{-\sigma}\mu_{0,2}^{\frac{n}{2}-2s}\mu_{0,1}^{-s}|x|^{2s-n}
1_{\{\bar{\mu}_{0,2}\le|x|\le1\}},
\end{equation*}
\begin{equation*} \omega_{2,j}(x,t)=t^{-\sigma}\mu_{0,j+1}^{\frac{n}{2}-2s}\mu_{0,j}^{-s}|x|^{2s-n}
1_{\{\bar{\mu}_{0,j+1}\le|x|\le\bar{\mu}_{0,j}\}},
\end{equation*}
for $j=2,\cdots,k-1.$
\begin{equation*}
\omega_3(x,t)=t^{\delta(n-4s)}\cdot t^{-1-\sigma}|x|^{2s-n}1_{\{|x|\ge\bar{\mu}_{0,1}\}},
\end{equation*}
where $\delta(n-4s)<\sigma.$

Applying Lemma $\ref{leb.1}-\ref{leb.5},$ we have the following Lemma.
\begin{lemma}
\label{le5.1}
We have the following estimates:
\begin{equation}\label{eq5.4}
\mathcal{T}^{out}[\omega_{1,1}]\lesssim\omega_{1,1}^*:=
\begin{cases}
t^{\gamma_1}&\text{if $|x|\le 1,$}\\
t^{\gamma_1}|x|^{-\alpha}&\text{if $1\le|x|\le2\bar{\mu}_{0,1},$}\\
t^{\gamma_1+\delta(n-2s-\alpha)}|x|^{2s-n}&\text{if $|x|\ge2\bar{\mu}_{0,1},$}
\end{cases}
\end{equation}
\begin{equation}\label{eq5.5}
\mathcal{T}^{out}[\omega_{1,1}']\lesssim(\omega_{1,1}')^*:=
\begin{cases}
t^{\gamma_1-\delta(s+\alpha)}&\text{if $|x|\le\bar{\mu}_{0,1}$,}\\
t^{\gamma_1+(n-3s-\alpha)\delta}|x|^{2s-n}&\text{if $|x|\ge\bar{\mu}_{0,1},$}
\end{cases}
\end{equation}
\begin{equation}\label{eq5.7}
\mathcal{T}^{out}[\omega_{1,j}]\lesssim \omega_{1,j}^*:=
\begin{cases}
t^{\gamma_j}&\text{if $|x|\le4\bar{\mu}_{0,j},$}\\
t^{\gamma_j}\mu_{0,j}^\alpha\bar{\mu}_{0,j}^{n-2s-\alpha}|x|^{2s-n}&\text{if $|x|\ge4\bar{\mu}_{0,j},$}
\end{cases}\quad j=2,\cdots,k,
\end{equation}
\begin{equation}\label{eq5.8}
\mathcal{T}^{out}[\omega_{2,1}](x,t)\lesssim\omega_{2,1}^*:=
\begin{cases}
t^{-\sigma}&\text{if $|x|\le\bar{\mu}_{0,2},$}\\
t^{-\sigma}\bar{\mu}_{0,2}^{n-4s}|x|^{4s-n}&\text{if $\bar{\mu}_{0,2}\le|x|\le 1,$}\\
t^{-\sigma}\bar{\mu}_{0,2}^{n-4s}|x|^{2s-n}&\text{if $|x|\ge 1,$}
\end{cases}
\end{equation}
\begin{equation}
\label{eq5.9}
\mathcal{T}^{out}[\omega_{2,j}]\lesssim\omega_{2,j}^*:=
\begin{cases}
t^{-\sigma}\mu_{0,j}^{s-\frac{n}{2}}&\text{if $|x|\le\bar{\mu}_{0,j+1},$}\\
t^{-\sigma}\mu_{0,j+1}^{\frac{n}{2}-2s}\mu_{0,j}^{-s}|x|^{4s-n}&\text{if $\bar{\mu}_{0,j+1}\le|x|\le\bar{\mu}_{0,j},$}\\
t^{-\sigma}\mu_{0,j+1}^{\frac{n}{2}-2s}\mu_{0,j-1}^{s}|x|^{2s-n}&\text{if $|x|\ge\bar{\mu}_{0,j},$}
\end{cases}\quad j=2,\cdots,k-1,
\end{equation}
\begin{equation}\label{eq5.10}
\mathcal{T}^{out}[\omega_3]\lesssim \omega_3^*:= t^{\delta(n-4s)}\cdot
\begin{cases}
t^{-1-\sigma}\bar{\mu}_{0,1}^{4s-n}&\text{if $|x|\le\bar{\mu}_{0,1},$}\\
t^{-1-\sigma}|x|^{4s-n}&\text{if $\bar{\mu}_{0,1}\le|x|\le t^{\frac{1}{2s}},$}\\
t^{-\sigma}|x|^{2s-n}&\text{if $|x|\ge t^{\frac{1}{2s}},$}
\end{cases}
\end{equation}
\begin{equation}\label{eq5.11}
\begin{split}
\mathcal{T}^{out}[\omega_{1,j}']\lesssim (\omega_{1,j}')^*=
\begin{cases}
\mu_{0,j}^nt^{\gamma_j}\cdot\bar{\mu}_{0,j}^{-n}&\text{if $|x|\le \bar{\mu}_{0,j},$}\\
\mu_{0,j}^nt^{\gamma_j}\cdot\bar{\mu}_{0,j}^{-2s}|x|^{2s-n}&\text{if $|x|\ge \bar{\mu}_{0,j},$}
%\\t^{-n\alpha_j+\gamma_j+s(\alpha_j+\alpha_{j-1})}\cdot|x|^{2s-n}&\text{if $|x|\ge t^{\frac{1}{2s}}.$}
\end{cases}\quad j=2,\cdots,k,
\end{split}
\end{equation}
\begin{equation}\label{eq5.13}
\begin{split}
\mathcal{T}^{out}[\omega_{1,j}'']\lesssim(\omega_{1,j}'')^*=
\begin{cases}
t^{\gamma_j}&\text{if $|x|\le\bar{\mu}_{0,j},$}\\
\bar{\mu}_{0,j}^{n-2s}t^{\gamma_j}|x|^{2s-n}&\text{if $|x|\ge\bar{\mu}_{0,j},$}
%\\t^{-\frac{(\alpha_j+\alpha_{j-1})(n-2s)}{2}+\gamma_j}\cdot|x|^{2s-n}&\text{if $|x|\ge t^{\frac{1}{2s}}.$}
\end{cases}\quad j=2,\cdots,k.
\end{split}
\end{equation}
\end{lemma}

Based on the results of Lemma $\ref{le5.1},$ for a function $h(x,t),$ we define the weighted $L^\infty$ norm $\|h\|_{\alpha,\sigma}^{out},$ $\|h\|_{\alpha,\sigma}^{out,*}$ as the following form respectively. For $(x,t)\in\mathbb{R}^n\times(t_0,\infty),$
\begin{equation*}
\|h\|_{\alpha,\sigma}^{out}:=\inf\left\{M||h(x,t)|\le M\left(\sum_{j=1}^{k}(\omega_{1,j}+\omega_{1,j}')
+\sum_{j=2}^k\omega_{1,j}''+\sum_{j=1}^{k-1}\omega_{2,j}+\omega_3\right)\right\},
\end{equation*}
\begin{equation*}
\|h\|_{\alpha,\sigma}^{out,*}:=\inf\left\{M||h(x,t)|\le M\left(\sum_{j=1}^k(\omega_{1,j}^*+(\omega_{1,j}')^*)+\sum_{j=2}^k(\omega_{1,j}'')^*
+\sum_{j=1}^{k-1}\omega_{2,j}^*+\omega_3^*\right)\right\}.
\end{equation*}
In addition, for a number $b>0$ and a function $g(t),$ we define
\begin{equation}
\label{eq4.14}
\|g\|_b:=\sup_{t\ge t_0}|t^b g(t)|.
\end{equation}
Then we introduce the norm for $\vec{\mu}_1:$
\begin{equation}
\label{eq4.15} \|\vec{\mu}_1\|_{\sigma}:=\sum_{i=1}^k(\|\dot{\mu}_{1,i}\|_{1+\alpha_i+\sigma}
+\|\mu_{1,i}\|_{\alpha_i+\sigma}),
\end{equation}
where $\sigma>0.$ In the last, applying Lemma $\ref{le5.1}$ and Lemma $\ref{lea.3}-\ref{lea.8},$ we have the following proposition.
\begin{proposition}
\label{pro5.2}
Suppose that $\sigma,\epsilon>0$ are small enough, $\|\vec{\mu}_1\|_{\sigma}\le 1$ and $t_0$ large enough. Then there exist  constants $l>0$ small enough and  $C>0,$ independent of $t_0$  such that the outer problem $\eqref{eq3.13}$ has a solution $\mathcal{T}^{out,*}[\mathcal{G}]$ in $\mathbb{R}^n\times(t_0,\infty)$ satisfying
\begin{equation*}
\|\mathcal{T}^{out,*}[\mathcal{G}]\|_{\alpha,\sigma}^{out,*}\le Ct_0^{-l}(1+\|\vec{\phi}\|_{a,\sigma}^{in}+\|\Psi\|_{\alpha,\sigma}^{out,*}
+(\|\vec{\phi}\|_{a,\sigma}^{in})^p+(\|\Psi\|_{\alpha,\sigma}^{out,*})^p).
\end{equation*}
\end{proposition}

\section{The linear inner problem}
In this section, we show a fractional linear theory motivated by \cite{Davila2019,Chen2020c,Musso2018}  for the inner problem $\eqref{eq3.12}$. In order to solve the inner problem $\eqref{eq3.12},$ we consider the following equation
\begin{equation}
\label{eq4.1}
\mu^{2s}\phi_t=-(-\Delta)_y^s\phi+pU^{p-1}\phi+h(y,t),\quad \text{in}\quad B_{8R}(0)\times(t_0,\infty).
\end{equation}
We set
\begin{equation*}
\tau(t)=\tau_0+\int_{t_0}^{t}\mu^{-2s}\,dl
\end{equation*}
Then equation $\eqref{eq4.1}$ is transformed as
\begin{equation}
\label{eq4.2}
\phi_{\tau}=-(-\Delta)_y^s\phi+pU^{p-1}\phi+h(y,\tau),\quad\tau_0\le\tau,~|y|\le 8R.
\end{equation}
Recall that the linearized operator
\begin{equation*}
L_0:=-(-\Delta)^s+pU^{p-1}
\end{equation*}
has a only positive eigenvalue $\mu_0$ such that
\begin{equation*}
L_0(Z_0)=\mu_0 Z_0,\quad Z_0\in L^\infty(\mathbb{R}^n),
\end{equation*}
where the corresponding eigenfunction $Z_0$ is radially symmetric and
\begin{equation}
\label{eq4.3}
Z_0(y)\sim|y|^{-n-2s}\quad \text{as}\quad |y|\to+\infty,
\end{equation}
see \cite{Frank2015} for instance. Multiplying equation $\eqref{eq4.2}$ by $Z_0$ and integrating over $\mathbb{R}^n,$ we obtain that
\begin{equation*}
\dot{p}(\tau)-\mu_0p(\tau)=q(\tau),
\end{equation*}
where
\begin{equation*}
p(\tau)=\int_{\mathbb{R}^n}\phi(y,\tau)Z_0(y)\,dy
\quad\mbox{and}\quad
q(\tau)=\int_{\mathbb{R}^n}h(y,\tau)Z_0(y)\,dy.
\end{equation*}
Then we get
\begin{equation*}
p(\tau)=-\int_{\tau}^{\infty}e^{\mu_0(\tau-l)}q(l)\,dl.
\end{equation*}
As a consequence, the initial value $\phi(y,\tau_0)$ is determined by the equation below
\begin{equation*}
\int_{\mathbb{R}^n}\phi(y,\tau_0)Z_0(y)\,dy= e_0[h]:=\int_{\tau_0}^{\infty}e^{\mu_0 (\tau_0-l)}\int_{\mathbb{R}^n}h(y,l)Z_0(y)\,dydl.
\end{equation*}
Therefore, we consider the associated linear Cauchy problem of the inner problem $\eqref{eq3.12}$
\begin{equation}
\label{eq4.4}
\left\{\begin{array}{ll}
\phi_\tau=-(-\Delta)_y^s\phi+pU^{p-1}(y)\phi+h(y,\tau),\quad & (y,t)\in B_{8R}(0)\times(\tau_0,\infty),\\
\phi(y,\tau_0)=e_0Z_0(y),\quad& y\in B_{8R}(0).
\end{array}\right.
\end{equation}
Defining
\begin{equation*}
\|h\|_{a,\nu}:=\sup_{y\in B_{8R},\tau>\tau_0}\tau^\nu(1+|y|^a)|h(y,\tau)|.
\end{equation*}
In the sequel, we consider $h=h(y,\tau)$ as a function in the whole space $\mathbb{R}^n$ with zero extension outside of $B_{8R}$ for all $\tau>\tau_0.$ By the proof of Proposition $5.1$ in \cite{Chen2020c,Musso2018}, we can obtain a better estimate as follows.
\begin{proposition}
\label{pro4.1}
Assume $2s<a<n-2s,\nu>0,$ $\|h\|_{2s+a,\nu}<+\infty$ and
\begin{equation}
\label{eq4.8}
\int_{B_{8R}}h(y,\tau)Z_i(y)\,dy=0,~\forall\tau\in(\tau_0,\infty),\quad  i=1,\cdots,n+1.
\end{equation}
For sufficiently large $R,$ there exist $\phi=\phi[h](y,\tau)$ and $e_0=e_0[h](\tau)$ solving $\eqref{eq4.4}$ with
\begin{equation}
\label{eq4.5} (1+|y|^s)\left(\int_{\mathbb{R}^n}\frac{|\phi(y,\tau)-\phi(x,\tau)|^2}{|y-x|^{n+2s}}\,dx\right)^{\frac{1}{2}}
+(1+|y|)|\nabla_{y}\phi|\chi_{B_{8R}(0)}+|\phi(y,\tau)|\lesssim\tau^{-\nu}(1+|y|)^{-a}\|h\|_{2s+a,\nu}
\end{equation}
and
\begin{equation}
\label{eq4.6}
|e_0[h]|\lesssim\|h\|_{2s+a,\nu},
\end{equation}
for $(y,\tau)\in\mathbb{R}^n\times(\tau_0,\infty).$
\end{proposition}
\begin{remark}
\label{r1.4}
For $\tau'\in(-\infty,\tau_0'),$ $\tau_0'$ is very negative, similar to the proof of Proposition \ref{pro4.1}, we have the following results concerning the ancient solution to the fractional heat equation. Consider
\begin{equation*}
\partial_{\tau'}\phi=-(-\Delta)_y^s\phi+pU^{p-1}(y)\phi+h'(y,\tau'),\quad \text{in}\quad B_{8R'}\times (-\infty,\tau_0').
\end{equation*}
For $R'$ is large enough, we have
\begin{equation*}
\begin{aligned}
&(1+|y|^s)\left(\int_{\mathbb{R}^n}\frac{|\phi(y,\tau')-\phi(x,\tau')|^2}{|y-x|^{n+2s}}\,dx\right)^{\frac{1}{2}}
+(1+|y|)|\nabla_{y}\phi|\chi_{B_{8R}(0)}+|\phi(y,\tau')|\\
&\lesssim(-\tau')^{-\nu}(1+|y|)^{-a}\|h\|_{2s+a,\nu},
\end{aligned}
\end{equation*}
where
\begin{equation*}
\|h\|_{a,\nu}':=\sup_{y\in B_{8R'},\tau'<\tau_0'}(-\tau')^\nu(1+|y|^a)|h(y,\tau')|.
\end{equation*}
\end{remark}
Next we will formulate the inner problem $\eqref{eq3.12}$ for the functions $\phi_j(y,t)$ using the setting introduced in Proposition $\ref{pro4.1}.$ Let us write problem $\eqref{eq3.12}$ in the form
\begin{equation}
\label{eq4.7} \mu_j^{2s}\partial_t\phi_j=-(-\Delta)_{y_j}^s\phi_j+pU(y_j)^{p-1}\phi_j+H_j[\Psi,\vec{\mu}_1](y_j,t),\quad \text{in}\quad B_{8R}\times(t_0,+\infty),\,j=1,2,\cdots,k,
\end{equation}
where $H_j[\Psi,\vec{\mu}_1](y_j,t)$ is defined in $\eqref{eq3.14}.$ First we modify the right hand side of $\eqref{eq4.7}$ to achieve the solvability conditions $\eqref{eq4.8},$ and introducing an initial condition as in $\eqref{eq4.4}.$ We consider the problem
\begin{equation}
\label{eq4.9}
\left\{\begin{array}{ll}
\mu_j^{2s}\partial_t\phi_j=-(-\Delta)_{y_j}^s\phi_j+pU^{p-1}(y_j)\phi_j+\mathcal{H},& (y_j,t)\in B_{8R}(0)\times(\tau_0,\infty),\\
\phi_j(y_j,t_0)=e_0Z_0(y_j),& y_j\in B_{8R}(0),
\end{array}\right.
\end{equation}
where
\begin{equation*}
\mathcal{H}:=H_j[\Psi,\vec{\mu}_1](y_j,t)-\sum_{i=1}^{n+1}d_{j,i}[\Psi,\vec{\mu}_1]Z_i,
\quad\mbox{and}\quad	d_{j,i}[\Psi,\vec{\mu}_1]:=\frac{\int_{B_{8R}}H_j[\Psi,\vec{\mu}_1](y_j,t)Z_i\,dy_j}{\int_{B_{8R}}Z_i^2\,dy_j}.
\end{equation*}
Let us denote by $\mathcal{T}_{\mu_j}^{in}$ the linear operator in Proposition $\ref{pro4.1}.$ Then $\eqref{eq4.9}$ is solved if the following equation holds
\begin{equation}
\label{eq4.10}
\phi_j=\mathcal{T}_{\mu_j}^{in}[H_j[\Psi,\vec{\mu}_1]],\quad j=1,\cdots,k.
\end{equation}
In  order to find a suitable solution to the original inner problem, we have to consider the following balanced condition
\begin{equation}
\label{eq4.11}
d_{j,i}[\Psi,\vec{\mu}_1]=0~\mathrm{for}~t\in(t_0,+\infty),\quad
i=1,\cdots,n+1,~j=1,\cdots,k.
\end{equation}
To solve $\eqref{eq4.11},$ it is enough to consider the indices with $i=n+1$, since the other ones are automatically zero due to the parity property. Then we have the following lemma

\begin{lemma}
\label{le4.1}
The equation $\eqref{eq4.11}$ is equivalent to
\begin{equation}
\label{eq4.13}
\left\{\begin{array}{lr}
\dot{\mu}_{1,1}=M_1[\Psi,\vec{\mu}_1](t),\\ \dot{\mu}_{1,j}+\frac{n-6s+2}{2}\frac{\alpha_j}{t}\mu_{1,j}
-\frac{n-2s}{2}\frac{\alpha_j}{t}\lambda_{0,j}\mu_{1,j-1}=M_j[\Psi,\vec{\mu}_1](t),\quad j=2,\cdots,k,
\end{array}\right.
\end{equation}
where
\begin{equation*}	M_1[\Psi,\vec{\mu}_1](t):=-\frac{\mu_1^{\frac{n-2s}{2}}}{(1+\mu_{1,1})^{2s-1}}
\frac{\int_{B_{8R}}\zeta_1pU^{p-1}\Psi\,dy_1}{\int_{B_{8R}}Z_{n+1}^2\,dy_1}
\end{equation*}
and
\begin{equation*}
\begin{split}		M_j[\Psi,\vec{\mu}_1](t):&=-\frac{\mu_j^{\frac{n-2s}{2}}}{\mu_{0,j}^{2s-1}}
\frac{\int_{B_{8R}}\zeta_jp(-1)^{j-1}U^{p-1}Z_{n+1}\Psi\,dy_j}{\int_{B_{8R}}Z_{n+1}^2\,dy_j}
-\dot{\mu}_{0,j}o\left(\frac{\mu_{1,j}}{\mu_{0,j}}\right)\\&\quad	-\dot{\mu}_{1,j}\left((2s-1)\frac{\mu_{1,j}}{\mu_{0,j}}
+o\left(\frac{\mu_{1,j}}{\mu_{0,j}}\right)\right)
-O(R^{-2s})\frac{\lambda_{0,j}^{\frac{n-2s}{2}}}{\mu_{0,j}^{2s-1}}
\left(\frac{\mu_{1,j}}{\mu_{0,j}}-\frac{\mu_{1,j-1}}{\mu_{0,j-1}}\right).
\end{split}
\end{equation*}
\end{lemma}

\begin{proof}
For $j=1,$ we get
\begin{equation*}
\begin{split}		d_{1,n+1}&=\frac{\int_{B_{8R}}(\zeta_1pU^{p-1}\mu_1^{\frac{n-2s}{2}}\Psi+D_1)Z_{n+1}\,dy_1}
{\int_{B_{8R}}Z_{n+1}^2\,dy_1}\\&
=\frac{(\int_{B_{8R}}\zeta_1pU^{p-1}\mu_1^{\frac{n-2s}{2}}\Psi
+(1+\mu_{1,1})^{2s-1}\dot{\mu}_{1,1}Z_{n+1})Z_{n+1}\,dy_1}{\int_{B_{8R}}Z_{n+1}^2\,dy_1}\\
&=(1+\mu_{1,1})^{2s-1}\dot{\mu}_{1,1}+\frac{\int_{B_{8R}}\zeta_1pU^{p-1}\mu_1^{\frac{n-2s}{2}}\Psi Z_{n+1}\,dy_1}{\int_{B_{8R}}Z_{n+1}^2\,dy_1},
\end{split}
\end{equation*}
which implies that $d_{1,n+1}=0$ is equivalent to
\begin{equation*}
\dot{\mu}_{1,1}=M_1[\Psi,\vec{\mu}_1](t).
\end{equation*}
For $j=2,\cdots,k,$ we have
\begin{equation*}
\begin{split}		d_{j,n+1}[\Psi,\vec{\mu}_1]&=\frac{\int_{B_{8R}}(\zeta_jp(-1)^{j-1}U^{p-1}\mu_j^{\frac{n-2s}{2}}\Psi+D_j)
Z_{n+1}\,dy_j}{\int_{B_{8R}}Z_{n+1}^2\,dy_j}\\
&=\frac{\int_{B_{8R}}(\zeta_jp(-1)^{j-1}U^{p-1}\mu_j^{\frac{n-2s}{2}}\Psi+\Pi\cdot Z_{n+1}(y_j))Z_{n+1}\,dy_j}{\int_{B_{8R}}Z_{n+1}^2\,dy_j}\\
&\quad-\frac{\int_{B_{8R}}(\frac{n-2s}{2}pU^{p-1}(y_j)U(0)\lambda_{0,j}^{\frac{n-2s}{2}}
(\frac{\mu_{1,j}}{\mu_{0,j}}-\frac{\mu_{1,j-1}}{\mu_{0,j-1}}))Z_{n+1}\,dy_j}{\int_{B_{8R}}Z_{n+1}^2\,dy_j},
\end{split}
\end{equation*}
where $$\Pi:=\dot{\mu}_{0,j}[(\mu_{0,j}+\mu_{1,j})^{2s-1}-\mu_{0,j}^{2s-1}]
+\dot{\mu}_{1,j}(\mu_{0,j}+\mu_{1,j})^{2s-1}.$$
Then we have
\begin{equation*}
\begin{split}		&\dot{\mu}_{0,j}[(\mu_{0,j}+\mu_{1,j})^{2s-1}-\mu_{0,j}^{2s-1}]+\dot{\mu}_{1,j}(\mu_{0,j}+\mu_{1,j})^{2s-1}\\
&\quad-\frac{n-2s}{2}\frac{U(0)\int_{B_{8R}}pU^{p-1}Z_{n+1}\,dy_j}{\int_{B_{8R}}Z_{n+1}^2\,dy_j}
\lambda_{0,j}^{\frac{n-2s}{2}}\left(\frac{\mu_{1,j}}{\mu_{0,j}}-\frac{\mu_{1,j-1}}{\mu_{0,j-1}}\right)\\
&=-\frac{\mu_j^{\frac{n-2s}{2}}\int_{B_{8R}}\zeta_jp(-1)^{j-1}U^{p-1}\Psi Z_{n+1}\,dy_j}{\int_{B_{8R}}Z_{n+1}^2\,dy_j}.
\end{split}
\end{equation*}
Since $|Z_{n+1}(y_j)|\lesssim\langle y_j\rangle^{2s-n},$
\begin{equation*}	\int_{B_{8R}}pU^{p-1}(y_j)Z_{n+1}(y_j)\,dy_j=\int_{\mathbb{R}^n}pU^{p-1}(y_j)Z_{n+1}(y_j)\,dy_j+O(R^{-2s}),
\end{equation*}
\begin{equation*}
\int_{B_{8R}}Z_{n+1}^2\,dy_j=\int_{\mathbb{R}^n}Z_{n+1}^2\,dy_j+O(R^{4s-n}),
\end{equation*}
which yields that
\begin{equation*}
-\frac{U(0)\int_{B_{8R}}pU^{p-1}Z_{n+1}(y_j)\,dy_j}{\int_{B_{8R}}Z_{n+1}^2(y_j)\,dy_j}=c+O(R^{-2s}),
\end{equation*}
where $c$ is the positive constant defined in $\eqref{eq2.12}.$ Then we can  deduce that
\begin{equation*}
\begin{split}		&\dot{\mu}_{0,j}[(\mu_{0,j}+\mu_{1,j})^{2s-1}-\mu_{0,j}^{2s-1}]
+\dot{\mu}_{1,j}(\mu_{0,j}+\mu_{1,j})^{2s-1}\\
&\quad-\frac{n-2s}{2}\frac{U(0)\int_{B_{8R}}pU^{p-1}Z_{n+1}\,dy_j}{\int_{B_{8R}}Z_{n+1}^2\,dy_j}
\lambda_{0,j}^{\frac{n-2s}{2}}\left(\frac{\mu_{1,j}}{\mu_{0,j}}-\frac{\mu_{1,j-1}}{\mu_{0,j-1}}\right)\\
&=\dot{\mu}_{0,j}\mu_{0,j}^{2s-1}\left[\left(1+\frac{\mu_{1,j}}{\mu_{0,j}}\right)^{2s-1}-1\right]
+\dot{\mu}_{1,j}\mu_{0,j}^{2s-1}\left(1+\frac{\mu_{1,j}}{\mu_{0,j}}\right)^{2s-1}\\
&\quad+\frac{n-2s}{2}(c+O(R^{-2s}))\lambda_{0,j}^{\frac{n-2s}{2}}\left(\frac{\mu_{1,j}}{\mu_{0,j}}
-\frac{\mu_{1,j-1}}{\mu_{0,j-1}}\right)\\
&=\dot{\mu}_{0,j}\mu_{0,j}^{2s-1}\left((2s-1)\frac{\mu_{1,j}}{\mu_{0,j}}
+o\left(\frac{\mu_{1,j}}{\mu_{0,j}}\right)\right)
+\dot{\mu}_{1,j}\mu_{0,j}^{2s-1}\left(1+(2s-1)\frac{\mu_{1,j}}{\mu_{0,j}}
+o\left(\frac{\mu_{1,j}}{\mu_{0,j}}\right)\right)\\
&\quad-\frac{n-2s}{2}\dot{\mu}_{0,j}\mu_{0,j}^{2s-1}
\left(\frac{\mu_{1,j}}{\mu_{0,j}}-\frac{\mu_{1,j-1}}{\mu_{0,j-1}}\right)
+O(R^{-2s})\lambda_{0,j}^{\frac{n-2s}{2}}\left(\frac{\mu_{1,j}}{\mu_{0,j}}
-\frac{\mu_{1,j-1}}{\mu_{0,j-1}}\right)\\&=\mu_{0,j}^{2s-1}\left(\dot{\mu}_{1,j}
+\frac{n-6s+2}{2}\frac{\alpha_j}{t}\mu_{1,j}-\frac{n-2s}{2}\frac{\alpha_j}{t}\lambda_{0,j}\mu_{1,j-1}\right)\\
&\quad+\mu_{0,j}^{2s-1}\left(\dot{\mu}_{0,j}o\left(\frac{\mu_{1,j}}{\mu_{0,j}}\right)+\dot{\mu}_{1,j}
\left((2s-1)\frac{\mu_{1,j}}{\mu_{0,j}}+o\left(\frac{\mu_{1,j}}{\mu_{0,j}}\right)\right)\right)\\
&\quad+\mu_{0,j}^{2s-1}\left(O(R^{-2s})\frac{\lambda_{0,j}^{\frac{n-2s}{2}}}{\mu_{0,j}^{2s-1}}
\left(\frac{\mu_{1,j}}{\mu_{0,j}}-\frac{\mu_{1,j-1}}{\mu_{0,j-1}}\right)\right).
\end{split}
\end{equation*}
This implies that
$d_{j,n+1}=0$ is equivalent to
\begin{equation*}	\dot{\mu}_{1,j}+\frac{n-6s+2}{2}\frac{\alpha_j}{t}\mu_{1,j}
-\frac{n-2s}{2}\frac{\alpha_j}{t}\lambda_{0,j}\mu_{1,j-1}=M_j.
\end{equation*}
\end{proof}

Next we solve $\eqref{eq4.13}$ by the fixed point theorem. We reformulate $\eqref{eq4.13}$ as the following mapping. Let us define $\vec{\mathcal{S}}[\Psi,\vec{\mu}_1]=(\mathcal{S}_1[\Psi,\vec{\mu}_1],\cdots,\mathcal{S}_k[\Psi,\vec{\mu}_1])$ with
\begin{equation}
\label{eq4.19}
\begin{split}
&\mathcal{S}_1[\Psi,\vec{\mu}_1](t)=\int_{t}^{\infty}M_1[\Psi,\vec{\mu}_1](l)\,dl,\\
&\mathcal{S}_j[\Psi,\vec{\mu}_1](t)=t^{-\frac{n-6s+2}{2}\alpha_j}
\int_{t_0}^{t}l^{\frac{n-6s+2}{2}\alpha_j}\left(\frac{n-2s}{2}\frac{\alpha_j}{l}\lambda_{0,j}(l)
\mathcal{S}_{j-1}[\Psi,\vec{\mu}_1](l)+M_j[\Psi,\vec{\mu}_1](l)\right)\,dl.
\end{split}
\end{equation}

\begin{lemma}
\label{le4.3}
Assume  that $\Psi$ and $\vec{\mu}_1$ satisfy $\|\Psi\|_{\alpha,\sigma}^{out,*}<\infty,$ $\|\vec{\mu}_1\|_{\sigma}\le 1,0<\sigma<1,$ there exists $C>0$ such that for $t_0$ large enough, $\epsilon$ small enough,
\begin{equation}
\label{eq4.16}
\|\vec{\mathcal{S}}[\Psi,\vec{\mu}_1]\|_{\sigma}\le C(\|\Psi\|_{\alpha,\sigma}^{out,*}+O(R^{-2s})).
\end{equation}
\end{lemma}
\begin{proof}
For $j=1,$ we find that the support of $\zeta_1$ is contained in $\{R^{-1}\mu_{0,1}\le|x|\le 2R\mu_{0,1}\}.$ Then applying Lemma $\ref{lea.5}$ and $\ref{lea.6},$ we have
\begin{equation*}
\begin{split}		|\Psi|\lesssim(\omega_{1,1}^*+\omega_3^*+\omega_{2,1}^*+(\omega_{1,2}'')^*)\|\Psi\|_{\alpha,\sigma}^{out,*}\lesssim t^{-1-\sigma}\|\Psi\|_{\alpha,\sigma}^{out,*}.
\end{split}
\end{equation*}
By the definition of $M_1,$ we get
\begin{equation*}
|M_1[\Psi,\vec{\mu}_1]|\lesssim t^{-1-\sigma}\|\Psi\|_{\alpha,\sigma}^{out,*}.
\end{equation*}
Then using the definition of $\mathcal{S}_1,$ $\eqref{eq4.14},$ and $\eqref{eq4.15},$ we deduce that
\begin{equation}
\label{eq4.17}	\|\dot{\mathcal{S}}_1[\Psi,\vec{\mu}_1]\|_{1+\sigma}
+\|\mathcal{S}_1[\Psi,\vec{\mu}_1]\|_{\sigma}\lesssim\|\Psi\|_{\alpha,\sigma}^{out,*}.
\end{equation}
	
For $j=2,\cdots,k,$ the support of $\zeta_j$ is contained in $\{R^{-1}\mu_{0,j}\le|x|\le 2R\mu_{0,j}\}.$ Using Lemma $\ref{lea.5}$ and $\ref{lea.6},$ we deduce that
\begin{equation*}
\begin{split}
|\Psi| \lesssim &
\begin{cases}			((\omega_{1,j}'')^*+(\omega_{1,j+1}'')^*+\omega_{2,j}^*+\omega_{2,j-1}^*)\|\Psi\|_{\alpha,\sigma}^{out,*},
~&\mbox{if}\quad j=2,\cdots,k-1,\\
((\omega_{1,k}'')^*+\omega_{2,k-1}^*)\|\Psi\|_{\alpha,\sigma}^{out,*},\quad&\mbox{if}\quad j=k,
\end{cases}\\
\lesssim&~
t^{\gamma_{j}}\|\Psi\|_{\alpha,\sigma}^{out,*},\quad  j=1,\cdots,k.
\end{split}
\end{equation*}
Combining with the definition of $M_j,$ for $\epsilon$ small enough, we have
\begin{equation*}
\begin{split}
|M_j[\Psi,\vec{\mu}_1]|\lesssim t^{-\sigma}\mu_{0,j}^{1-2s}\lambda_{0,j}^{\frac{n-2s}{2}}\|\Psi\|_{\alpha,\sigma}^{out,*}
+t^{-1-\alpha_j-\sigma}O(R^{-2s})\lesssim t^{-1-\alpha_j-\sigma}(\|\Psi\|_{\alpha,\sigma}^{out,*}+O(R^{-2s})).
\end{split}
\end{equation*}
Now we shall use the induction method to show that
\begin{equation}
\label{eq4.18}		
\|\dot{\mathcal{S}}_j[\Psi,\vec{\mu}_1]\|_{1+\alpha_j+\sigma}
+\|\mathcal{S}_j[\Psi,\vec{\mu}_1]\|_{\alpha_j+\sigma}\lesssim\|\Psi\|_{\alpha,\sigma}^{out,*}+O(R^{-2s}).
\end{equation}
The case $j=1$ has been already proved by $\eqref{eq4.17}.$ Then we suppose that the conclusion holds up to $j-1$ with $j\ge2$, i.e.,
\begin{equation*}	\|\dot{\mathcal{S}}_{j-1}[\Psi,\vec{\mu}_1]\|_{1+\alpha_{j-1}+\sigma}
+\|\mathcal{S}_{j-1}[\Psi,\vec{\mu}_1]\|_{\alpha_{j-1}+\sigma}\lesssim\|\Psi\|_{\alpha,\sigma}^{out,*}
+O(R^{-2s}).
\end{equation*}
For $\sigma$ small enough we have
\begin{equation*}
\begin{split}
|\mathcal{S}_j[\Psi,\vec{\mu}_1]|\lesssim t^{-\frac{n-6s+2}{2}\alpha_j}\int_{t_0}^{t}l^{\frac{n-6s+2}{2}\alpha_j}\cdot l^{-\alpha_j-1-\sigma}\,dl(\|\Psi\|_{\alpha,\sigma}^{out,*}+O(R^{-2s}))\lesssim t^{-\alpha_j-\sigma}(\|\Psi\|_{\alpha,\sigma}^{out,*}+O(R^{-2s})).
\end{split}
\end{equation*}
Similarly, we also get
\begin{equation*}
|\dot{\mathcal{S}}_j[\Psi,\vec{\mu}_1]|\lesssim t^{-1-\alpha_j-\sigma}(\|\Psi\|_{\alpha,\sigma}^{out,*}+O(R^{-2s})).
\end{equation*}
Thus we get the conclusion holds up to $j$ and \eqref{eq4.18} is proved. As a consequence, the lemma is established.
\end{proof}

\begin{lemma}
\label{le4.4}
There exists $t_0>0$ large enough such that $t>t_0,$
\begin{equation*}
|H_j[\Psi,\vec{\mu}_1]|\lesssim \mu_{0,j}^{\frac{n-2s}{2}}t^{\gamma_j}\langle y_j\rangle^{-4s}(\|\vec{\mu}_1\|_{\sigma}+\|\Psi\|_{\alpha,\sigma}^{out,*}),\quad j=1,\cdots,k.
\end{equation*}
\end{lemma}

\begin{proof}
From the definition of $H_j,$ we get
\begin{equation*}
|D_1[\vec{\mu}_1]|\lesssim|\dot{\mu}_{1,1}Z_{n+1}(y_1)|\lesssim t^{\gamma_1}\langle y_1\rangle^{2s-n}\|\vec{\mu}_1\|_{\sigma},
\end{equation*}
\begin{equation*}
|D_j[\vec{\mu}_1]|\lesssim\lambda_j^{\frac{n-2s}{2}}t^{-\sigma}(|Z_{n+1}(y_j)|+\langle y_j\rangle^{-4s})\|\vec{\mu}_1\|_{\sigma}\lesssim\mu_{0,j}^{\frac{n-2s}{2}}t^{\gamma_j}\langle y_j\rangle^{-4s}\|\vec{\mu}_1\|_{\sigma},
\end{equation*}
for $j=2,\cdots,k.$ Similar to the discussion in Lemma $\ref{le4.3},$ we obtain that
\begin{equation*}	|\zeta_jp(-1)^{j-1}U(y_j)^{p-1}\mu_j^{\frac{n-2s}{2}}\Psi|\lesssim\mu_{0,j}^{\frac{n-2s}{2}}t^{\gamma_j}\langle y_j\rangle^{-4s}\|\Psi\|_{\alpha,\sigma}^{out,*}.
\end{equation*}
Then we complete the proof.
\end{proof}

\section{The fixed point problem}
In this section, we shall solve the system $\eqref{eq3.12}$ and $\eqref{eq3.13}$ by the fixed point argument. We reformulate the inner-outer gluing system and the orthogonal equation into the mapping $\vec{T}:$
\begin{equation*}
(\vec{\phi},\Psi,\vec{\mu}_1)=\vec{T}[\vec{\phi},\Psi,\vec{\mu}_1],
\end{equation*}
where $\vec{T}=(\vec{T}_1,T_2,\vec{T}_3),$ $\vec{T}_1=(\vec{T}_1^1,\cdots,\vec{T}_1^k),\vec{T}_3=(\vec{T}_3^1,\cdots,\vec{T}_3^k),$ with the following expressions,
\begin{equation}
\label{6.1}
\begin{cases}
\vec{T}_1^j[\Psi,\vec{\mu}_1]=\mathcal{T}_j^{in}[H_j[\Psi,\vec{\mu}_1]],\quad j=1,\cdots,k,\\T_2[\vec{\phi},\Psi,\vec{\mu}_1]
=\mathcal{T}^{out,*}[\mathcal{G}[\vec{\phi},\Psi,\vec{\mu}_1]],\\
\vec{T}_3^j[\Psi,\vec{\mu}_1]=\mathcal{S}_j[\Psi,\vec{\mu}_1],\quad j=1,\cdots,k,
\end{cases}
\end{equation}
where  $\mathcal{T}_j^{in}$ is defined in $\eqref{eq4.10},$ $\mathcal{T}^{out,*}$ is defined in $\eqref{eq5.3},$ $\mathcal{S}_j$ is introduced in $\eqref{eq4.19}.$

In order to find a solution to \eqref{6.1}, we introduce the following norm for $\vec{\phi}$:
\begin{equation*}
\|\vec{\phi}\|_{a,\sigma}^{in}:=\sum_{j=1}^k\|\phi_j\|_{j,a,\sigma}^{in},
\end{equation*}
where
$$\|\phi_j\|_{j,a,\sigma}^{in}:=\sup\limits_{t>t_0}\sup_{y\in \mathbb{R}^n}\frac{ t^{\frac{n-2s}{2}\alpha_j-\gamma_j}}{R^{a-2s}}\langle y\rangle^{a}
\left((1+|y|^s)\left(\int_{\mathbb{R}^n}\frac{|\phi(y,t)-\phi(x,t)|^2}{|y-x|^{n+2s}}\,dx\right)^{\frac{1}{2}}
+(1+|y|)|\nabla_{y}\phi|\chi_{B_{8R}(0)}+|\phi|\right).$$

Denote $\mathcal{B}: \mathcal{B}_{in}\times \mathcal{B}_{out}\times\mathcal{B}_{\mu},$
where
\begin{equation*}
\begin{split}
&\mathcal{B}_{in}:=\left\{\vec{\phi}\in [C^1(B_{8R}\times(t_0,\infty))]^k|\|\vec{\phi}\|_{\alpha,\sigma}^{in}\le 1\right\},\\&
\mathcal{B}_{out}:=\left\{\Psi\in C(\mathbb{R}^n\times(t_0,\infty))|\|\Psi\|_{\alpha,\sigma}^{out,*}\le t_0^{-\frac{l}{2}}\right\},\\&
\mathcal{B}_{\mu}:=\left\{\vec{\mu}_1\in [C^1(t_0,\infty)]^k|\|\vec{\mu}_1\|_{\sigma}\le t_0^{-\frac{l}{4}}\right\}.
\end{split}
\end{equation*}

\begin{proof}[Proof of Theorem $\ref{th1.1}$]
Firstly, we claim that $\vec{T}$ maps $\mathcal{B}$ to $\mathcal{B}$ for $t_0$ large enough. Indeed, for $(\vec{\phi},\Psi,\vec{\mu}_1)\in\mathcal{B},$ applying  Proposition $\ref{pro4.1}$ and Lemma $\ref{le4.4},$ we have
\begin{equation*}
\|\vec{T}_1[\Psi,\vec{\mu}_1]\|_{\alpha,\sigma}^{in}\le\sum_{j=1}^k C\left(t_0^{-\frac{l}{2}}+t_0^{-\frac{l}{4}}\right)\le 1
\end{equation*}
with $t_0$ large enough. In addition, applying Proposition $\ref{pro5.2},$ we deduce that
\begin{equation*}
\|T_2[\vec{\phi},\Psi,\vec{\mu}_1]\|_{\alpha,\sigma}^{out,*}\le Ct_0^{-l}\left(1+t_0^{-\frac{l}{2}}+t_0^{-p\frac{l}{2}}\right)\le t_0^{-\frac{l}{2}}.
\end{equation*}
Applying Lemma $\ref{le4.3},$ we get
\begin{equation*}
\|\vec{T}_3[\Psi,\vec{\mu}_1]\|_\sigma\le C\left(t_0^{-\frac{l}{2}}+O(R^{-2s})\right)\le t_0^{-\frac{l}{4}}.
\end{equation*}
Thus the claim is true, that is $\vec{T}:\mathcal{B}\to\mathcal{B}.$
	
Next, the existence of a fixed point in $B$ will then follow from Schauder's theorem if we establish the compactness of the operator $\vec{T}.$ Therefore, we consider any sequence $(\vec{\phi}^n,\Psi^n,\vec{\mu}_{1}^n)\in \mathcal{B},$ where $\vec{\phi}^n=(\phi_1^n,\cdots,\phi_k^n),$ $\vec{\mu}_1^n=(\mu_1^n,\cdots,\mu_k^n).$ We have to prove that the sequence $\vec{T}[\vec{\phi}^n,\Psi^n,\vec{\mu}_{1}^n]$ has a convergent subsequence. Let us consider first the sequence $\tilde{\phi}_j^n=\mathcal{T}_j^{in}[H_j[\Psi^n,\vec{\mu}_1^n]].$ We write $\bar{\phi}_j^n(y_j,\tau_n(t))=\tilde{\phi}_j^n(y_j,t),$ where $\tau_n(t)=\int_{t_0}^{t}\frac{ds}{\mu_{j,n}^{2s}}.$ Then  we see that $\bar{\phi}_j^n(y,\tau_n(t))$ satisfies
\begin{equation*}
\partial_{\tau}\bar{\phi}_j^n=-(-\Delta)_{y_j}^s\bar{\phi}_j^n+h_j^n(y_j,\tau),
\end{equation*}
where $h_j^n(y,\tau)=H_j[\Psi^n,\vec{\mu}_1^n].$ Applying the regularity estimates for fractional parabolic equations (see \cite{Silvestre_2012}),  we know that $\bar{\phi}_j^n$ are equi-continuous in compact sets of $B_{8R}\times(t_0,\infty)$ by using Lemma $\ref{le4.4}.$ Using Arezla-Ascoli theorem, we obtain that $\bar{\phi}_j^n$ will convergence uniformly in compact sets of $B_{8R}\times(t_0,\infty).$ Since $\bar{\phi}_j^n\in\mathcal{B}_{in},$ then the limit will also belong to $\mathcal{B}_{in}.$ Similarly, consider $\bar{\Psi}^n=\mathcal{T}^{out,*}[\mathcal{G}[\vec{\phi}^n,\Psi^n,\vec{\mu}_{1}^n]].$ Since $\mathcal{G}$ is uniformly bounded, $\bar{\Psi}^n$ are equi-continuous in compact sets of $\mathbb{R}^n\times(t_0,\infty).$ By Arezla-Ascoli theorem, $\bar{\Psi}^n$ converges uniformly to a function $\bar{\Psi}\in \mathcal{B}_{out}.$ In addition, consider $\mathcal{S}[\Psi^n,\vec{\mu}_1^n].$ From Lemma $\ref{le4.1},$ we see that $M_1,M_j$ are $C^1(t_0,\infty),$ which yields that $\mathcal{S}\in C^2(t_0,\infty).$ Thus $\mathcal{S}[\Psi^n,\vec{\mu}_1^n]$ has a convergent subsequence in $\mathcal{B}_{\mu}.$ By Schauder's fixed point theorem, $\vec{T}:\mathcal{B}\to\mathcal{B}$ has a fixed point $(\vec{\phi},\Psi,\vec{\mu}_1).$ Equivalently, we have constructed a bubble tower solution for $\eqref{eq1.1}$ and $u=\bar{U}+\varphi_0+\sum_{j=1}^k\varphi_j\eta_j+\Psi.$ Moreover, we have $u=\bar{U}(1+o(1)).$
\end{proof}

\appendix
\section{Appendix: Some estimates for Proposition $\ref{pro5.2}$}
In this section, we introduce the notation  $y_j=\frac{x}{\mu_j},\bar{y}_j=\frac{x}{\bar{\mu}_j},y_{0,j}=\frac{x}{\mu_{0,j}},\bar{y}_{0,j}
=\frac{x}{\bar{\mu}_{0,j}}$ for $j=1,\cdots,k.$ From the definitions of $\mu_j,\bar{\mu}_j,\mu_{0,j},\bar{\mu}_{0,j},$ we observe that $|y_j|\approx|y_{0,j}|,|\bar{y}_j|\approx|\bar{y}_{0,j}|$ for $j=1,\cdots,k.$

\begin{lemma}
\label{lea.1}
For $U_j$ defined in $\eqref{eq2.3}$ and $j=1,\cdots,k-1,$ one has
\begin{equation}\label{eqa.1}
|U_j|<|U_{j+1}|\quad in\quad\{|x|<\bar{\mu}_{j+1}\}
\end{equation}
and
\begin{equation}\label{eqa.2}
|U_j|>|U_{j+1}|\quad in\quad \{|x|>\bar{\mu}_{j+1}\}.
\end{equation}
In $\{|x|\le\bar{\mu}_{0,k}\},$
\begin{equation}\label{eqa.3}
|U_k|\gtrsim|U_{k-1}|>|U_{k-2}|>\cdots>|U_1|.
\end{equation}
In $\{|x|\ge\bar{\mu}_{0,2}\},$
\begin{equation}\label{eqa.4}
|U_1|\gtrsim |U_2|>|U_3|>\cdots>|U_k|.
\end{equation}
In $\{\bar{\mu}_{0,j+1}\le|x|\le\bar{\mu}_{0,j}\},\quad j=2,\cdots,k-1,$
\begin{equation}\label{eqa.5}
|U_j|\gtrsim|U_{j+1}|>|U_{j+2}|>\cdots|U_k|,\quad |U_j|\gtrsim|U_{j-1}|>|U_{j-2}|>\cdots|U_1|.
\end{equation}
Moreover,
\begin{equation}\label{eqa.6}
\frac{|U_{j+1}|}{|U_j|}\approx\lambda_{j+1}^{-\frac{n-2s}{2}}\langle y_{j+1} \rangle^{-(n-2s)}1_{\{|x|\le\mu_{0,j}\}}+\lambda_{j+1}^{\frac{n-2s}{2}}1_{\{|x|>\mu_{0,j}\}}\quad \text{for}\quad j=1,\cdots,k-1,
\end{equation}
\begin{equation}\label{eqa.7}
\frac{|U_{j-1}|}{|U_j|}\approx\lambda_j^{\frac{n-2s}{2}}\langle y_j\rangle^{n-2s}1_{\{|x|\le\mu_{0,j-1}\}}+\lambda_j^{-\frac{n-2s}{2}}1_{\{|x|>\mu_{0,j-1}\}}\quad \text{for}\quad j=2,\cdots,k.
\end{equation}
\end{lemma}

\begin{proof}
Recall that $U_j=(-1)^{j-1}\mu_j^{\frac{2s-n}{2}}(1+|y_j|^2)^{\frac{2s-n}{2}}.$ Consider the monotonicity of $\left|\frac{U_{j+1}}{U_j}\right|$ in the different regions. Then we can obtain $\eqref{eqa.1}-\eqref{eqa.4}.$ In addition, by  direct calculation, we find that
\begin{equation*}	\left|\frac{U_{j+1}}{U_j}\right|=\lambda_{j+1}^{\frac{n-2s}{2}}
\frac{(1+|y_j|^2)^{\frac{n-2s}{2}}}{(\lambda_{j+1}^2+|y_j|^2)^{\frac{n-2s}{2}}}
\approx\lambda_{j+1}^{-\frac{n-2s}{2}}\langle y_{j+1}\rangle^{-(n-2s)}1_{\{|x|\le\mu_{0,j}\}}+\lambda_{j+1}^{\frac{n-2s}{2}}1_{\{|x|>\mu_{0,j}\}}
\end{equation*}
for $j=1,\cdots,k-1.$ That is, $\eqref{eqa.5}$ holds. Similarly, $\eqref{eqa.6}$ holds.
\end{proof}

\begin{lemma}
\label{lea.2}
For $\varphi_0$ defined in $\eqref{eq2.6},$ one has $|\varphi_0|\lesssim\sum_{j=2}^k\lambda_j^s|U_j|\chi_j.$
\end{lemma}

\begin{proof}
Using $\eqref{eq2.6}$ and $\eqref{eq2.14},$ we deduce that
$$|\varphi_0|\lesssim\sum_{j=2}^k\mu_{j-1}^{-\frac{n-2s}{2}}\langle y_j\rangle^{-2s}\chi_j.$$
It follows from the definition of $\chi_j$ that the support of $\chi_j$ is contained in $\left\{\frac{1}{2}\lambda_{j+1}^{\frac{1}{2}}\le|y_j|\le\lambda_j^{-\frac{1}{2}}\right\}.$ Then by straightforward calculation, we have $\mu_{j-1}^{-\frac{n-2s}{2}}\langle y_j\rangle^{-2s}\chi_j\lesssim\lambda_j^s|U_j|$ in this set.
\end{proof}

\begin{lemma}
\label{lea.3}
For $\alpha\in(0,s),$ $a=2s+\alpha,$ there exist $R,t_0$ large enough such that $B[\vec{\phi}]$ defined in $\eqref{eq3.8}$ satisfies
\begin{equation*}
\|B[\vec{\phi}]\|_{\alpha,\sigma}^{out}\lesssim t_0^{-l}\|\vec{\phi}\|_{a,\sigma}^{in},
\end{equation*}
where $l$ is a small enough positive constant.
\end{lemma}

\begin{proof}
According to the definition of $\|\phi_j\|_{j,a,\sigma}^{in},$  we have
\begin{equation*}
\begin{split}			&(1+|y_j|^s)\left(\int_{\mathbb{R}^n}\frac{|\phi_j(y_j,\tau)-\phi_j(x,\tau)|^2}{|y_j-x|^{n+2s}}\,dx\right)
^{\frac{1}{2}}+(1+|y_j|)|\nabla_{y_j}\phi_j|\chi_{B_{8R}(0)}+|\phi_j(y_j,\tau)|\\
&\lesssim\mu_{0,j}^{\frac{n-2s}{2}}t^{\gamma_j}R^{a-2s}\langle y_j\rangle^{-a}\|\phi_j\|_{j,a,\sigma}^{in}.	
\end{split}
\end{equation*}
\smallskip

\noindent (i)~\ From the definition of $\varphi_j,$ we deduce that
\begin{equation*}
\begin{split}
\left|\dot{\mu}_j\frac{\partial\varphi_j}{\partial\mu_j}\eta_j\right|
&=\left|\dot{\mu}_j\mu_j^{-\frac{n-2s+2}{2}}
\left(\frac{n-2s}{2}\phi_j(y_j,t)+y_j\cdot\nabla_{y_j}\phi_j(y_j,t)\right)\eta_j\right|\\
&\lesssim|\dot{\mu}_j|\mu_j^{-\frac{n-2s+2}{2}}\eta_j\cdot\mu_{0,j}^{\frac{n-2s}{2}}t^{\gamma_j}R^{a-2s}
\langle y_j\rangle^{-a}\|\phi_j\|_{j,a,\sigma}^{in} \lesssim t_0^{-l}\omega_{1,j}\|\phi_j\|_{j,a,\sigma}^{in},
\end{split}
\end{equation*}
where $l>0$ is chosen such that $|\dot{\mu}_j\mu_j^{\frac{4s-2}{2}}R^{a-2s}|\lesssim t_0^{-l}$ for $j=1,2,\cdots,k.$
\smallskip
	
\noindent (ii)~\ Concerning the term $|-(-\Delta)^s\eta_j\varphi_j|$, by direct computation we get
\begin{equation*}
\begin{split}
|-(-\Delta)^s\eta_j\varphi_j|&\lesssim
\left|\int_{\mathbb{R}^n}\frac{\eta_j(x)-\eta_j(y)}{|x-y|^{n+2s}}\,dy\right|
\mu_j^{-\frac{n-2s}{2}}|\phi_j|\\&\lesssim\frac{1}{R^{2s}\mu_{0,j}^{2s}}
\left|\int_{\mathbb{R}^n}\frac{\chi\left(\frac{x}{2R\mu_{0,j}}\right)
-\chi\left(\frac{y}{2R\mu_{0,j}}\right)}{\left|\frac{x}{2R\mu_{0,j}}
-\frac{y}{2R\mu_{0,j}}\right|^{n+2s}}\,d\frac{y}{2R\mu_{0,j}}\right|
\mu_j^{-\frac{n-2s}{2}}\mu_{0,j}^{\frac{n-2s}{2}}t^{\gamma_j}R^{a-2s}\langle y_j\rangle^{-a}\|\phi_j\|_{j,a,\sigma}^{in}\\
&\lesssim\begin{cases}
t_0^{-l}\omega_{1,j}\|\phi_j\|_{j,a,\sigma}^{in},\quad &\mbox{if}\quad |x|\le 2\bar{\mu}_{0,j},\\
\\
t_0^{-l}\omega_{1,j}'\|\phi_j\|_{j,a,\sigma}^{in}, &\mbox{if}\quad |x|\ge 2\bar{\mu}_{0,j}.
\end{cases}
\end{split}
\end{equation*}
\smallskip

\noindent (iii)~\ For the Lie bracket of $-(-\Delta)^{\frac{s}{2}}\eta_j$ and $-(-\Delta)^{\frac{s}{2}}\varphi_j$ we get that
\begin{equation*}
\begin{split}
\left|[-(-\Delta)^{\frac{s}{2}}\eta_j,-(-\Delta)^{\frac{s}{2}}\varphi_j]\right|
&\lesssim\left[\int_{\mathbb{R}^n}
\left(\frac{\eta_j(x)-\eta_j(y)}{|x-y|^{\frac{n}{2}+s}}\right)^2\,dy\right]^{\frac{1}{2}}
\cdot\left[\int_{\mathbb{R}^n}
\left(\frac{\varphi_j(x)-\varphi_j(y)}{|x-y|^{\frac{n}{2}+s}}\right)^2\,dy\right]^{\frac{1}{2}}\\
&\lesssim\frac{1}{R^s\mu_{0,j}^s}\left[\int_{\mathbb{R}^n}\left(\frac{\chi\left(\frac{x}{2R\mu_{0,j}}\right)
-\chi\left(\frac{y}{2R\mu_{0,j}}\right)}{\left|\frac{x-y}{2R\mu_{0,j}}\right|^{\frac{n}{2}+s}}\right)^2\,
d\frac{y}{2R\mu_{0,j}}\right]^{\frac{1}{2}}\\&\quad\cdot\frac{\mu_{0,j}^{-\frac{n-2s}{2}}}{\mu_{0,j}^s}
\left[\int_{\mathbb{R}^n}\left(\frac{\phi_j\left(\frac{x}{\mu_{0,j}},t\right)
-\phi_j\left(\frac{y}{\mu_{0,j}},t\right)}{\left|\frac{x-y}{\mu_{0,j}}\right|^{\frac{n}{2}+s}}\right)^2\,
d\frac{y}{\mu_{0,j}}\right]^{\frac{1}{2}}.
\end{split}
\end{equation*}
If  $|x|\le2\bar{\mu}_{0,j},$ we deduce that
\begin{equation*}
\left|[-(-\Delta)^{\frac{s}{2}}\eta_j,-(-\Delta)^{\frac{s}{2}}\varphi_j]\right|\lesssim t_0^{-l}\omega_{1,j}\|\phi_j\|_{j,a,\sigma}^{in}.
\end{equation*}
If $|x|\ge2\bar{\mu}_{0,j},j=2,\cdots,k,$ we observe that
\begin{equation*}
\left|\int_{\mathbb{R}^n}\frac{[\eta_j(y)-\eta_j(x)]\cdot\varphi_j(x)}{|x-y|^{n+2s}}\,dy\right|\lesssim t_0^{-l}\omega_{1,j}'\|\phi_j\|_{j,a,\sigma}^{in}
\end{equation*}
and
\begin{equation*}
\left|\int_{\mathbb{R}^n}\frac{[\eta_j(y)-\eta_j(x)]\cdot\varphi_j(y)}{|x-y|^{n+2s}}\,dy\right|\lesssim t_0^{-l}\omega_{1,j}''\|\phi_j\|_{j,a,\sigma}^{in}.
\end{equation*}
So
\begin{equation*}
\left|[-(-\Delta)^{\frac{s}{2}}\eta_j,-(-\Delta)^{\frac{s}{2}}\varphi_j]\right|\lesssim t_0^{-l}\left(\omega_{1,j}'+\omega_{1,j}''\right)\|\phi_j\|_{j,a,\sigma}^{in}.
\end{equation*}
Similarly, if $|x|\ge 2\bar{\mu}_{0,1},$ we have
\begin{equation*}
\left|[-(-\Delta)^{\frac{s}{2}}\eta_1,-(-\Delta)^{\frac{s}{2}}\varphi_1]\right|\lesssim t_0^{-l}\omega_3\|\phi_1\|_{1,a,\sigma}^{in}.
\end{equation*}

\noindent (iv)~\ For $\partial_t\eta_j\varphi_j$ we have
\begin{equation*}
\begin{split}
|\partial_t\eta_j\varphi_j|\lesssim\left|\chi'\left(\frac{x}{|2R\mu_{0,j}|}\right)
\cdot\left(\frac{|x|\cdot\dot{\mu}_{0,j}R+|x|\mu_{0,j}
\cdot\dot{R}}{R^2\mu_{0,j}^2}\right)\right|\mu_{0,j}^{-\frac{n-2s}{2}}|\phi_j| t_0^{-l}\omega_{1,j}\|\phi_j\|_{j,a,\sigma}^{in}.
\end{split}
\end{equation*}

\noindent (v)~\ In the end, we study the term  $|\eta_j(f'(u_*)-f'(U_j))\varphi_j|,$ we shall give the details for $j=2,\cdots,k-1,$ while the other two cases $j=1$ or $j=k$ can be handled similarly. For fixed $j$, we notice that the support of $\eta_j$ is $\{|x|\le 4R\mu_{0,j}\}$ and we can divide it as follows
\begin{equation*}
\{x\mid|x|\le 4R\mu_{0,j}\}=\{x\mid\bar\mu_{0,j+1}\le|x|\le 4R\mu_{0,j}\}\cup\cup_{i=j+1}^k\{x\mid\bar\mu_{0,i+1}\le|x|\le \bar\mu_{0,i}\}.
\end{equation*}
In the first region $\{\bar{\mu}_{0,j+1}\le|x|\le4R\mu_{0,j}\}.$ Using Lemma $\ref{lea.1},$ Lemma $\ref{lea.2}$ and the mean value theorem, we deduce that
\begin{equation*}
\begin{split}		|f'(u_*)-f'(U_j)|&\lesssim|U_j|^{p-1}\left(\frac{|U_{j+1}|}{|U_j|}+\frac{|U_{j-1}|}{|U_j|}+\lambda_j^s\right)
\\&\lesssim\mu_j^{-2s}\langle y_j\rangle^{-4s}\cdot\left(\lambda_{j+1}^{-\frac{n-2s}{2}}\langle y_{j+1}\rangle^{2s-n}1_{\{\bar{\mu}_{0,j+1}\le|x|\le\mu_{0,j}\}}+\lambda_{j+1}^{\frac{n-2s}{2}}
+\lambda_j^{\frac{n-2s}{2}}\langle y_j\rangle^{n-2s}+\lambda_j^{s}\right).
\end{split}
\end{equation*}
Then
\begin{equation*}
\begin{split}
|(f'(u_*)-f'(U_j))\varphi_j\eta_j|1_{\{\bar{\mu}_{0,j+1}\le|x|\le 4R\mu_j\}}&\lesssim\mu_j^{-2s}\mu_j^{-\frac{n-2s}{2}}\lambda_j^{\frac{n-2s}{2}}
t^{-\sigma}\lambda_{j+1}^{-\frac{n-2s}{2}}\langle y_{j+1}\rangle^{2s-n}R^{a-2s}\|\phi_j\|_{j,a,\sigma}^{in}\\&\quad+(\lambda_{j+1}^{\frac{n-2s}{2}}
+\lambda_j^{\frac{n-2s}{2}}R^{n-2s}+\lambda_j^s)R^{a-2s}\omega_{1,j}\|\phi_j\|_{j,a,\sigma}^{in}\\
&\lesssim t_0^{-l}(\omega_{2,j}+\omega_{1,j})\|\phi_j\|_{j,a,\sigma}^{in}.
\end{split}
\end{equation*}
While in the region $\{\bar{\mu}_{0,i+1}\le|x|\le\bar{\mu}_{0,i}\},i=j+1,\cdots,k.$ Using Lemma $\ref{lea.1}$ and Lemma $\ref{lea.2}$ again, we have
\begin{equation*}
|f'(u_*)-f'(U_j)|\lesssim|U_i|^{p-1}\approx\mu_i^{-2s}\langle y_i\rangle^{-4s}
\end{equation*}
and
\begin{equation*}
|(f'(u_*)-f'(U_j))\varphi_j\eta_j|\lesssim t_0^{-l}\omega_{1,i}\|\phi_j\|_{j,a,\sigma}^{in}.
\end{equation*}
From (i)-(v), we derive the lemma.
\end{proof}

\begin{lemma}
\label{lea.4}
For $\alpha\in\left(0,\min\{s,\frac{n-6s}{2}\}\right),$  $\|\vec{\mu}_1\|_{\sigma}<1,$ there exists $\sigma,\epsilon$ small enough and   $t_0$ large enough such that
\begin{equation*}
E^{out}\lesssim t_0^{-l}\left(\sum_{j=1}^k\omega_{1,j}+\sum_{j=1}^{k-1}\omega_{2,j}+\omega_3\right).
\end{equation*}
\end{lemma}
\begin{proof}
By $\eqref{eq2.8}$ and $\eqref{eq3.11},$ we rewrite $E^{out}$ as
\begin{equation*}
\begin{split}		
E^{out}=S[u_*]-\sum_{j=1}^k\frac{(-1)^{j}}{\mu_j^{\frac{n+2s}{2}}}D_j[\vec{\mu}_1]\eta_j=
\bar{E}_{11}+\bar{E}_2+\bar{E}_3+\bar{E}_4+\bar{E}_5,
\end{split}
\end{equation*}
where $\bar{E}_{11}$ is defined in $\eqref{eq2.9},$ and
\begin{equation*}	\bar{E}_2:=-\mu_1^{-\frac{n+2s}{2}}D_1[\vec{\mu}_1](1-\eta_1)
+\sum_{j=2}^k\frac{(-1)^{j}}{\mu_j^{\frac{n+2s}{2}}}D_j[\vec{\mu}_1](\chi_j-\eta_j)
+\sum_{j=2}^k\frac{(-1)^{j}}{\mu_j^{\frac{n+2s}{2}}}\Theta_j[\vec{\mu}_1]\chi_j,
\end{equation*}
\begin{equation*}
\bar{E}_3:=-\sum_{j=2}^k(f'(\bar{U})-f'(U_j))\varphi_{0,j}\chi_j,\quad 	\bar{E}_4:=-N_{\bar{U}}[\varphi_0],
\end{equation*}
\begin{equation*}	\bar{E}_5:=\sum_{j=2}^k\left((-\Delta)_x^s\chi_j\varphi_{0,j}
-[-(-\Delta)_x^{\frac{s}{2}}\chi_j,-(-\Delta)_x^{\frac{s}{2}}\varphi_{0,j}]\right)
+\sum_{j=2}^k\partial_t(\varphi_{0,j}\chi_j).
\end{equation*}
We shall give the estimation for each term respectively in the following

\noindent (1)~\ \textbf{Estimate of $\bar{E}_2.$} For the term $-\mu_1^{-\frac{n+2s}{2}}D_1[\vec{\mu}_1](1-\eta_1),$ we observe that the support of $1-\eta_1$ is $\{|y_{0,1}|\ge 2R\}.$ From  the assumption $\|\vec{\mu}_1\|_\sigma<1,$ we have $|\dot{\mu}_{11}|\le \|\vec{\mu}_1\|_\sigma t^{-1-\sigma}\le t^{-1-\sigma}.$ In the view of $\eqref{eq2.19},$ we deduce that
\begin{equation*}
\begin{aligned}
\left|\mu_1^{-\frac{n+2s}{2}}D_1[\vec{\mu}_1](1-\eta_1)\right|&\lesssim t^{-1-\sigma}|x|^{2s-n}1_{\{|x|\ge 2R\}}\lesssim R^{4s+\alpha-n}t^{-1-\sigma}|x|^{-2s-\alpha}1_{\{1\le|x|\le\bar{\mu}_{0,1}\}}\\
&\quad+t^{-\delta(n-4s)}t^{\delta(n-4s)}\cdot t^{-1-\sigma}|x|^{2s-n}1_{\{|x|\ge\bar{\mu}_{0,1}\}}\\&=R^{4s+\alpha-n}\omega_{1,1}+ t^{-\delta(n-4s)}\omega_3\lesssim t_0^{-l}(\omega_{1,1}+\omega_3).
\end{aligned}
\end{equation*}
For $\sum_{j=2}^k\frac{(-1)^{j}}{\mu_j^{\frac{n+2s}{2}}}D_j[\vec{\mu}_1](\chi_j-\eta_j),$ we find that  $$\mathrm{Supp}\{\chi_j-\eta_j\}=\{|x|\le\bar{\mu}_{0,j+1}\}\cup\{2R\mu_{0,j}\le|x|\le\bar{\mu}_{0,j}\}.$$
In the former set, we see that $|\chi_j-\eta_j|\le 1_{\{|x|\le\bar{\mu}_{0,j+1}\}}.$ Hence,
\begin{equation*}
\begin{split}		\left|\frac{(-1)^{j}}{\mu_j^{\frac{n+2s}{2}}}D_j[\vec{\mu}_1](\chi_j-\eta_j)\right|
&\lesssim\frac{1}{\mu_j^{\frac{n+2s}{2}}}\lambda_{0,j}^{\frac{n-2s}{2}}t^{-\sigma}(|Z_{n+1}(y_j)|+\langle y_j\rangle^{-4s}) 1_{\{|x|\le\bar{\mu}_{0,j+1}\}}\\
&\lesssim\frac{\lambda_{0,j}^{\frac{n-2s}{2}}}{\mu_{0,j}^{\frac{n+2s}{2}}}t^{-\sigma}\langle y_j\rangle^{-4s} 1_{\{|x|\le\bar{\mu}_{0,j+1}\}}\lesssim t_0^{-l}\omega_{1,j+1}.
\end{split}
\end{equation*}
In the later set, we see that $|y_{0,j}|\ge 2R.$ Then
\begin{equation*}	\left|\frac{(-1)^{j}}{\mu_j^{\frac{n+2s}{2}}}D_j[\vec{\mu}_1](\chi_j-\eta_j)\right|
\le\frac{\lambda_{0,j}^{\frac{n-2s}{2}}}{\mu_{0,j}^{\frac{n+2s}{2}}}t^{-\sigma}\langle y_j\rangle^{-4s}1_{\{2R\mu_{0,j}\le|x|\le\bar{\mu}_{0,j}\}}\lesssim t_0^{-l}\omega_{1,j}.
\end{equation*}
For the left term in $\bar E_2$, by $\eqref{eq2.18}$ we have
\begin{equation*}
\left|\frac{(-1)^{j}}{\mu_j^{\frac{n+2s}{2}}}\Theta_j[\vec{\mu}_1]\chi_j\right|\lesssim t_0^{-l}\frac{\lambda_j^{\frac{n-2s}{2}}}{\mu_j^{\frac{n+2s}{2}}}t^{-\sigma}\langle y_j\rangle^{-4s}\chi_j\lesssim t_0^{-l}\omega_{1,j}.
\end{equation*}
\noindent (2)~\ \textbf{Estimate of $\bar{E}_3.$} Concerning the support of the function $\chi_j,$ by Lemma $\ref{lea.1}$ and the mean value theorem we have
\begin{equation*}
|f'(\bar{U})-f'(U_j)|\lesssim |U_j|^{p-1}\left(\frac{|U_{j+1}|}{|U_j|}+\frac{|U_{j-1}|}{|U_j|}\right).
\end{equation*}
Using $\eqref{eq2.6}$ and $\eqref{eq2.14}$ we get that $|\varphi_{0,j}|\le\frac{\lambda_{0,j}^{\frac{n-2s}{2}}}{\mu_j^{\frac{n-2s}{2}}}\langle y_j\rangle^{-2s}$ by Green's representation formula. Then following a similar argument as in Lemma $\ref{lea.3},$ we deduce that
\begin{equation*}
|\bar{E}_3|\lesssim t_0^{-l}\left(\sum_{j=2}^k\omega_{1,j}+\sum_{j=1}^{k-1}\omega_{2,j}\right).
\end{equation*}
\noindent (3)~\ \textbf{Estimate of $\bar{E}_{4}.$} Since $p\in(1,2),$ by Lemma $\ref{lea.2},$ we deduce that
\begin{equation*}
\begin{split}		|N_{\bar{U}}[\varphi_0]|\lesssim|\varphi_0|^p\lesssim\sum_{j=2}^k\mu_{j-1}^{-\frac{n-2s}{2}p}\langle y_j\rangle^{-2sp}\chi_j^p\lesssim\sum_{j=2}^k t^{2s(\alpha_{j-1}-\alpha_j)}\langle y_j\rangle^{2s+\alpha-2sp}\mu_j^{-2s}t^{\frac{n-2s}{2}\alpha_{j-1}}\langle y_j\rangle^{-2s-\alpha}\chi_j^p.
\end{split}
\end{equation*}
If $2sp\ge 2s+\alpha,$ then $|N_{\bar{U}}[\varphi_0]|\lesssim t_0^{-l}\sum_{j=2}^k\omega_{1,j}.$
Otherwise,  we get
\begin{equation*}	|N_{\bar{U}}[\varphi_0]|\lesssim\sum_{j=2}^kt^{\frac{2s-\alpha+2sp}{2}
(\alpha_{j-1}-\alpha_j)}\mu_j^{-2s}t^{\frac{n-2s}{2}\alpha_{j-1}}\langle y_j\rangle^{-2s-\alpha}\chi_j^p\lesssim t_0^{-l}\sum_{j=2}^k\omega_{1,j}.
\end{equation*}
\noindent (4)~\ \textbf{Estimate of $\bar{E}_5.$} For $\partial_t(\varphi_{0,j}\chi_j),$ we write
\begin{equation*}
\begin{split}		|\partial_t(\varphi_{0,j}\chi_j)|\le(|\partial_t(\varphi_{0,j})\chi_j|+|\varphi_{0,j}\partial_t\chi_j|)
\lesssim t^{\frac{n-2s}{2}\alpha_{j-1}-1}\langle y_j\rangle^{-2s}(\chi_j+|\nabla_x\chi_j|)
\lesssim t_0^{-l}\omega_{1,j}.
\end{split}
\end{equation*}
For $(-\Delta)_x^s\chi_j\varphi_{0,j},$ we get
\begin{equation*}
\begin{split}		|(-\Delta)_x^s\chi_j\varphi_{0,j}|&\lesssim\int_{\mathbb{R}^n}\frac{|\chi_j(x)-\chi_j(y)|}{|x-y|^{n+2s}}\,dy
\cdot t^{\frac{n-2s}{2}\alpha_{j-1}}\langle y_j\rangle^{-2s}\\&=t^{\frac{n-2s}{2}\alpha_{j-1}}\langle y_j\rangle^{-2s}\int_{\mathbb{R}^n}\frac{\chi\left(\frac{2|x|}{\bar{\mu}_{0,j}}\right)
-\chi\left(\frac{2|y|}{\bar{\mu}_{0,j}}\right)+\chi\left(\frac{2|y|}{\bar{\mu}_{0,j+1}}\right)
-\chi\left(\frac{2|x|}{\bar{\mu}_{0,j+1}}\right)}{|x-y|^{n+2s}}\,dy  \\&=\int_{\mathbb{R}^n}\frac{\chi\left(\frac{2|x|}{\bar{\mu}_{0,j}}\right)
-\chi\left(\frac{2|y|}{\bar{\mu}_{0,j}}\right)}{|x-y|^{n+2s}}\,dy\cdot t^{\frac{n-2s}{2}\alpha_{j-1}}\langle y_j\rangle^{-2s}\\&\quad+\int_{\mathbb{R}^n}\frac{\chi\left(\frac{2|y|}{\bar{\mu}_{0,j+1}}\right)
-\chi\left(\frac{2|x|}{\bar{\mu}_{0,j+1}}\right)}{|x-y|^{n+2s}}\,dy\cdot t^{\frac{n-2s}{2}\alpha_{j-1}}\langle y_j\rangle^{-2s}.
\end{split}
\end{equation*}
If $|x|\le2\bar{\mu}_{0,j},$
\begin{equation*}
\begin{aligned}   			
\left|\int_{\mathbb{R}^n}\frac{\chi\left(\frac{2|x|}{\bar{\mu}_{0,j}}\right)
-\chi\left(\frac{2|y|}{\bar{\mu}_{0,j}}\right)}{|x-y|^{n+2s}}\,dy\cdot t^{\frac{n-2s}{2}\alpha_{j-1}}\langle y_j\rangle^{-2s}\right|\lesssim t^{\sigma}\left(\frac{\mu_{0,j}}{\bar{\mu}_{0,j}}\right)^{2s}
\left(\frac{\mu_{j-1}}{\mu_j}\right)^{\frac{\alpha}{2}}\omega_{1,j}\lesssim t_0^{-l}\omega_{1,j}.
\end{aligned}
\end{equation*}
While if $|x|\ge2\bar{\mu}_{0,j},$
\begin{equation*}
\begin{split}		
\left|\int_{\mathbb{R}^n}\frac{\chi\left(\frac{2|x|}{\bar{\mu}_{0,j}}\right)
-\chi\left(\frac{2|y|}{\bar{\mu}_{0,j}}\right)}{|x-y|^{n+2s}}\,dy\cdot t^{\frac{n-2s}{2}\alpha_{j-1}}\langle y_j\rangle^{-2s}\right|\lesssim\bar{\mu}_{0,j}^{-2s}\cdot\frac{\bar{\mu}_{0,j}^{n+2s}}{|x|^{n+2s}}\cdot t^{\frac{n-2s}{2}\alpha_{j-1}}\langle y_j\rangle^{-2s}\lesssim t_0^{-l}\omega_{1,j}''.
\end{split}
\end{equation*}
Similarly,
\begin{equation*}	\left|\int_{\mathbb{R}^n}\frac{\chi\left(\frac{2|y|}{\bar{\mu}_{0,j+1}}\right)
-\chi\left(\frac{2|x|}{\bar{\mu}_{0,j+1}}\right)}{|x-y|^{n+2s}}\,dy\cdot t^{\frac{n-2s}{2}\alpha_{j-1}}\langle y_j\rangle^{-2s}\right|\lesssim t_0^{-l}\left(\omega_{1,j+1}+\omega_{1,j+1}''\right).
\end{equation*}
Therefore
\begin{equation}
\label{a.4.1}
|(-\Delta)_x^s\chi_j\varphi_{0,j}|\lesssim t_0^{-l}(\omega_{1,j}+\omega_{1,j+1}+\omega_{1,j}''+\omega_{1,j+1}'').
\end{equation}
For $-\sum_{j=2}^k\left[-(-\Delta)_x^{\frac{s}{2}}\chi_j,-(-\Delta)_x^{\frac{s}{2}}\varphi_{0,j}\right],$ we get
\begin{equation*}
\begin{split}	[-(-\Delta)_x^{\frac{s}{2}}\chi_j,-(-\Delta)_x^{\frac{s}{2}}\varphi_{0,j}]
&=\int_{\mathbb{R}^n}\frac{[\chi_j(y)-\chi_j(x)]\cdot[\varphi_{0,j}(x)-\varphi_{0,j}(y)]}{|x-y|^{n+2s}}\,dy\\
&=\int_{\mathbb{R}^n}\frac{\left[-\chi\left(\frac{2|x|}{\bar{\mu}_{0,j}}\right)
+\chi\left(\frac{2|y|}{\bar{\mu}_{0,j}}\right)\right]\cdot[\varphi_{0,j}(x)
-\varphi_{0,j}(y)]}{|x-y|^{n+2s}}\,dy\\&\quad+\int_{\mathbb{R}^n}
\frac{\left[-\chi\left(\frac{2|y|}{\bar{\mu}_{0,j+1}}\right)
+\chi\left(\frac{2|x|}{\bar{\mu}_{0,j+1}}\right)\right]
\cdot[\varphi_{0,j}(x)-\varphi_{0,j}(y)]}{|x-y|^{n+2s}}\,dy.
\end{split}
\end{equation*}
If $|x|\le2\bar{\mu}_{0,j},$
\begin{equation*}
\begin{split}	&\left|\int_{\mathbb{R}^n}\frac{\left[-\chi\left(\frac{2|x|}{\bar{\mu}_{0,j}}\right)
+\chi\left(\frac{2|y|}{\bar{\mu}_{0,j}}\right)\right]
\cdot[\varphi_{0,j}(x)-\varphi_{0,j}(y)]}{|x-y|^{n+2s}}\,dy\right|\\
&\lesssim\bar{\mu}_{0,j}^{-s}\frac{\lambda_{0,j}^{\frac{n-2s}{2}}}{\mu_j^{\frac{n-2s}{2}}}\mu_j^{-s}
\left[\int_{\mathbb{R}^n}\left(\frac{-\chi\left(\frac{2|x|}{\bar{\mu}_{0,j}}\right)
+\chi\left(\frac{2|y|}{\bar{\mu}_{0,j}}\right)}{|\frac{2x}{\bar{\mu}_{0,j}}
-\frac{2y}{\bar{\mu}_{0,j}}|^{\frac{n}{2}+s}}\right)^2\,d\frac{2y}{\bar{\mu}_{0,j}}\right]^{\frac{1}{2}}
\left[\int_{\mathbb{R}^n}\left(\frac{\bar{\phi}(\frac{x}{\mu_j})-\bar{\phi}(\frac{y}{\mu_j})}{|\frac{x}{\mu_j}
-\frac{y}{\mu_j}|^{\frac{n}{2}+s}}\right)^2\,d\frac{y}{\mu_j}\right]^{\frac{1}{2}}\\&\lesssim t_0^{-l}\omega_{1,j}.
\end{split}
\end{equation*}
While if $|x|\ge2\bar{\mu}_{0,j},$
\begin{equation*}
\begin{split}		\left|\int_{\mathbb{R}^n}\frac{\left[-\chi\left(\frac{2|x|}{\bar{\mu}_{0,j}}\right)
+\chi\left(\frac{2|y|}{\bar{\mu}_{0,j}}\right)\right]\cdot[\varphi_{0,j}(x)
-\varphi_{0,j}(y)]}{|x-y|^{n+2s}}\,dy\right|
\lesssim\frac{\lambda_{0,j}^{\frac{n-2s}{2}}}{\mu_j^{\frac{n-2s}{2}}}
\cdot\frac{\bar{\mu}_{0,j}^{n+2s}}{|x|^{n+2s}}\cdot\bar{\mu}_{0,j}^{-2s}
\frac{\mu_{0,j}^{2s}}{\bar{\mu}_{0,j}^{2s}}\lesssim t_0^{-l}\omega_{1,j}''.
\end{split}
\end{equation*}
Following a similar argument, we have
\begin{equation*} \left|\int_{\mathbb{R}^n}\frac{\left[-\chi\left(\frac{2|y|}{\bar{\mu}_{0,j+1}}\right)
+\chi\left(\frac{2|x|}{\bar{\mu}_{0,j+1}}\right)\right]\cdot[\varphi_{0,j}(x)
-\varphi_{0,j}(y)]}{|x-y|^{n+2s}}\,dy\right|\lesssim t_0^{-l}(\omega_{1,j+1}+\omega_{1,j+1}'').
\end{equation*}
Then
\begin{equation*}		\left|\left[-(-\Delta)_x^{\frac{s}{2}}\chi_j,-(-\Delta)_x^{\frac{s}{2}}\varphi_{0,j}\right]\right|\lesssim t_0^{-l}(\omega_{1,j}+\omega_{1,j+1}+\omega_{1,j}''+\omega_{1,j+1}'').
\end{equation*}
\noindent (5)~\ \textbf{Estimate of $\bar{E}_{11}.$} By $\eqref{eq2.9},$ we see that
\begin{equation*}
\begin{split}		
\bar{E}_{11}=-\left[f(\bar{U})-\sum_{j=1}^kf(U_j)-\sum_{j=2}^kf'(U_j)U_{j-1}(0)\chi_j
-\sum_{j=2}^k(1-\chi_j)\partial_t U_j\right]=J_1+J_2+J_3+J_4,
\end{split}
\end{equation*}
where
\begin{equation*}		J_1:=-\sum_{j=2}^kf'(U_j)\left(\sum_{l\not=j,j-1}U_l\right)\chi_j
-\sum_{j=2}^kf'(U_j)(U_{j-1}-U_{j-1}(0))\chi_j,
\end{equation*}
\begin{equation*}
J_2:=-\sum_{j=2}^k\left[f(\bar{U})-\sum_{i=1}^kf(U_i)-f'(U_i)\sum_{l\not=j}U_l\right]\chi_j,
\end{equation*}
\begin{equation*}
J_3:=\sum_{j=2}^k(1-\chi_j)\partial_t U_j,\quad	J_4:=-\left[f(\bar{U})-\sum_{i=1}^k f(U_i)\right]\cdot\left(1-\sum_{j=2}^k\chi_j\right).
\end{equation*}
\noindent (5.1)~\ \textbf{Estimate of $J_1.$} We start with the term
\begin{equation*}
-\sum_{j=2}^kf'(U_j)\left(\sum_{i\not=j,j-1}U_i\right)\chi_j.
\end{equation*}
Let us fix $j.$ If $i\le j-2,$ we have
\begin{equation*}
|f'(U_j)U_i\chi_j|\lesssim\mu_j^{-2s}\langle y_j\rangle^{-4s}\mu_{j-2}^{-\frac{n-2s}{2}}\chi_j\lesssim t_0^{-l}\omega_{1,j}.
\end{equation*}
If $i>j,$ using Lemma $\ref{lea.1},$ we have
\begin{equation*}
\begin{split}
|f'(U_j)U_i\chi_j|&\lesssim |U_j|^p\cdot\frac{|U_{j+1}|}{|U_j|}\chi_j\lesssim\mu_j^{-\frac{n+2s}{2}}\langle y_j\rangle^{-n-2s}\left(\lambda_{j+1}^{-\frac{n-2s}{2}}\langle y_{j+1}\rangle^{2s-n}1_{\{\bar{\mu}_{0,j+1}\le|x|\le {\mu}_{0,j}\}}+\lambda_{j+1}^{\frac{n-2s}{2}}\right)\chi_j\\
&\lesssim\mu_{j+1}^{\frac{n-2s}{2}}\mu_j^{-2s}|x|^{2s-n}1_{\{\bar{\mu}_{0,j+1}\le|x|\le\mu_{0,j}\}}
+\left(\frac{\lambda_{j+1}}{\lambda_j}\right)^{\frac{n-2s}{2}}t^\sigma\omega_{1,j}\lesssim t_0^{-l}(\omega_{2,j}+\omega_{1,j}).
\end{split}
\end{equation*}
The another term in $J_1$ is $-\sum_{j=2}^kf'(U_j)(U_{j-1}-U_{j-1}(0))\chi_j,$ which can be bounded as follow
\begin{equation*}
|f'(U_j)(U_{j-1}-U_{j-1}(0))\chi_j|\lesssim\mu_j^{-2s}\langle y_j\rangle^{-4s}\mu_{j-1}^{-\frac{n-2s}{2}}\lambda_j\chi_j\lesssim t_0^{-l}\omega_{1,j}.
\end{equation*}
			
\noindent (5.2)~\ \textbf{Estimate of $J_2.$} Using the mean value theorem, the definition of $\chi_j$ and Lemma $\ref{lea.1},$ we deduce that
\begin{equation*}			\left|f(\bar{U})-\sum_{i=1}^kf(U_i)-f'(U_j)\sum_{l\not=j}U_l\right|\chi_j\lesssim(|U_{j-1}|^p+|U_{j+1}|^p)\chi_j.
\end{equation*}
Then
\begin{equation*}
\begin{split}
|U_{j-1}|^p\chi_j\lesssim\mu_{j-1}^{-\frac{n+2s}{2}}\chi_j\lesssim t^{\frac{n+2s}{2}\alpha_{j-1}}\mu_j^{2s}t^{-\frac{n-2s}{2}\alpha_{j-1}}t^\sigma
\left(\frac{\bar{\mu}_j}{\mu_j}\right)^{2s+\alpha}\omega_{1,j}\chi_j
\lesssim\left(\frac{\mu_j}{\mu_{j-1}}\right)^{\frac{2s-\alpha}{2}}t^\sigma\omega_{1,j}\chi_j
\lesssim t_0^{-l}\omega_{1,j},
\end{split}
\end{equation*}
and
\begin{equation*}
|U_{j+1}|^p\chi_j\lesssim\mu_{j+1}^{\frac{n+2s}{2}}|x|^{-2s-n}\chi_j\lesssim t_0^{-l}\omega_{2,j}.
\end{equation*}
			
\noindent (5.3)~\ \textbf{Estimate of $J_3.$} For $j=2,\cdots,k,$ we find that
\begin{equation*}
|\partial_t U_j|=|\dot{\mu}_j\mu_j^{2s-1}\mu_j^{-\frac{n+2s}{2}}Z_{n+1}(y_j)|
\lesssim\mu_j^{-2s}\mu_{j-1}^{-\frac{n-2s}{2}}\langle y_j\rangle^{2s-n}.
\end{equation*}
In addition, the support of $1-\chi_j$ is $\{|x|\le\bar{\mu}_{0,j+1}\}\cup\{\frac{1}{2}\bar{\mu}_{0,j}\le|x|<\bar{\mu}_{0,j}\}
\cup\{\bar{\mu}_{0,j}\le|x|\}.$
In the first set, we see that $1-\chi_j\le1_{\{|x|\le\bar{\mu}_{0,j+1}\}}.$ Then
\begin{equation*}
\begin{aligned}
|(1-\chi_j)\partial_t U_j|	\lesssim\left(\frac{\mu_{j+1}}{\mu_j}\right)^{\frac{2s-\alpha}{2}}
\left(\frac{\mu_j}{\mu_{j-1}}\right)^{\frac{n-2s}{2}}t^\sigma\omega_{1,j+1}
\chi\left(\frac{2|x|}{\bar{\mu}_{0,j+1}}\right)\lesssim t_0^{-l}\omega_{1,j+1}.
\end{aligned}
\end{equation*}
In the second set, we get
\begin{equation*}
\begin{aligned}
|\partial_t U_j|1_{\{\frac{1}{2}\bar{\mu}_{0,j}\le|x|<\bar{\mu}_{0,j}\}}\lesssim
\left(\frac{\mu_j}{\mu_{j-1}}\right)^{\frac{n-4s-\alpha}{2}}t^\sigma\omega_{1,j} 1_{\{\frac{1}{2}\bar{\mu}_{0,j}\le|x|<\bar{\mu}_{0,j}\}}\lesssim t_0^{-l}\omega_{1,j}.
\end{aligned}
\end{equation*}
In the third set, we can split it further to be
\begin{equation*}		\{\bar{\mu}_{0,j}\le|x|\}=\cup_{i=2}^j\{\bar{\mu}_{0,i}\le|x|\le\bar{\mu}_{0,i-1}\}\cup\{\bar{\mu}_{0,1}\le|x|\}.
\end{equation*}
In $\{\bar{\mu}_{0,i}\le|x|\le\bar{\mu}_{0,i-1}\},i=2,\cdots,j,$
\begin{equation*}
|\partial_t U_j|\lesssim\mu_i^{n-4s}\mu_{i-1}^{-\frac{n-2s}{2}}|x|^{2s-n}\lesssim t_0^{-l}\omega_{2,i-1}.
\end{equation*}
While in $\{\bar{\mu}_{0,1}\le |x|\},$ we get $|\partial_t U_j|\lesssim\mu_2^{n-4s}|x|^{2s-n}\lesssim t_0^{-l}\omega_3.$
\medskip
			
\noindent (5.4)~\ \textbf{Estimate of $J_4.$} Observing that  $$\mathrm{Supp}\{J_4\}\subset\cup_{i=3}^k\left\{\frac{1}{2}\bar{\mu}_{0,i}\le|x|\le\bar{\mu}_{0,i}\right\}
\cup\left\{\frac{1}{2}\bar{\mu}_{0,2}\le|x|\right\}.$$
By Lemma \ref{lea.1}, we have
$$|U_i|\approx |U_{i-1}|\approx\mu_{i-1}^{-\frac{n-2s}{2}}\gg|U_m|~\mbox{in}~
\left\{\frac{1}{2}\bar{\mu}_{0,i}\le|x|\le\bar{\mu}_{0,i}\right\}~\mbox{for}~m\not=i,i-1,\quad i=3,\cdots,k.$$
Hence
\begin{equation*}
\begin{aligned}			
|J_4|1_{\{\frac{1}{2}\bar{\mu}_{0,i}\le|x|\le\bar{\mu}_{0,i}\}}
\lesssim\mu_{i-1}^{-\frac{n+2s}{2}}1_{\{\frac{1}{2}\bar{\mu}_{0,i}\le|x|\le\bar{\mu}_{0,i}\}}
\lesssim\left(\frac{\mu_{i-1}}{\mu_i}\right)^{\frac{\alpha-2s}{2}}t^\sigma\omega_{1,i}\lesssim t_0^{-l}\omega_{1,i}.
\end{aligned}
\end{equation*}
In particular, in $\{\frac{1}{2}\bar{\mu}_{0,2}\le|x|\},$ by Lemma $\ref{lea.1},$ we deduce that
\begin{equation*}
\begin{split}
|J_4|\lesssim |U_1|^{p-1}|U_2|&\approx\mu_2^{\frac{n-2s}{2}}\left(|x|^{2s-n}
1_{\{\frac{1}{2}\bar{\mu}_{0,2}\le|x|\le\bar{\mu}_{0,2}\}}+|x|^{2s-n}1_{\{\bar{\mu}_{0,2}\le|x|\le1\}}\right)
\\&\quad+\mu_2^{\frac{n-2s}{2}}\left(|x|^{-2s-n}1_{\{1\le|x|\le\bar{\mu}_{0,1}\}}
+|x|^{-2s-n}1_{\{\bar{\mu}_{0,1}\le|x|\}}\right)\\&\lesssim t_0^{-l}(\omega_{1,2}+\omega_{2,1}+\omega_{1,1}+\omega_3).
\end{split}
\end{equation*}

Combining the computations from (1) to (5), we complete the proof  of Lemma $\ref{lea.4}.$
\end{proof}

Based on Lemma $\ref{le5.1},$ we have the following estimates.
\begin{lemma}\label{lea.5}
There exist $\sigma>0,\delta>0$ small enough and $t_0$ large enough, such that for $t>t_0,$
\begin{itemize}
\item[(1)] $(\omega_{1,j}')^*+(\omega_{1,j})^*\lesssim(\omega_{1,j}'')^*,~j=2,\cdots,k$. While if $|x|\le t^{\frac{1}{2s}},$ $(\omega_{1,1}')^*\lesssim(\omega_{1,1})^*.$
		
\item [(2)] In $\{|x|\le\bar{\mu}_{0,i}\},i=3,\cdots,k,$ we have $(\omega_{1,j}'')^*\lesssim(\omega_{1,i}'')^*$ for $j=2,\cdots,i-1.$
		
\item [(3)] In $\{\bar{\mu}_{0,i}\le|x|\le t^{\frac{1}{2s}}\},i=2,\cdots,k,$ we have $(\omega_{1,j}'')^*\lesssim(\omega_{1,i}'')^*$ for $j=i+1,\cdots,k.$
		
\item [(4)]
%		\begin{equation*}
%		\begin{cases}
%		\omega_3^*\lesssim(\omega_{1,j}'')^*~\mbox{in}~\{|x|\leq\bar\mu_{0,j}\}~&\mbox{for}~j=2,\cdots,k,\\
%		\omega_3^*\gtrsim\omega_{1,j}^*+(\omega_{1,j}')^*+(\omega_{1,i}'')^*~\mbox{in}~\{|x|\ge t^{\frac{1}{2s}}\quad &\mbox{for}~j=2,\cdots,k,~i=2,\cdots,k.
%		\end{cases}
%		\end{equation*}
In $\{|x|\le\bar{\mu}_{0,j}\},$ $\omega_3^*\lesssim(\omega_{1,j}'')^*$ for $j=2,\cdots,k.$   In  $\{|x|\ge t^{\frac{1}{2s}}\},$ $\omega_3^*\gtrsim\omega_{1,j}^*+(\omega_{1,j}')^*+(\omega_{1,i}'')^*$ for $j=1,\cdots,k,i=2,\cdots k.$
\end{itemize}
Consequently,
\begin{equation}\label{eqa.20}
\sum_{j=1}^k(\omega_{1,j}^*+(\omega_{1,j}')^*)+\sum_{j=2}^k(\omega_{1,j}'')^*+\omega_3^*\lesssim
\begin{cases}
(\omega_{1,k}'')^*&\text{if $|x|\le\bar{\mu}_{0,k},$}\\
(\omega_{1,i}'')^*+(\omega_{1,i+1}'')^*&\text{if $\bar{\mu}_{0,i+1}\le|x|\le\bar{\mu}_{0,i},~i=2,\cdots,k-1,$}\\
\omega_{1,1}^*+\omega_3^*+(\omega_{1,2}'')^*&\text{if $\bar{\mu}_{0,2}\le|x|\le t^{\frac{1}{2s}},$}\\
\omega_3^*&\text{if $|x|\ge t^{\frac{1}{2s}}.$}
\end{cases}
\end{equation}
\end{lemma}
\begin{proof}
From the results in Lemma $\ref{le5.1},$ by direct computation we get $(1)$. Concerning the point (2), we notice that
$$(\omega_{1,j}'')^*=t^{\gamma_j}\le t^{\gamma_{i}}=(\omega_{1,i}'')^*~\mathrm{in}~\{|x|\le\bar{\mu}_{0,i}\},\quad j=2,\cdots,i-1.$$
For the point (3), we have
$$(\omega_{1,j}'')^*(x,t)=\bar{\mu}_{0,j}^{n-2s}t^{\gamma_j}|x|^{2s-n}
\lesssim\bar{\mu}_{0,i}^{n-2s}t^{\gamma_i}|x|^{2s-n}=(\omega_{1,i}'')^*~\mathrm{in}~\{\bar{\mu}_{0,i}\le|x|\le t^{\frac{1}{2s}}\},\quad j=i+1,\cdots,k.$$
For the last point, using the fact $t^{\delta(n-4s)}t^{-1-\sigma}\bar{\mu}_{0,1}^{4s-n}\lesssim t^{-\sigma}=(\omega_{1,2}'')^*$ for $x\in\{|x|\le \bar{\mu}_{0,2}\}$. Combined with the above arguments we get
$$\omega_3^*\lesssim (w_{1,j}'')^*~\mbox{in}~\{|x|\leq\bar\mu_{0,j}\}~\mbox{for}~j=2,\cdots,k.$$
In the end, if $|x|\ge t^{\frac{1}{2s}},$ by Lemma $\ref{le5.1}$ we see that
\begin{equation*}
\begin{split}
t^{\delta(n-4s)}\cdot t^{-\sigma}\cdot|x|^{2s-n}\gtrsim t^{\gamma_1+\delta(n-2s-\alpha)}|x|^{2s-n}+t^{\gamma_1+\delta(n-3s-\alpha)}|x|^{2s-n}
+t^{-\frac{(\alpha_j+\alpha_{j-1})(n-2s)}{2}+\gamma_j}|x|^{2s-n}.
\end{split}
\end{equation*}
It implies that $\omega_3^*\gtrsim\omega_{1,j}^*+(\omega_{1,j}')^*+(\omega_{1,i}'')^*$ for $j=1,\cdots,k,i=2,\cdots k.$
\end{proof}

Following a similar argument as above we have the following Lemma.
\begin{lemma}\label{lea.6}
There exists $t_0$ large enough such that
\begin{enumerate}
\item [(1)] In $\{\bar{\mu}_{0,i+1}\le|x|\},$ $i=1,\cdots,k-1,$ we have $\omega_{2,j}^*\lesssim\omega_{2,i}^*$ for $j=i+1,\cdots,k.$
		
\item [(2)] In $\{|x|\le\bar{\mu}_{0,i}\},$ $i=2,\cdots,k,$ we have $\omega_{2,i-1}^*\gtrsim\omega_{2,j}^*$ for $j=1.\cdots,i-1.$
		
\item [(3)] In $\{|x|\ge\bar{\mu}_{0,1}\},$ $\omega_{2,j}^*\lesssim\omega_3^*$ for $j=1,\cdots,k.$
\end{enumerate}
Consequently
\begin{equation}\label{eqa.21}
\sum_{j=1}^{k-1}\omega_{2,j}^*\lesssim
\begin{cases}
\omega_{2,k-1}^*&\text{if $|x|\le\bar{\mu}_{0,k},$}\\
\omega_{2,i}^*+\omega_{2,i-1}^*&\text{if $\bar{\mu}_{0,i+1}\le|x|\le\bar{\mu}_{0,i},i=2,\cdots,k-1,$}\\
\omega_{2,1}^*&\text{if $\bar{\mu}_{0,2}\le|x|\le\bar{\mu}_{0,1},$}\\
\omega_3^*&\text{if $|x|\ge\bar{\mu}_{0,1}.$}
\end{cases}
\end{equation}
\end{lemma}

\begin{lemma}\label{lea.7}
There exists  $t_0>0$ large enough such that
\begin{equation*}
\|\mathcal{T}^{out}[V\Psi]\|_{\alpha,\sigma}^{out,*}\lesssim t_0^{-l}\|\Psi\|_{\alpha,\sigma}^{out,*}.
\end{equation*}
\end{lemma}
\begin{proof}
Without loss of generality, we suppose that $\|\Psi\|_{\alpha,\sigma}^{out,*}\le 1.$ Using $\eqref{eq3.10},$ we have
\begin{equation*}
V=pu_*^{p-1}(1-\sum_{j=1}^k\zeta_j)+\sum_{j=1}^k\zeta_jp(u_*^{p-1}-U_j^{p-1}).
\end{equation*}
For the term $pu_*^{p-1}(1-\sum_{j=1}^k\zeta_j),$ by the definition of $\zeta_j,$ we find that
$$\mathrm{Supp}\{1-\sum_{j=1}^k\zeta_j\}=\cup_{i=2}^k\{R\mu_{0,i}\le|x|\le2R^{-1}\mu_{0,i-1}\}
\cup\{R\mu_{0,1}\le|x|\}.$$
We shall give the estimation in each interval as follows
\medskip

\noindent (1)~\ In the region $\{R\mu_{0,1}\le|x|\},$ using Lemma $\ref{lea.1}$ and $\ref{lea.2},$ we obtain that $$p|u_*|^{p-1}\lesssim\mu_1^{2s}|x|^{-4s}\lesssim|x|^{-4s}\lesssim R^{-s}|x|^{-3s}.$$ We write
\begin{equation*}
\{R\mu_{0,1}\le|x|\}=\left\{R\mu_{0,1}\le|x|\le\bar{\mu}_{0,1}\right\}\cup\left\{\bar{\mu}_{0,1}\le|x|\le t^{\frac{1}{2s}}\right\}\cup\left\{|x|\ge t^{\frac{1}{2s}}\right\}.
\end{equation*}
In the first set, using $\eqref{eqa.20}$ and $\eqref{eqa.21},$ we see that
\begin{equation*}
|\Psi|\lesssim \omega_{1,1}^*+\omega_3^*+(\omega_{1,2}'')^*+\omega_{2,1}^*.
\end{equation*}
Observing that $(\omega_{1,2}'')^*+\omega_{2,1}^*\lesssim\omega_{1,1}^*$ in $\{R\mu_{0,1}\le|x|\le\bar{\mu}_{0,1}\}.$ Hence, we have
\begin{equation*}
\left|pu_*^{p-1}(1-\sum_{j=1}^k\zeta_j)\Psi\right|\lesssim R^{-s}|x|^{-3s}(\omega_{1,1}^*+\omega_3^*)\lesssim t_0^{-l}\omega_{1,1}.
\end{equation*}
In the second set, using $\eqref{eqa.20}$ and $\eqref{eqa.21}$ again, we see that $|\Psi|\lesssim\omega_{1,1}^*+\omega_3^*+(\omega_{1,2}'')^*\lesssim\omega_{1,1}^*+\omega_3^*.$ Then
\begin{equation*}
\left|pu_*^{p-1}(1-\sum_{j=1}^k\zeta_j)\Psi\right|\lesssim R^{-s}|x|^{-3s}(\omega_{1,1}^*+\omega_3^*)\lesssim t_0^{-l}(\omega_{1,1}'+\omega_3).
\end{equation*}
In the third set, we find that
\begin{equation*}
\left|pu_*^{p-1}(1-\sum_{j=1}^k\zeta_j)\Psi\right|\lesssim R^{-s}|x|^{-3s}\omega_3^*\lesssim t_0^{-l}\omega_3.
\end{equation*}

\noindent (2)~\  In the region $\{R\mu_{0,i}\le|x|\le 2R^{-1}\mu_{0,i-1}\},i=2,\cdots,k.$ We split the region $\{R\mu_{0,i}\le|x|\le 2R^{-1}\mu_{0,i-1}\}$ as
$$\{R\mu_{0,i}\le|x|\le 2R^{-1}\mu_{0,i-1}\}=\{R\mu_{0,i}\le|x|\le\bar{\mu}_{0,i}\}\cup\{\bar{\mu}_{0,i}\le|x|\le 2R^{-1}\mu_{0,i-1}\}.$$
In the first set, we get $p|u_*|^{p-1}\lesssim\mu_i^{2s}|x|^{-4s}.$ Using $\eqref{eqa.20}$ and $\eqref{eqa.21},$ we deduce that
\begin{equation*}
|\Psi|\lesssim  \begin{cases}
(\omega_{1,i}'')^*+(\omega_{1,i+1}'')^*+\omega_{2,i-1}^*+\omega_{2,i}^*, \quad &\mathrm{if}~i=2,\cdots,k-1,\\
(\omega_{1,k}'')^*+\omega_{2,k-1}^*,\quad &\mathrm{if}~i=k.
\end{cases}
\end{equation*}
On the other hand, by Lemma $\ref{le5.1},$ we have $$(\omega_{1,i+1}'')^*\lesssim(\omega_{1,i}'')^*~\mathrm{in}~\{R\mu_{0,i}\le|x|\le\bar{\mu}_{0,i}\},\quad i=2,\cdots,k-1.$$
Thus
\begin{equation*}
\begin{split}	\left|pu_*^{p-1}(1-\sum_{j=1}^k\zeta_j)\Psi\right|\lesssim\mu_i^{2s}|x|^{-4s}
((\omega_{1,i}'')^*+\omega_{2,i-1}^*+\omega_{2,i}^*)\lesssim t_0^{-l}(\omega_{1,i}+\omega_{2,i}),
\end{split}
\end{equation*}
where we have used the fact that in $\{R\mu_{0,i}\le|x|\le\bar{\mu}_{0,i}\},$
\begin{equation*}
\begin{split}
&\mu_i^{2s}|x|^{-4s}(\omega_{1,i}'')^*\lesssim R^{\alpha-2s} t^{\gamma_i}\mu_{0,i}^{\alpha}|x|^{-2s-\alpha}\lesssim t_0^{-l}\omega_{1,i},\quad i=2,\cdots,k,\\&
\mu_i^{2s}|x|^{-4s}\omega_{2,i-1}^*\lesssim R^{\alpha-2s}\mu_{0,i}^\alpha t^{\gamma_i}|x|^{-2s-\alpha}\lesssim t_0^{-l}\omega_{1,i},\quad i=2,\cdots,k,\\&
\mu_i^{2s}|x|^{-4s}\omega_{2,i}^*\lesssim R^{-2s}t^{-\sigma}\mu_{0,i+1}^{\frac{n}{2}-2s}\mu_{0,i}^{-s}|x|^{2s-n}\lesssim t_0^{-l}\omega_{2,i},\quad i=2,\cdots,k-1.
\end{split}
\end{equation*}
In the second set, applying Lemma $\ref{lea.1},$ we have $pu_*^{p-1}\lesssim|U_{i-1}|^{p-1}\lesssim\mu_{0,i-1}^{-2s}.$ On the other hand, by Lemma $\ref{le5.1}$ we also get that $\omega_{2,i-2}^*\lesssim(\omega_{1,i-1}'')^*$.   Then
\begin{equation*}
\begin{split}		\left|pu_*^{p-1}(1-\sum_{j=1}^k\zeta_j)\Psi\right|
\lesssim\mu_{0,i-1}^{-2s}((\omega_{1,i-1}'')^*+(\omega_{1,i}'')^*+\omega_{2,i-1}^*)
\lesssim\mu_{0,i-1}^{-2s}t^{\gamma_{i-1}}1_{\{|x|\le 2R^{-1}\mu_{0,i-1}\}}+t_0^{-l}\omega_{2,i-1},
\end{split}
\end{equation*}
where we have used the fact that in $\{\bar{\mu}_{0,i}\le|x|\le 2R^{-1}\mu_{0,i-1}\},$
\begin{equation*}
\begin{split}
&\mu_{0,i-1}^{-2s}(\omega_{1,i}'')^*\lesssim t_0^{-l}t^{-\sigma}\mu_{0,i}^{\frac{n}{2}-2s}\mu_{0,i-1}^{-s}|x|^{2s-n}\lesssim t_0^{-l}\omega_{2,i-1},\quad i=2,\cdots,k,\\&
\mu_{0,i-1}^{-2s}\omega_{2,i-1}^*\lesssim R^{-2s}t^{-\sigma}\mu_{0,i}^{\frac{n}{2}-2s}\mu_{0,i-1}^{-s}|x|^{2s-n}\lesssim t_0^{-l}\omega_{2,i-1},\quad i=2,\cdots,k.
\end{split}
\end{equation*}
Moreover, by Section $B,$ we deduce that
\begin{equation*}
\begin{split}
\mathcal{T}^{out}[\mu_{0,i-1}^{-2s}t^{\gamma_{i-1}}1_{\{|x|\le 2R^{-1}\mu_{0,i-1}\}}]\lesssim&
\begin{cases}
R^{-2s}t^{\gamma_{i-1}}\quad &\text{if $|x|\le 4R^{-1}\mu_{0,i-1}$}\\
R^{-n}\mu_{0,i-1}^{-2s}t^{\gamma_{i-1}}\mu_{0,i-1}^n|x|^{2s-n}\quad &\text{if $|x|\ge4R^{-1}\mu_{0,i-1}$}
\end{cases}\\
\lesssim&~ t_0^{-l}(\omega_{1,i-1}'')^*.
\end{split}
\end{equation*}
\medskip
	
For the term $\sum_{i=1}^k\zeta_ip(u_*^{p-1}-U_i^{p-1}).$ From the definition of $\zeta_i,$ we know that the support of $\zeta_i$ is contained in $\{R^{-1}\mu_{0,i}\le|x|\le 2R\mu_{0,i}\}.$
\smallskip

\noindent (1)~\ For $i=1,$ it follows from the definition of $\chi_j$ that $\varphi_0=0$ in the support of $\zeta_1.$ Then by Lemma $\ref{lea.1},$ we know that
\begin{equation*}	
|\zeta_1(u_*^{p-1}-U_1^{p-1})|\lesssim|U_2|^{p-1}\zeta_1
\lesssim\mu_2^{2s}|x|^{-4s}1_{\{R^{-1}\mu_{0,1}\le|x|\le2R\mu_{0,1}\}}.
\end{equation*}
Therefore we deduce that
\begin{equation*}
\begin{split}		|\zeta_1(u_*^{p-1}-U_1^{p-1})\Psi|\lesssim\mu_2^{2s}|x|^{-4s}
1_{\{R^{-1}\mu_{0,1}\le|x|\le2R\mu_{0,1}\}}(\omega_{1,1}^*+\omega_3^*+\omega_{2,1}^*+(\omega_{1,2}'')^*)
\lesssim t_0^{-l}\omega_{1,1},
\end{split}
\end{equation*}
where we have used the fact that in $\{R^{-1}\mu_{0,1}\le|x|\le2R\mu_{0,1}\},$
\begin{equation*}
\begin{split}
&\mu_2^{2s}|x|^{-4s}\omega_{1,1}^*\lesssim\mu_2^{2s}|x|^{-4s}t^{\gamma_1}\lesssim t_0^{-l}\omega_{1,1},\\&
\mu_2^{2s}|x|^{-4s}\omega_3^*\lesssim\mu_2^{2s}|x|^{-4s}t^{-1-\sigma}\lesssim t_0^{-l}\omega_{1,1}\\		&\mu_2^{2s}|x|^{-4s}\omega_{2,1}^*\lesssim\mu_2^{2s}|x|^{-4s}t^{-\sigma}\bar{\mu}_{0,2}^{n-4s}|x|^{2s-n}
\lesssim t_0^{-l}\omega_{1,1},\\&\mu_2^{2s}|x|^{-4s}(\omega_{1,2}'')^*
\lesssim\mu_2^{2s}|x|^{-4s}\bar{\mu}_{0,2}^{n-2s}t^{\gamma_2}|x|^{2s-n}\lesssim t_0^{-l}\omega_{1,1}.
\end{split}
\end{equation*}
		
\noindent (2)~\ For $i=2,\cdots,k,$ we have $$|\zeta_i(u_*^{p-1}-U_i^{p-1})|\lesssim(|U_{i-1}|^{p-1}+|U_{i+1}|^{p-1}+\varphi_{0,i}^{p-1})|\zeta_i|
\lesssim\mu_{0,i-1}^{-2s}|\zeta_i|.$$ Then
\begin{equation*}
\begin{split}
|\zeta_i(u_*^{p-1}-U_i^{p-1})\Psi|&\lesssim
\begin{cases}	\mu_{0,i-1}^{-2s}((\omega_{1,i}'')^*+(\omega_{1,i+1}'')^*+\omega_{2,i-1}^*+\omega_{2,i}^*)|\zeta_i|,\quad &i=2,\cdots,k-1,\\
\mu_{0,k-1}^{-2s}((\omega_{1,k}'')^*+\omega_{2,k-1}^*), &i=k,
\end{cases}
\\&\lesssim\begin{cases}
t_0^{-l}(\omega_{1,i}+\omega_{2,i}),\quad &i=2,\cdots,k-1,\\
t_0^{-l}(\omega_{1,k}+\omega_{2,k-1}), &i=k,
\end{cases}
\end{split}
\end{equation*}	
where we have used the fact that $\omega_{2,i-1}^*\lesssim\omega_{1,i}^*,~i=2,\cdots,k$ and
\begin{equation*}
\begin{split}
&\mu_{i-1}^{-2s}(\omega_{1,i}'')^*\lesssim\mu_{i-1}^{-2s}t^{\gamma_i}\lesssim t_0^{-l}\omega_{1,i},\quad i=2,\cdots,k,\\&		\mu_{j-1}^{-2s}(\omega_{1,j+1}'')^*\lesssim\mu_{j-1}^{-2s}\bar{\mu}_{0,j+1}^{n-2s}t^{\gamma_{j+1}}|x|^{2s-n}
\lesssim t_0^{-l}\omega_{2,j}\quad i=2,\cdots,k-1,\\& \mu_{j-1}^{-2s}\omega_{2,j}^*\lesssim\mu_{j-1}^{-2s}t^{-\sigma}\mu_{0,j+1}^{\frac{n}{2}-2s}\mu_{0,j}^{-s}
|x|^{4s-n}\lesssim t_0^{-l}\omega_{2,j},\quad i=2,\cdots,k-1.
\end{split}
\end{equation*}	
Combining the above computations we finish the proof.
\end{proof}

\begin{lemma}\label{lea.8}
There exists $t_0>0$ large enough such that
\begin{equation*}
\|N[\vec{\phi},\Psi,\vec{\mu}_1]\|_{\alpha,\sigma}^{out}\lesssim t_0^{-l}(\|\vec{\phi}\|_{a,\sigma}^{in}+\|\Psi\|_{\alpha,\sigma}^{out,*})^p.
\end{equation*}
\end{lemma}
\begin{proof}
By $\eqref{eq3.9}$ and $p\in(1,2),$ we have
\begin{equation*}
|N[\vec{\phi},\Psi,\vec{\mu}_1]|\lesssim\sum_{j=1}^k\mu_j^{-\frac{n+2s}{2}}|\phi_j|^p\eta_j+|\Psi|^p.
\end{equation*}	
Consider the term $\mu_j^{-\frac{n+2s}{2}}|\phi_j|^p\eta_j,$ using  the norm definition of $\phi_j,$ we deduce that
\begin{equation}\label{eqa.24}
\begin{split}		\mu_j^{-\frac{n+2s}{2}}|\phi_j|^p\eta_j\lesssim\mu_j^{-\frac{n+2s}{2}}
\lambda_{0,j}^{\frac{n+2s}{2}}t^{-p\sigma}R^{(a-2s)p}\langle y_j\rangle^{-pa}(\|\phi_j\|_{j,a,\sigma}^{in})^p1_{\{|x|\le 4R\mu_{0,j}\}}\lesssim t_0^{-l}\omega_{1,j}(\|\phi_j\|_{j,a,\sigma}^{in})^p.
\end{split}
\end{equation}
For $|\Psi|^p,$ we divide $\mathbb{R}^n$ into $k+2$ parts and provide the estimation on it in these regions
\smallskip
	
\noindent (1)~\ In $\{|x|\ge\bar{\mu}_{0,1}\},$ by $\eqref{eqa.20}$ and $\eqref{eqa.21},$ we deduce that
\begin{equation*}
|\Psi|\lesssim\begin{cases}
\omega_3^*\|\Psi\|_{\alpha,\sigma}^{out,*}, &	
\mbox{if}~\	t^{\frac{1}{2s}}\le |x|,\\
(\omega_{1,1}^*+\omega_3^*+(\omega_{1,2}'')^*)\|\Psi\|_{\alpha,\sigma}^{out,*},\quad &\mbox{if}~\ \bar{\mu}_{0,1}\le |x|\le t^{\frac{1}{2s}}.
\end{cases}
\end{equation*}
If $|x|>t^{\frac{1}{2s}},$
\begin{equation*}
\begin{aligned}
(\omega_3^*)^p=(t^{\delta(n-4s)})^pt^{-p\sigma}|x|^{(2s-n)p}\lesssim (t^{\delta(n-4s)})^{p-1}t^{1+\sigma-p\sigma}|x|^{(2s-n)(p-1)}\omega_3\lesssim t_0^{-l}\omega_3.
\end{aligned}
\end{equation*}
While if $\bar{\mu}_{0,1}\le|x|\le t^{\frac{1}{2s}},$
\begin{equation*}
\begin{split}
(\omega_3^*)^p=(t^{\delta(n-4s)})^pt^{-p-p\sigma}|x|^{4sp-pn}\lesssim (t^{\delta(n-4s)})^{p-1}t^{-(p-1)(1+\sigma)}|x|^{-2s\frac{n-6s}{n-2s}}\omega_3\lesssim t_0^{-l}\omega_3,
\end{split}
\end{equation*}	
\begin{equation*}
\begin{aligned}
(\omega_{1,1}^*)^p\lesssim t^{p\gamma_1+p\delta(n-2s-\alpha)}|x|^{p(2s-n)}\lesssim t^{-\delta(n-4s)}|x|^{-4s}t^{(p-1)\gamma_1+\delta p(n-2s-\alpha)}\omega_3\lesssim t_0^{-l}\omega_3,
\end{aligned}
\end{equation*}
\begin{equation*}
\begin{aligned}
((\omega_{1,2}'')^*)^p\lesssim t^{-\delta(n-4s)}\bar{\mu}_2^{np-2sp}|x|^{(2s-n)(p-1)}t^{1+\sigma}\omega_3\lesssim t_0^{-l}\omega_3.
\end{aligned}
\end{equation*}
As a consequence, $|\Psi|^p\lesssim t_0^{-l}(\|\Psi\|_{\alpha,\sigma}^{out,*})^p\omega_3.$
\smallskip
		
\noindent (2)~\	In $\{1\le|x|\le\bar{\mu}_{0,1}\},$ by $\eqref{eqa.20}$ and $\eqref{eqa.21} $ we have $$|\Psi|\lesssim(\omega_{1,1}^*+\omega_{2,1}^*+\omega_3^*+(\omega_{1,2}'')^*)
\|\Psi\|_{\alpha,\sigma}^{out,*}.$$
In addition, we have
$$\omega_{2,1}^*\lesssim\omega_{1,1}^*\quad \mbox{in}\quad \{1\le|x|\le\bar{\mu}_{0,1}\}.$$
Thus
\begin{equation*}
\begin{split}	|\Psi|^p&\lesssim((\omega_{1,1}^*)^p+((\omega_{1,2}'')^*)^p
+(\omega_3^*)^p)(\|\Psi\|_{\alpha,\sigma}^{out,*})^p\\&\lesssim t^{-(1+\sigma)(p-1)}|x|^{-(p-1)\alpha+2s}\omega_{1,1}(\|\Psi\|_{\alpha,\sigma}^{out,*})^p
+\bar{\mu}_{0,2}^{np-2sp}|x|^{(2s-n)p}t^{\gamma_2p-\gamma_1}|x|^{2s+\alpha}\omega_{1,1}
(\|\Psi\|_{\alpha,\sigma}^{out,*})^p\\&\quad+t^{-(1+\sigma)(p-1)}|x|^{2s+\alpha}\omega_{1,1}
(\|\Psi\|_{\alpha,\sigma}^{out,*})^p\\&\lesssim t_0^{-l}\omega_{1,1}(\|\Psi\|_{\alpha,\sigma}^{out,*})^p.
\end{split}
\end{equation*}
\smallskip
		
\noindent (3)~\ In $\{\bar{\mu}_{0,2}\le|x|\le 1\},$ we get from $\eqref{eqa.20}$ and $\eqref{eqa.21} $ that $$|\Psi|\lesssim(\omega_{1,1}^*+\omega_3^*+\omega_{2,1}^*+(\omega_{1,2}'')^*)
\|\Psi\|_{\alpha,\sigma}^{out,*}.$$Then
\begin{equation*}
\begin{split}
&(\omega_{1,1}^*)^p\lesssim t^{\gamma_1p}\lesssim t_0^{-l}\omega_{1,1},\quad
(\omega_3^*)^p\lesssim t^{-(1+\sigma)(p-1)}t^{\gamma_1}\lesssim t_0^{-l}\omega_{1,1}\\&
(\omega_{2,1}^*)^p\lesssim t^{-\sigma(p-1)}|x|^{(4s-n)(p-1)+2s}\bar{\mu}_{0,2}^{p(n-4s)}\mu_{0,2}^{-\frac{n}{2}+2s}\omega_{2,1}\lesssim t_0^{-l}\omega_{2,1},\\&
((\omega_{1,2}'')^*)^p\lesssim\bar{\mu}_{0,2}^{np-n-2s}\bar{\mu}_{0,2}^{2s(1-p)}
|x|^{(2s-n)p+n+2s}t^{-(p-1)\sigma}\omega_{1,2}''.
\end{split}
\end{equation*}
Hence
\begin{equation*}
\begin{aligned}				
|\Psi|^p\lesssim((\omega_{1,1}^*)^p+(\omega_3^*)^p
+(\omega_{2,1}^*)^p+((\omega_{1,2}'')^*)^p)(\|\Psi\|_{\alpha,\sigma}^{out,*})^p\lesssim t_0^{-l}(\omega_{1,1}+\omega_{2,1}+\omega_{1,2}'')(\|\Psi\|_{\alpha,\sigma}^{out,*})^p.
\end{aligned}
\end{equation*}
		
\noindent (4)~\ In $\{\bar{\mu}_{0,i+1}\le|x|\le\bar{\mu}_{0,i}\},i=2,\cdots,k-1,$ we get
\begin{equation*}	
|\Psi|\lesssim((\omega_{1,i}'')^*+(\omega_{1,i+1}'')^*+\omega_{2,i-1}^*+\omega_{2,i}^*)
\|\Psi\|_{\alpha,\sigma}^{out,*}.
\end{equation*}
Then
\begin{equation*}
((\omega_{1,i}'')^*)^p\lesssim t^{\gamma_ip}\mu_{0,i}^{2s}t^{-\gamma_i}\omega_{1,i}\lesssim t_0^{-l}\omega_{1,i},
\end{equation*}
\begin{equation*}	((\omega_{1,i+1}'')^*)^p\lesssim\bar{\mu}_{i+1}^{np-2sp}
t^{\gamma_{i+1}p}|x|^{(2s-n)p}\bar{\mu}_{i+1}^{-n}t^{-\gamma_{i+1}}|x|^{n+2s}\omega_{1,i+1}''\lesssim t_0^{-l}\omega_{1,i+1}'',
\end{equation*}
\begin{equation*}
(\omega_{2,i-1}^*)^p\lesssim t^{-p\sigma}\mu_{0,i-1}^{(s-\frac{n}{2})p}\lesssim t^{-p\sigma}\mu_{0,i-1}^{(s-\frac{n}{2})p}\mu_{0,i}^{-\alpha}t^{-\gamma_i}|x|^{2s+\alpha}\omega_{1,i}\lesssim t_0^{-l}\omega_{1,i},
\end{equation*}
\begin{equation*}
(\omega_{2,i}^*)^p\lesssim t^{-p\sigma}\mu_{0,i+1}^{(\frac{n}{2}-2s)p}\mu_{0,j}^{-sp}|x|^{(4s-n)p}t^{\sigma}
\mu_{0,i+1}^{2s-\frac{n}{2}}\mu_{0,i}^s|x|^{n-2s}\omega_{2,i}\lesssim t_0^{-l}\omega_{2,i}.
\end{equation*}
Thus,   we obtain that
\begin{equation*}
\begin{aligned}
|\Psi|^p \lesssim((\omega_{1,i}'')^*+(\omega_{1,i+1}'')^*+\omega_{2,i-1}^*
+\omega_{2,i}^*)^p(\|\Psi\|_{\alpha,\sigma}^{out,*})^p \lesssim t_0^{-l}(\omega_{1,i}+\omega_{2,i}+\omega_{1,i+1}'')(\|\Psi\|_{\alpha,\sigma}^{out,*})^p.
\end{aligned}
\end{equation*}		

\noindent (5) ~\ In $\{0\le|x|\le \bar\mu_{0,k}\},$ we get
\begin{equation*}
|\Psi|\lesssim ((w_{1,k}'')^*+w_{2,k-1}^*)\|\Psi\|_{\alpha,\sigma}^{out,*}.
\end{equation*}
Then
\begin{equation*}
((\omega_{1,k}'')^*)^p\lesssim t^{\gamma_kp}\mu_{0,k}^{2s}t^{-\gamma_k}\omega_{1,k}\lesssim t_0^{-l}\omega_{1,k},
\end{equation*}
\begin{equation*}
(\omega_{2,k-1}^*)^p\lesssim t^{-p\sigma}\mu_{0,k-1}^{(s-\frac{n}{2})p}\lesssim t^{-p\sigma}\mu_{0,k-1}^{(s-\frac{n}{2})p}\mu_{0,k}^{-\alpha}t^{-\gamma_k}|x|^{2s+\alpha}\omega_{1,k}\lesssim t_0^{-l}\omega_{1,k}.
\end{equation*}
Thus,
\begin{equation*}
|\Psi|^p \lesssim((\omega_{1,k}'')^*+\omega_{2,k-1}^*)^p(\|\Psi\|_{\alpha,\sigma}^{out,*})^p \lesssim t_0^{-l}\omega_{1,k}(\|\Psi\|_{\alpha,\sigma}^{out,*})^p.
\end{equation*}

Combining $\eqref{eqa.24}$ and $(1)-(5),$ we complete the proof of Lemma $\ref{lea.8}.$
\end{proof}

\section{Appendix: some estimates for outer problem}
In this section we shall present several useful estimates for the fractional heat operator $\partial_t+(-\Delta)^s$. Recall that the fractional heat kernel is given by
\begin{equation*}
K_s(x,t)=\frac{t}{(t^{\frac{1}{s}}+|x|^2)^{\frac{n+2s}{2}}}.
\end{equation*}
Then for
\begin{equation*}
\mathcal{T}^{out}[f](x,t):=\int_{t}^{\infty}\int_{\mathbb{R}^n}K_s(x-z,l-t)f(z,l)\,dzdl,
\end{equation*}
we have  the following lemmas.
\begin{lemma}\label{leb.1}
Suppose that $n>6s,$ $a\in\{n+2s,n+s,n-2s,2s+\alpha,0\},$ $0\le c_1,c_2\le c_{**},$ $d_1\le d_2\le\frac{1}{2s}$ and $b$ satisfies
\begin{equation}\label{eqb.1}
\begin{cases}
\frac{n}{2s}-b+d_2(a-n)>1&\text{if $a\in\{0,n-2s,2s+\alpha\},$}\\
\frac{n}{2s}-b+d_1(a-n)>1&\text{if $a\in\{n+s,n+2s\},$}
\end{cases}
\end{equation}
Then there exists $C$ depending on $n,a,b,d_1,d_2,c_{**}$ such that for $t>1,$
\begin{equation*}
u(x,t):=\mathcal{T}^{out}\left[\frac{t^b}{|x|^a}1_{\{c_1t^{d_1}\le|x|\le c_2t^{d_2}\}}\right](x,t)\le C
\begin{cases}
t^{b+d_1(2s-a)}&\text{if $a\in\{n+2s,n+s,n-2s,2s+\alpha\},$}\\
t^{b+d_2(2s-a)}&\text{if $a=0.$}\\
\end{cases}
\end{equation*}
Moreover, in the case of
$$a\in\{n-2s,2s+\alpha,0\},\quad b+d_2(n-a)<0,$$
we have
\begin{equation}
\label{eqb.2}
u(x,t)\lesssim t^{b+d_2(n-a)}|x|^{2s-n}\quad\mbox{for}\quad
\begin{cases}
|x|>2c_2|2t|^{d_2}&\text{if\quad $0<d_2\le\frac{1}{2s},$}\\
|x|>2c_2|t|^{d_2}&\text{if\quad $d_2<0.$}\\
\end{cases}
\end{equation}
While in the case of
$$a\in\{n+2s,n+s\}, \quad b<0,\quad b+d_1(n-a)<0,$$
we get
\begin{equation}
\label{eqb.3}
u(x,t)\lesssim t^{b+d_1(n-a)}|x|^{2s-n}	
\quad\mathrm{for}\quad
\begin{cases}
|x|>2c_1|2t|^{d_1}&\text{if\quad $0<d_1\le\frac{1}{2s},$}\\
|x|>2c_1|t|^{d_1}&\text{if\quad $d_1<0.$}\\
\end{cases}
\end{equation}
\end{lemma}
\begin{proof}
When $x=0,$
\begin{equation*}
\begin{split}
\mathcal{T}^{out}\left[\frac{t^b}{|x|^a}1_{\{c_1t^{d_1}\le|x|\le c_2t^{d_2}\}}\right](0,t)&=\int_{t}^{\infty}
\int_{\mathbb{R}^n}\frac{l-t}{((l-t)^{\frac{1}{s}}+|z|^2)^{\frac{n+2s}{2}}}\frac{l^b}{|z|^a}1_{\{c_1l^{d_1}
\le|z|\le c_2l^{d_2}\}}\,dzdl\\&=\int_{t}^{\infty}\int_{\mathbb{R}^n}\frac{1}{(l-t)^{\frac{n}{2s}}}
\cdot\left(\frac{1}{1+\frac{|z|^2}{(l-t)^{\frac{1}{s}}}}\right)^{\frac{n+2s}{2}}
\frac{l^b}{|z|^a}1_{\{c_1l^{d_1}\le|z|\le c_2l^{d_2}\}}\,dzdl\\&=\int_{t}^{\infty}\int_{c_1 l^{d_1}}^{c_2 l^{d_2}}\frac{l^b}{(l-t)^{\frac{n}{2s}}}
\left(\frac{1}{1+\frac{r^2}{(l-t)^{\frac{1}{s}}}}\right)^{\frac{n+2s}{2}}r^{n-1-a}\,drdl\\
&=\int_{t}^{\infty}\frac{l^b}{(l-t)^{\frac{a}{2s}}}F\left(\frac{c_1^2 l^{2d_1}}{(l-t)^{\frac{1}{s}}},\frac{c_2^2 l^{2d_2}}{(l-t)^{\frac{1}{s}}}\right)\,dl,
\end{split}
\end{equation*}
where
\begin{equation*}
F(A,B):=\int_{A}^{B}\frac{1}{(1+y)^{\frac{n+2s}{2}}}y^{\frac{n-a-2}{2}}\,dy.
\end{equation*}
Next, we divide
$$[t,\infty)=[t,t+c_1^2t^{2d_1s}]\cup[t+c_1^2t^{2d_1s},t+c_2^2 t^{2d_2s}]\cup[t+c_2^2 t^{2d_2s},2t+c_2^2t^{2d_2s}]\cup[2t+c_2^2t^{2d_2s},+\infty).$$
In the following, we shall give the estimation for $u(0,t)$ by analyzing the integrals case by case.
	\medskip
	
\noindent (1) In the region $l\in[t,t+c_1^2t^{2d_1s}],$ we have
\begin{equation*}
\begin{split}
I_1:&=\int_{t}^{t+c_1^2t^{2d_1s}}\frac{l^b}{(l-t)^{\frac{a}{2s}}}F\left(\frac{c_1^2 l^{2d_1}}{(l-t)^{\frac{1}{s}}},\frac{c_2^2 l^{2d_2}}{(l-t)^{\frac{1}{s}}}\right)\,dl\\
&\lesssim t^b\int_{t}^{t+c_1^2t^{2d_1s}}\frac{1}{(l-t)^{\frac{a}{2s}}}
F\left(\frac{c_1^2t^{2d_1}(1+c_{**}^2)^{2(d_1)^{-}}}{(l-t)^{\frac{1}{s}}},
\frac{c_2^2t^{2d_2}(1+c_{**}^2)^{2(d_2)^{+}}}{(l-t)^{\frac{1}{s}}}\right)\,dl,
\end{split}
\end{equation*}
where $d^{-}:=\min\{0,d\},d^{+}:=\max\{0,d\}.$ If $a=0,$ then
\begin{equation*}
I_1\lesssim t^b\cdot(c_1^2t^{2d_1s})\lesssim t^{b+2d_1s}.
\end{equation*}
While if $a\in\{n+2s,n+s,n-2s,2s+\alpha\},$ then
\begin{equation*}
\begin{split}
I_1\lesssim t^b\int_{t}^{t+c_1^2 t^{2d_1s}}\frac{1}{(l-t)^{\frac{a}{2s}}}
\frac{1}{\left(1+\frac{c_1^2t^{2d_1}(1+c_{**}^2)^{2(d_1)^{-}}}{(l-t)^{\frac{1}{s}}}\right)^{2s}}\,dl\lesssim t^{b+d_1(2s-a)}.
\end{split}
\end{equation*}

\noindent	$(2)$ In the region $l\in[t+c_1^2t^{2d_1s},t+c_2^2 t^{2d_2s}],$
\begin{equation*}
\begin{split}
I_2:&=\int_{t+c_1^2t^{2d_1s}}^{t+c_2^2 t^{2d_2s}}\frac{l^b}{(l-t)^{\frac{a}{2s}}}F\left(\frac{c_1^2 l^{2d_1}}{(l-t)^{\frac{1}{s}}},\frac{c_2^2 l^{2d_2}}{(l-t)^{\frac{1}{s}}}\right)\,dl\\&\lesssim\int_{t+c_1^2t^{2d_1s}}^{t+c_2^2 t^{2d_2s}}\frac{l^b}{(l-t)^{\frac{a}{2s}}}F\left(\frac{c_1^2 t^{2d_1}(1+c_{**}^2)^{2(d_1)^{-}}}{(l-t)^{\frac{1}{s}}},\frac{c_2^2 t^{2d_2}(1+c_{**}^2)^{2(d_2)^{+}}}{(l-t)^{\frac{1}{s}}}\right)\,dl\\
&\lesssim\begin{cases}
\int_{t+c_1^2t^{2d_1s}}^{t+c_2^2 t^{2d_2s}}\frac{t^b}{(l-t)^{\frac{a}{2s}}}dl\quad &\mathrm{if}~ a\in\{n-2s,2s+\alpha,0\},\\
\int_{t+c_1^2t^{2d_1s}}^{t+c_2^2 t^{2d_2s}}\frac{t^b}{(l-t)^{\frac{a}{2s}}}\left(\frac{t^{2d_1}}{(l-t)^{\frac{1}{s}}}\right)^{\frac{n-a}{2}}dl
\quad &\mathrm{if}~ a\in\{n+2s,n+s\},\\
\end{cases}
\\&\lesssim
\begin{cases}
t^{b+d_1(2s-a)}&\text{if~\  $a\in\{n+2s,n+s,n-2s,2s+\alpha\},$}\\
t^{b+d_2(2s-a)}&\text{if~\  $a=0.$}
\end{cases}
\end{split}
\end{equation*}
	
\noindent	$(3)$ In the region $l\in[t+c_2^2 t^{2d_2s},2t+c_2^2t^{2d_2s}],$
\begin{equation*}
\begin{split}
I_3:&=\int_{t+c_2^2 t^{2d_2s}}^{2t+c_2^2t^{2d_2s}}\frac{l^b}{(l-t)^{\frac{a}{2s}}}F\left(\frac{c_1^2 l^{2d_1}}{(l-t)^{\frac{1}{s}}},\frac{c_2^2 l^{2d_2}}{(l-t)^{\frac{1}{s}}}\right)\,dl\\&\lesssim\int_{t+c_2^2 t^{2d_2s}}^{2t+c_2^2t^{2d_2s}}\frac{t^b}{(l-t)^{\frac{a}{2s}}}F\left(\frac{c_1^2 l^{2d_1}(2+c_{**}^2)^{2(d_1)^{-}}}{(l-t)^{\frac{1}{s}}},\frac{c_2^2 l^{2d_2}(2+c_{**}^2)^{2(d_2)^{+}}}{(l-t)^{\frac{1}{s}}}\right)\,dl\\&\lesssim
\begin{cases}
\int_{t+c_2^2 t^{2d_2s}}^{2t+c_2^2t^{2d_2s}}\frac{t^b}{(l-t)^{\frac{a}{2s}}}\quad \left(\frac{t^{2d_2}}{(l-t)^{\frac{1}{s}}}\right)^{\frac{n-a}{2}}dl &\text{if~\  $a\in\{n-2s,2s+\alpha,0\},$}\\	
\int_{t+c_2^2 t^{2d_2s}}^{2t+c_2^2t^{2d_2s}}\frac{t^b}{(l-t)^{\frac{a}{2s}}}\quad \left(\frac{t^{2d_1}}{(l-t)^{\frac{1}{s}}}\right)^{\frac{n-a}{2}}dl &\text{if~\  $a\in\{n+2s,n+s\},$}	
\end{cases}
\\&\lesssim
\begin{cases}
t^{b+d_1(2s-a)}&\text{if~\  $a\in\{n+2s,n+s,n-2s,2s+\alpha\},$}\\
t^{b+d_2(2s-a)}&\text{if~\  $a=0.$}
\end{cases}
\end{split}
\end{equation*}
	
\noindent	$(4)$ In the region $l\in[2t+c_2^2t^{2d_2s},+\infty),$ for $a\in\{0,n-2s,2s+\alpha\},$ we have
\begin{equation*}
\begin{split}
I_4:&=\int_{2t+c_2^2 t^{2d_2s}}^{\infty}\frac{l^b}{(l-t)^{\frac{a}{2s}}}F\left(\frac{c_1^2 l^{2d_1}}{(l-t)^{\frac{1}{s}}},\frac{c_2^2 l^{2d_2}}{(l-t)^{\frac{1}{s}}}\right)\,dl\\&\lesssim\int_{2t+c_2^2 t^{2d_2s}}^{\infty} l^{b-\frac{a}{2s}}F\left(\frac{c_1^2 l^{2d_1}}{(2l)^{\frac{1}{s}}},\frac{c_2^2 l^{2d_2}}{(\frac{1}{2}l)^{\frac{1}{s}}}\right)\,dl\\&\lesssim
\begin{cases}
\int_{2t+c_2^2 t^{2d_2s}}^{\infty}l^{b-\frac{a}{2s}}\cdot l^{(2d_2-\frac{1}{s})\frac{n-a}{2}}\,dl \quad &\mathrm{if}~ a\in\{n-2s,2s+\alpha,0\},\\
\int_{2t+c_2^2 t^{2d_2s}}^{\infty}l^{b-\frac{a}{2s}}\cdot l^{(2d_1-\frac{1}{s})\frac{n-a}{2}}\,dl \quad &\mathrm{if}~ a\in\{n+2s,n+s\},
\end{cases}
\\&\lesssim
\begin{cases}
t^{b+d_2(2s-a)},\quad &\mathrm{if}~ a\in\{n-2s,2s+\alpha,0\},\\
t^{b+d_1(2s-a)},\quad &\mathrm{if}~ a\in\{n+2s,n+s\},
\end{cases}
\end{split}
\end{equation*}
where we have used
\begin{equation*}
\begin{cases}
\frac{n}{2s}-b+d_2(a-n)>1~&\mbox{if}~\quad~ a\in\{n-2s,2s+\alpha,0\},\\
\frac{n}{2s}-b+d_1(a-n)>1~&\mbox{if}~\quad~
a\in\{n+2s,n+s\}.
\end{cases}
\end{equation*}
Combining with the estimations from $I_1$ to $I_4,$ we see that
\begin{equation*}
\mathcal{T}^{out}\left[\frac{t^b}{|x|^a}1_{\{c_1t^{d_1}\le|x|\le c_2t^{d_2}\}}\right](0,t)\le C
\begin{cases}
t^{b+d_1(2s-a)}&\text{if $a\in\{n+2s,n+s,n-2s,2s+\alpha\},$}\\
t^{b+d_2(2s-a)}&\text{if $a=0.$}\\
\end{cases}
\end{equation*}
In particular, if $c_1=0$ and $a=0$, by repeating almost the same argument as above we get that
\begin{equation}
\label{b.2}
\mathcal{T}^{out}\left[\frac{t^b}{|x|^a}1_{\{|x|\le c_2 t^{d_2}\}}\right](0,t)\lesssim t^{b+2sd_2}.
\end{equation}
Note that $K_s(x,t)$ and $\frac{t^b}{|x|^a}1_{\{c_1 t^{d_1}\le|x|\le c_2t^{d_2}\}}$ are both decreasing functions for each time slice. Using Hardy-Little-wood rearrangement inequality,  we know that
\begin{equation*}
u(x,t)\le u(0,t).
\end{equation*}
Hence,
\begin{equation}
\label{b.3}
u(x,t)\lesssim
\begin{cases}
t^{b+d_1(2s-a)}&\text{if $a\in\{2s+\alpha,n+2s,n-2s,n+s\},$}\\
t^{b+d_2(2s-a)}&\text{if $a=0.$}
\end{cases}
\end{equation}
	
Next, we study the behavior of $u(x,t)$ for  $|x|\ge 2c_2|2t|^{d_2},d_2>0, a\in\{n-2s,2s+\alpha,0\}.$ In this case, we have
$\left(\frac{|x|}{2c_2}\right)^{\frac{1}{d_2}}\ge 2t$ and divide
\begin{equation*}	[t,+\infty)=\left[t,\left(\frac{|x|}{2c_2}\right)^{\frac{1}{d_2}}\right]
\cup\left[\left(\frac{|x|}{2c_2}\right)^{\frac{1}{d_2}},\left(\frac{2|x|}{c_2}\right)^{\frac{1}{d_2}}\right]
\cup\left[\left(\frac{2|x|}{c_2}\right)^{\frac{1}{d_2}},+\infty\right).
\end{equation*}
Then we can write
$$u(x,t)=I_5+I_6+I_7,$$
where
\begin{equation*}
\begin{split}		I_5(x,t):&=\int_{t}^{\left(\frac{|x|}{2c_2}\right)^{\frac{1}{d_2}}}
\int_{\mathbb{R}^n}\frac{1}{(l-t)^{\frac{n}{2s}}}\frac{1}{\left(1+\frac{|x-z|^2}{(l-t)^{\frac{1}{s}}}\right)
^{\frac{n+2s}{2}}}\frac{l^b}{|z|^a}1_{\{c_1l^{d_1}\le|z|\le c_2 l^{d_2}\}}\,dzdl\\& \lesssim\int_{t}^{\infty}\int_{\mathbb{R}^n}\frac{1}{(l-t)^{\frac{n}{2s}}}
\frac{1}{\left(1+\frac{|x|^2}{4(l-t)^{\frac{1}{s}}}\right)^{\frac{n+2s}{2}}}\frac{l^b}{|z|^a}
1_{\{c_1l^{d_1}\le|z|\le c_2 l^{d_2}\}}\,dzdl\\&\lesssim\int_{t}^{\infty}\int_{c_1l^{d_1}}^{c_2l^{d_2}}\frac{1}{(l-t)^{\frac{n}{2s}}}
\frac{1}{\left(1+\frac{|x|^2}{4(l-t)^{\frac{1}{s}}}\right)^{\frac{n+2s}{2}}}l^b r^{n-1-a}\,drdl\\&\lesssim\int_{t}^{\infty}\frac{1}{(l-t)^{\frac{n}{2s}}}
\frac{1}{\left(1+\frac{|x|^2}{4(l-t)^{\frac{1}{s}}}\right)^{\frac{n+2s}{2}}}l^{b+d_2(n-a)}\,dl\\&\lesssim t^{b+d_2(n-a)}|x|^{2s-n} \int_{0}^{\infty}\frac{1}{(1+y)^{\frac{n+2s}{2}}}y^{\frac{n}{2}-s-1}\,dy\lesssim t^{b+d_2(n-a)}|x|^{2s-n},
\end{split}
\end{equation*}
\begin{equation*}
\begin{split}	I_6(x,t):&=\int_{\left(\frac{|x|}{2c_2}\right)^{\frac{1}{d_2}}}^{\left(\frac{2|x|}{c_2}\right)^{\frac{1}{d_2}}}
\int_{\mathbb{R}^n}\frac{1}{(l-t)^{\frac{n}{2s}}}\frac{1}{\left(1+\frac{|x-z|^2}{(l-t)^{\frac{1}{s}}}\right)
^{\frac{n+2s}{2}}}\frac{l^b}{|z|^a}1_{\{c_1l^{d_1}\le|z|\le c_2 l^{d_2}\}}\,dzdl\\&\lesssim|x|^{\frac{b}{d_2}}\int_{\left(\frac{|x|}{2c_2}\right)^{\frac{1}{d_2}}}
^{\left(\frac{2|x|}{c_2}\right)^{\frac{1}{d_2}}}\int_{\mathbb{R}^n}\frac{1}{(l-t)^{\frac{n}{2s}}}
\frac{1}{\left(1+\frac{|x-z|^2}{(l-t)^{\frac{1}{s}}}\right)^{\frac{n+2s}{2}}}\frac{1}{|z|^a}1_{\{|z|\le 2|x|\}}\,dzdl\\&\lesssim|x|^{\frac{b}{d_2}}\cdot|x|^{n-a}
\int_{\left(\frac{|x|}{2c_2}\right)^{\frac{1}{d_2}}}^{\left(\frac{2|x|}{c_2}\right)^{\frac{1}{d_2}}}
\frac{1}{(l-t)^{\frac{n}{2s}}}\,dl\lesssim t^{b+d_2(n-a)}|x|^{2s-n},
\end{split}
\end{equation*}
\begin{equation*}
\begin{split}
I_7(x,t):&=\int_{\left(\frac{2|x|}{c_2}\right)^{\frac{1}{d_2}}}^{+\infty}\int_{\mathbb{R}^n}
\frac{1}{(l-t)^{\frac{n}{2s}}}\frac{1}{\left(1+\frac{|x-z|^2}{(l-t)^{\frac{1}{s}}}\right)
^{\frac{n+2s}{2}}}\frac{l^b}{|z|^a}1_{\{c_1l^{d_1}\le|z|\le c_2 l^{d_2}\}}\,dzdl\\&\lesssim\int_{\left(\frac{2|x|}{c_2}\right)^{\frac{1}{d_2}}}^{+\infty}
l^{b-\frac{n}{2s}+d_2(n-a)}\,dl\lesssim t^{b+d_2(n-a)}|x|^{2s-n}.
\end{split}
\end{equation*}
Hence
\begin{equation}
\label{b.4}
u(x,t)\lesssim t^{b+d_2(n-a)}|x|^{2s-n}.
\end{equation}
When  $|x|\ge 2c_2|t|^{d_2},d_2<0, a\in\{n-2s,2s+\alpha,0\},$ we find that
\begin{equation*}
\frac{1}{2}|x|\le|x-z|\le 2|x|\quad\mbox{for}\quad|z|\le c_2l^{d_2}~\mathrm{and}~l\ge t.
\end{equation*}
Then we can deduce that
\begin{equation}
\label{b.5}
\begin{split}		u(x,t)&\lesssim\int_{t}^{\infty}\int_{\mathbb{R}^n}\frac{1}{(l-t)^{\frac{n}{2s}}}
\frac{1}{\left(1+\frac{|x|^2}{4(l-t)^{\frac{1}{s}}}\right)^{\frac{n+2s}{2}}}
\frac{l^b}{|z|^a}1_{\{c_1s^{d_1}\le|z|\le c_2 l^{d_2}\}}\,dzdl\\&\lesssim\int_{t}^{\infty}\int_{c_1l^{d_1}}^{c_2l^{d_2}}\frac{1}{(l-t)^{\frac{n}{2s}}}
\frac{1}{\left(1+\frac{|x|^2}{4(l-t)^{\frac{1}{s}}}\right)^{\frac{n+2s}{2}}}l^b r^{n-1-a}\,drdl\\&\lesssim\int_{t}^{\infty}\frac{1}{(l-t)^{\frac{n}{2s}}}
\frac{1}{\left(1+\frac{|x|^2}{4(l-t)^{\frac{1}{s}}}\right)^{\frac{n+2s}{2}}}l^{b+d_2(n-a)}\,dl
\\&\lesssim t^{b+d_2(n-a)}|x|^{2s-n} \int_{0}^{\infty}\frac{1}{(1+y)^{\frac{n+2s}{2}}}y^{\frac{n}{2}-s-1}\,dy\\&\lesssim t^{b+d_2(n-a)}|x|^{2s-n}.
\end{split}
\end{equation}
Combining \eqref{b.4} and \eqref{b.5} we get \eqref{eqb.2}, i.e.,
\begin{equation}
\label{b.8}
u(x,t)\lesssim t^{b+d_2(n-a)}|x|^{2s-n} \quad\mbox{for}\quad \begin{cases}
|x|>2c_2|2t|^{d_2}~ &\text{if\quad $0<d_2\le\frac{1}{2s},$}\\
|x|>2c_2|t|^{d_2}~ &\text{if\quad $d_2<0,$}\\
\end{cases}
\end{equation}
in the case of
$$a\in\{0,n-2s,2s+\alpha\},\quad b+d_2(n-a)<0.$$
In a similar way of deriving \eqref{b.8} we also obtain that
\begin{equation}
\label{b.9}
u(x,t)\lesssim t^{b+d_1(n-a)}|x|^{2s-n}\quad\mbox{for}\quad	\begin{cases}
|x|>2c_1|2t|^{d_1}~ &\text{if\quad $0<d_1\le\frac{1}{2s},$}\\
|x|>2c_1|t|^{d_1}~&\text{if\quad $d_1<0,$}\\
\end{cases}
\end{equation}
in the case of $$a\in\{n+2s,n+s\},\quad b<0,\quad b+d_1(n-a)<0.$$
\end{proof}

\begin{lemma}\label{leb.3}
Suppose that $a\in\{n-2s,2s+\alpha\},$ $d_2\le\frac{1}{2s},$ $0<c_2\le c_{**},$  $\frac{n}{2s}-b>1$ if $b>0$ and \eqref{eqb.1} holds. Then there exists $C$ depending on $n,a,b,d_2,c_{**}$ such that for $t>1,$
\begin{equation*}
\mathcal{T}^{out}\left[\frac{t^b}{|x|^a}1_{\{|x|\le c_2t^{d_2}\}}\right](x,t)\le Ct^b|x|^{2s-a}\quad \mbox{for}\quad  |x|<8c_2t^{d_2}.
\end{equation*}
\end{lemma}
\begin{proof}
We write
\begin{equation*}
\begin{split}	u(x,t)&=\int_{t}^{\infty}\int_{\mathbb{R}^n}\frac{1}{(l-t)^{\frac{n}{2s}}}
\frac{1}{\left(1+\frac{|x-z|^2}{(l-t)^{\frac{1}{s}}}\right)^{\frac{n+2s}{2}}}\frac{l^b}{|z|^a}1_{\{|z|\le c_2l^{d_2}\}}\,dzdl\\&\lesssim\int_{t}^{\infty}\int_{\mathbb{R}^n}\frac{1}{(l-t)^{\frac{n}{2s}}}
\frac{1}{\left(1+\frac{|x-z|^2}{(l-t)^{\frac{1}{s}}}\right)^{\frac{n+2s}{2}}}\frac{l^b}{|z|^a}
\left(1_{\{|z|\le \frac{1}{2}|x|\}}+1_{\{\frac{1}{2}|x|\le|z|\le 2|x|\}}+1_{\{2|x|\le|z|\le c_2l^{d_2}\}}\right)\,dzdl\\
&=:I_8+I_9+I_{10}
\end{split}
\end{equation*}
Then
\begin{equation*}
\begin{split}	I_8:=&\int_{t}^{\infty}\int_{\mathbb{R}^n}\frac{1}{(l-t)^{\frac{n}{2s}}}
\frac{1}{\left(1+\frac{|x-z|^2}{(l-t)^{\frac{1}{s}}}\right)^{\frac{n+2s}{2}}}\frac{l^b}{|z|^a}1_{\{|z|\le \frac{1}{2}|x|\}}\,dzdl\\&\lesssim\int_{t}^{\infty}\int_{\mathbb{R}^n}
\frac{1}{(l-t)^{\frac{n}{2s}}}\frac{1}{\left(1+\frac{|x|^2}{4(l-t)^{\frac{1}{s}}}\right)^{\frac{n+2s}{2}}}
\frac{l^b}{|z|^a}1_{\{|z|\le \frac{1}{2}|x|\}}\,dzdl\\&\lesssim|x|^{n-a}\int_{t}^{\infty}\frac{1}{(l-t)^{\frac{n}{2s}}}
\frac{1}{\left(1+\frac{|x|^2}{4(l-t)^{\frac{1}{s}}}\right)^{\frac{n+2s}{2}}}l^b\,dl\\
&\lesssim|x|^{2s-a}\int_{0}^{\infty}\frac{1}{(1+y)^{\frac{n+2s}{2}}}y^{\frac{n}{2}-s-1}
\left(t+\frac{|x|^{2s}}{y^s}\right)^b\,dy\\&\lesssim t^b |x|^{2s-a}\int_{\frac{|x|^2}{t^{\frac{1}{s}}}}^{\infty}
\frac{1}{(1+y)^{\frac{n+2s}{2}}}y^{\frac{n}{2}-s-1}\,dy+|x|^{2sb+2s-a}
\int_{0}^{\frac{|x|^2}{t^{\frac{1}{s}}}}\frac{1}{(1+y)^{\frac{n+2s}{2}}}y^{-sb+\frac{n}{2}-s-1}\,dy\\
&\lesssim t^b|x|^{2s-a}+|x|^{n-a}t^{b-\frac{n}{2s}+1}\lesssim t^b|x|^{2s-a},
\end{split}
\end{equation*}
where we have used the condition that $\frac{n}{2s}-b>1$ if $b>0.$ The case $b<0$ is similar and we omit the details. Concerning $I_9$ and $I_{10}$ we have
\begin{equation*}
\begin{split} I_9:&=\int_{t}^{\infty}\int_{\mathbb{R}^n}\frac{1}{(l-t)^{\frac{n}{2s}}}
\frac{1}{\left(1+\frac{|x-z|^2}{(l-t)^{\frac{1}{s}}}\right)^{\frac{n+2s}{2}}}\frac{l^b}{|z|^a}
1_{\{\frac{1}{2}|x|\le|z|\le 2|x|\}}\,dzdl\\&\lesssim|x|^{-a}\int_{t}^{\infty}\int_{\mathbb{R}^n}\frac{1}{(l-t)^{\frac{n}{2s}}}
\frac{1}{\left(1+\frac{|x-z|^2}{(l-t)^{\frac{1}{s}}}\right)^{\frac{n+2s}{2}}}l^b1_{\{|x-z|\le3|x|\}}\,dzdl\\
&\lesssim|x|^{-a}\int_{t}^{\infty}\int_{0}^{3|x|}\frac{1}{(l-t)^{\frac{n}{2s}}}
\frac{1}{\left(1+\frac{r^2}{(l-t)^{\frac{1}{s}}}\right)^{\frac{n+2s}{2}}}l^br^{n-1}\,drdl\\
&=|x|^{-a}\int_{t}^{\infty}\int_{0}^{\frac{9|x|^2}{(l-t)^{\frac{1}{s}}}}\frac{1}{(1+y)^{\frac{n+2s}{2}}}l^b y^{\frac{n}{2}-1}\,dydl\\&\lesssim |x|^{-a}\int_{9^s|x|^{2s}+t}^{+\infty}l^b\left(\frac{9|x|^2}{(l-t)^{\frac{1}{s}}}\right)^{\frac{n}{2}}\,dl
+|x|^{-a}\int_{t}^{9^s|x|^{2s}+t}l^b\,dl\\&=t^b|x|^{n-a}\int_{9^s|x|^{2s}}^{\infty}
\frac{1}{y^{\frac{n}{2s}}}\,dy+|x|^{n-a}\int_{9^s|x|^{2s}}^{\infty}y^{b-\frac{n}{2s}}\,dy\cdot1_{\{b>0\}}
+t^{b}|x|^{2s-a}\\&\lesssim t^b|x|^{2s-a}
\end{split}
\end{equation*}
and
\begin{equation*}
\begin{split}	I_{10}:&=\int_{t}^{\infty}\int_{\mathbb{R}^n}\frac{1}{(l-t)^{\frac{n}{2s}}}
\frac{1}{\left(1+\frac{|x-z|^2}{(l-t)^{\frac{1}{s}}}\right)^{\frac{n+2s}{2}}}\frac{l^b}{|z|^a}
1_{\{2|x|\le|z|\le c_2l^{d_2}\}}\,dzdl\\&\lesssim\int_{t}^{\infty}\int_{\mathbb{R}^n}\frac{1}{(l-t)^{\frac{n}{2s}}}
\frac{1}{\left(1+\frac{|z|^2}{4(l-t)^{\frac{1}{s}}}\right)^{\frac{n+2s}{2}}}\frac{l^b}{|z|^a}
1_{\{2|x|\le|z|\le c_2 l^{d_2}\}}\,dzdl\\&\lesssim\int_{t}^{\infty}\int_{2|x|}^{c_2l^{d_2}}\frac{1}{(l-t)^{\frac{n}{2s}}}
\frac{1}{\left(1+\frac{r^2}{4(l-t)^{\frac{1}{s}}}\right)^{\frac{n+2s}{2}}}l^br^{n-1-a}\,drdl\\
&\lesssim\int_{t}^{\infty}
\int_{\frac{4|x|^2}{(l-t)^{\frac{1}{s}}}}^{\frac{c_2^2l^{2d_2}}{(l-t)^{\frac{1}{s}}}}
\frac{1}{(l-t)^{\frac{a}{2s}}}\frac{1}{(1+\frac{y}{4})^{\frac{n+2s}{2}}}l^b y^{\frac{n-2-a}{2}}\,dydl\\&\lesssim\int_{|x|^{2s}+t}^{\infty}\frac{l^{b+d_2(n-a)}}{(l-t)^{\frac{n}{2s}}}\,dl
+t^b\int_{t}^{|x|^{2s}+t}\frac{1}{(l-t)^{\frac{a}{2s}}}
\left(\frac{|x|^2}{(l-t)^{\frac{1}{s}}}\right)^{-\frac{a}{2}}\,dl\\&\lesssim t^{b}|x|^{2s-a}.
\end{split}
\end{equation*}
Collecting the estimations of $I_8,I_9,I_{10}$ we get that
\begin{equation}
\label{b.10}
\mathcal{T}^{out}\left[\frac{t^b}{|x|^a}1_{\{|x|\le c_2t^{d_2}\}}\right]\lesssim t^b|x|^{2s-a}\quad \mbox{for}\quad |x|<c_2t^{d_2}.
\end{equation}
\end{proof}

\begin{lemma}
\label{leb.5}
If $a=n-2s,-2<b<0,$ then
\begin{equation*}
\begin{split}
\mathcal{T}^{out}\left[\frac{t^b}{|x|^a}1_{\{|x|\ge t^{\frac{1}{2s}}\}}\right](x,t)&
\lesssim t^{1+b-\frac{a}{2s}}1_{\{|x|\le t^{\frac{1}{2s}}\}}+1_{\{|x|\ge t^{\frac{1}{2s}}\}}|x|^{-a}\cdot
\begin{cases}
t^{1+b}&\text{if $b<-1,$}\\
1+\ln\left(\frac{|x|^{2s}}{t}\right)&\text{if $b=-1,$}\\
(|x|^{2s})^{1+b}&\text{if $b>-1.$}
\end{cases}
\end{split}
\end{equation*}
If $a=n+2s$ and $b<0,$ then
\begin{equation*}
\mathcal{T}^{out}\left[\frac{t^b}{|x|^a}1_{\{|x|\ge t^{\frac{1}{2s}}\}}\right](x,t)=t^{1+b-\frac{a}{2s}}1_{\{|x|\le t^{\frac{1}{2s}}\}}+1_{\{|x|\ge t^{\frac{1}{2s}}\}}t^b|x|^{2s-a}.
\end{equation*}
\end{lemma}

\begin{proof}
We shall prove the conclusion by studying the term $\mathcal{T}^{out}\left[\frac{t^b}{|x|^a}1_{\{|x|\ge t^{\frac{1}{2s}}\}}\right](x,t)$ for $|x|\le\frac{1}{2}t^{\frac{1}{2s}},$ $\frac{1}{2}t^{\frac{1}{2s}}\le|x|\le4t^{\frac{1}{2s}}$ and
$|x|\ge 4t^{\frac{1}{2s}}$ respectively.
\medskip

\noindent (1)  If $|x|\le\frac{1}{2}t^{\frac{1}{2s}},$ we find that $|x-z|\ge\frac{1}{2}|z|$ for $|z|\ge l^{\frac{1}{2s}}\ge t^{\frac{1}{2s}}.$ Then applying the arguments of Lemma \ref{leb.1} we have
\begin{equation}
\label{b.11}	u(x,t)\lesssim\int_{t}^{\infty}\int_{\mathbb{R}^n}\frac{1}{(l-t)^{\frac{n}{2s}}}
\frac{1}{\left(1+\frac{|z|^2}{4(l-t)^{\frac{1}{s}}}\right)^{\frac{n+2s}{2}}}\frac{l^b}{|z|^a}1_{\{|z|\ge l^{\frac{1}{2s}}\}}\,dzdl\lesssim t^{1+b-\frac{a}{2s}}.
\end{equation}
	
\noindent (2) If $|x|\ge 4t^{\frac{1}{2s}},$ we divide $$[t,\infty)=\left[t,\frac{2}{4^{2s}}|x|^{2s}\right]\cup
\left[\frac{2}{4^{2s}}|x|^{2s},2\cdot4^{2s}|x|^{2s}\right]\cup\left[2\cdot4^{2s}|x|^{2s},\infty\right).$$
For $t\in\left[t,\frac{2}{4^{2s}}|x|^{2s}\right]$, we write
\begin{equation*}
\begin{aligned}
I_{11}&=\int_{t}^{\frac{2}{4^{2s}}|x|^{2s}}\int_{\mathbb{R}^n}\frac{1}{(l-t)^{\frac{n}{2s}}}
\frac{1}{\left(1+\frac{|x-z|^2}{(l-t)^{\frac{1}{s}}}\right)^{\frac{n+2s}{2}}}\frac{l^b}{|z|^a}1_{\{|z|\ge l^{\frac{1}{2s}}\}}\,dzdl\\&=\int_{t}^{\frac{2}{4^{2s}}|x|^{2s}}\int_{\mathbb{R}^n}
\frac{1}{(l-t)^{\frac{n}{2s}}}\frac{1}{\left(1+\frac{|x-z|^2}{(l-t)^{\frac{1}{s}}}\right)^{\frac{n+2s}{2}}}
\frac{l^b}{|z|^a}1_{\{l^{\frac{1}{2s}}\le|z|\le\frac{1}{2}|x|\}}\,dzdl\\
&\quad+\int_{t}^{\frac{2}{4^{2s}}|x|^{2s}}\int_{\mathbb{R}^n}\frac{1}{(l-t)^{\frac{n}{2s}}}
\frac{1}{\left(1+\frac{|x-z|^2}{(l-t)^{\frac{1}{s}}}\right)^{\frac{n+2s}{2}}}
\frac{l^b}{|z|^a}1_{\{\frac{1}{2}|x|\le|z|\le2|x|\}}\,dzdl\\
&\quad+\int_{t}^{\frac{2}{4^{2s}}|x|^{2s}}\int_{\mathbb{R}^n}\frac{1}{(l-t)^{\frac{n}{2s}}}
\frac{1}{\left(1+\frac{|x-z|^2}{(l-t)^{\frac{1}{s}}}\right)^{\frac{n+2s}{2}}}
\frac{l^b}{|z|^a}1_{\{|z|\ge2|x|\}}\,dzdl\\&=:J_1+J_2+J_3.
\end{aligned}
\end{equation*}
Then for $a=n-2s,~-2<b<0,$
\begin{equation}
\label{b.12}
\begin{aligned}	J_1:&=\int_{t}^{\frac{2}{4^{2s}}|x|^{2s}}\int_{\mathbb{R}^n}\frac{1}{(l-t)^{\frac{n}{2s}}}
\frac{1}{\left(1+\frac{|x-z|^2}{(l-t)^{\frac{1}{s}}}\right)^{\frac{n+2s}{2}}}
\frac{l^b}{|z|^a}1_{\{l^{\frac{1}{2s}}\le|z|\le\frac{1}{2}|x|\}}\,dzdl\\
&\lesssim\int_{t}^{\frac{2}{4^{2s}}|x|^{2s}}\frac{1}{(l-t)^{\frac{n}{2s}}}
\frac{1}{\left(1+\frac{|x|^2}{4(l-t)^{\frac{1}{s}}}\right)^{\frac{n+2s}{2}}}l^b|x|^{n-a}\,dl\\
&\lesssim|x|^{n-a}\int_{t}^{\frac{2}{4^{2s}}|x|^{2s}}\frac{1}{(l-t)^{\frac{n}{2s}}}
\frac{l^b}{\left(1+\frac{|x|^2}{4(l-t)^{\frac{1}{s}}}\right)^{\frac{n+2s}{2}}}\,dl \lesssim |x|^{2sb+2s-a},
\end{aligned}
\end{equation}
While if $a=n+2s,b<0,$ then
\begin{equation}
\label{b.13}
\begin{aligned}
J_1&\lesssim\int_{t}^{\frac{2}{4^{2s}}|x|^{2s}}\frac{1}{(l-t)^{\frac{n}{2s}}}
\frac{1}{\left(1+\frac{|x|^2}{4(l-t)^{\frac{1}{s}}}\right)^{\frac{n+2s}{2}}}l^b|x|^{n-a}\,dl\\
&\lesssim\int_{\frac{|x|^2}{4
\left(\frac{2}{4^{2s}}|x|^{2s}-t\right)^{\frac{1}{s}}}}^{\infty}\frac{|x|^{2s-n}
\left(t+\frac{|x|^{2s}}{(4y)^s}\right)^{b+\frac{n-a}{2s}}}{y^{s+1-\frac{n}{2}}(1+y)^{\frac{n+2s}{2}}}dy\\
&\lesssim\int_{\frac{|x|^2}{4t^{\frac{1}{s}}}}^\infty
\frac{|x|^{2s-n}t^{b+\frac{n-a}{2s}}}{y^{s+1-\frac{n}{2}}(1+y)^{\frac{n+2s}{2}}}dy+\int_{\frac{|x|^2}{4
		\left(\frac{2}{4^{2s}}|x|^{2s}-t
		\right)^{\frac{1}{s}}}}^{\frac{|x|^2}{4t^{\frac{1}{s}}}}
\frac{|x|^{2s-n+2sb+n-a}}{y^{sb+\frac{n-a}{2}+s+1-\frac{n}{2}}(1+y)^{\frac{n+2s}{2}}}dy\\
&\lesssim |x|^{-2s-n}t^{2+b+\frac{n-a}{2s}}+|x|^{2s+2sb-a}\max\left\{1,
t^{b+1}|x|^{-2sb-2s}\right\}\lesssim t^b|x|^{2s-a}.
\end{aligned}
\end{equation}
Concerning $J_2$ and $J_3$ we have
\begin{equation}
\label{b.14}
\begin{aligned}		J_2:&=\int_{t}^{\frac{2}{4^{2s}}|x|^{2s}}\int_{\mathbb{R}^n}\frac{1}{(l-t)^{\frac{n}{2s}}}
\frac{1}{\left(1+\frac{|x-z|^2}{(l-t)^{\frac{1}{s}}}\right)^{\frac{n+2s}{2}}}
\frac{l^b}{|z|^a}1_{\{\frac{1}{2}|x|\le|z|\le2|x|\}}\,dzdl\\&\lesssim|x|^{-a}\left(\int_{t}^{2t}
+\int_{2t}^{\frac{2}{4^{2s}}|x|^{2s}}\right)\int_{0}^{3|x|}\frac{1}{(l-t)^{\frac{n}{2s}}}
\frac{1}{\left(1+\frac{|x-z|^2}{(l-t)^{\frac{1}{s}}}\right)^{\frac{n+2s}{2}}}l^b1_{\{|x-z|\le 3|x|\}}\,dzdl\\&=|x|^{-a}\left(\int_{t}^{2t}+\int_{2t}^{\frac{2}{4^{2s}}|x|^{2s}}\right)
\int_{0}^{3|x|}\frac{1}{(l-t)^{\frac{n}{2s}}}\frac{1}{\left(1+\frac{r^2}{(l-t)^{\frac{1}{s}}}\right)
^{\frac{n+2s}{2}}}l^b r^{n-1}\,drdl\\&\lesssim|x|^{-a}\left(\int_{t}^{2t}+\int_{2t}^{\frac{2}{4^{2s}}|x|^{2s}}\right)l^b
\int_{0}^{\frac{9|x|^2}{(l-t)^{\frac{1}{s}}}}\frac{1}{(1+y)^{\frac{n+2s}{2}}}y^{\frac{n}{2}-1}\,dydl
\\&\lesssim|x|^{-a}t^{1+b}+|x|^{-a}\cdot
\begin{cases}
t^{1+b}&\text{if $b<-1,$}\\
1+\ln\left(\frac{|x|^{2s}}{t}\right)&\text{if $b=-1,$}\\
(|x|^{2s})^{1+b}&\text{if $b>-1,$}
\end{cases}
\end{aligned}
\end{equation}
and
\begin{equation}
\label{b.15}
\begin{aligned}	J_3:&=\int_{t}^{\frac{2}{4^{2s}}|x|^{2s}}\int_{\mathbb{R}^n}\frac{1}{(l-t)^{\frac{n}{2s}}}
\frac{1}{\left(1+\frac{|x-z|^2}{(l-t)^{\frac{1}{s}}}\right)^{\frac{n+2s}{2}}}
\frac{l^b}{|z|^a}1_{\{|z|\ge2|x|\}}\,dzdl\\&\lesssim\int_{t}^{\frac{2}{4^{2s}}|x|^{2s}}
\int_{\mathbb{R}^n}\frac{1}{(l-t)^{\frac{n}{2s}}}
\frac{1}{\left(1+\frac{|z|^2}{4(l-t)^{\frac{1}{s}}}\right)^{\frac{n+2s}{2}}}
\frac{l^b}{|z|^a}1_{\{|z|\ge2|x|\}}\,dzdl\\&\lesssim|x|^{-a}\int_{t}^{\frac{2}{4^{2s}}|x|^{2s}}
\int_{2|x|}^{\infty}\frac{1}{(l-t)^{\frac{n}{2s}}}
\frac{1}{\left(1+\frac{r^2}{4(l-t)^{\frac{1}{s}}}\right)^{\frac{n+2s}{2}}}l^b r^{n-1}\,drdl
\\&\lesssim|x|^{-a}t^{1+b}+|x|^{-a}\cdot
\begin{cases}
t^{1+b}&\text{if $b<-1,$}\\
1+\ln\left(\frac{|x|^{2s}}{t}\right)&\text{if $b=-1,$}\\
(|x|^{2s})^{1+b}&\text{if $b>-1.$}
\end{cases}
\end{aligned}
\end{equation}
For $t $ in the other two intervals, we have
\begin{equation}
\label{b.16}
\begin{split}		
I_{12}&=\int_{\frac{2}{4^{2s}}|x|^{2s}}^{2\cdot4^{2s}|x|^{2s}}
\int_{\mathbb{R}^n}\frac{1}{(l-t)^{\frac{n}{2s}}}
\frac{1}{\left(1+\frac{|x-z|^2}{(l-t)^{\frac{1}{s}}}\right)^{\frac{n+2s}{2}}}\frac{l^b}{|z|^a}1_{\{|z|\ge l^{\frac{1}{2s}}\}}\,dzdl\\&\lesssim\int_{\frac{2}{4^{2s}}|x|^{2s}}^{2\cdot4^{2s}|x|^{2s}}
\int_{\mathbb{R}^n}\frac{1}{l^{\frac{n}{2s}}}
\frac{1}{\left(1+\frac{|x-z|^2}{l^{\frac{1}{s}}}\right)^{\frac{n+2s}{2}}}\frac{l^b}{|z|^a}1_{\{|z|\ge \frac{2^{\frac{1}{2s}}}{4}|x|\}}\,dzdl\\&\lesssim|x|^{-n+2sb-a}
\int_{\frac{2}{4^{2s}}|x|^{2s}}^{2\cdot4^{2s}|x|^{2s}}
\int_{\mathbb{R}^n}\frac{1}{\left(1+\frac{|x-z|^2}{l^{\frac{1}{s}}}\right)^{\frac{n+2s}{2}}}1_{\{|z|\ge \frac{2^{\frac{1}{2s}}}{4}|x|\}}\,dzdl \lesssim|x|^{2s+2sb-a},
\end{split}
\end{equation}
and
\begin{equation}
\label{b.17}
\begin{aligned}		I_{13}&=\int_{2\cdot4^{2s}|x|^{2s}}^{\infty}\int_{\mathbb{R}^n}
\frac{1}{(l-t)^{\frac{n}{2s}}}\frac{1}{\left(1+\frac{|x-z|^2}{(l-t)^{\frac{1}{s}}}\right)^{\frac{n+2s}{2}}}
\frac{l^b}{|z|^a}1_{\{|z|\ge l^{\frac{1}{2s}}\}}\,dzdl\\&\lesssim\int_{2\cdot4^{2s}|x|^{2s}}^{\infty}\int_{\mathbb{R}^n}
\frac{1}{l^{\frac{n}{2s}}}\frac{1}{\left(1+\frac{|z|^2}{4l^{\frac{1}{s}}}\right)^{\frac{n+2s}{2}}}
\frac{l^b}{|z|^a}1_{\{|z|\ge l^{\frac{1}{2s}}\}}\,dzdl \lesssim|x|^{2s+2sb-a}.
\end{aligned}
\end{equation}
	
\noindent (3)	If $\frac{1}{2}t^{\frac{1}{2s}}\le |x|\le 4t^{\frac{1}{2s}},$
\begin{equation}
\label{b.18}
\begin{aligned}	u(x,t)&=\int_{t}^{\infty}\int_{\mathbb{R}^n}\frac{1}{(l-t)^{\frac{n}{2s}}}
\frac{1}{\left(1+\frac{|x-z|^2}{(l-t)^{\frac{1}{s}}}\right)^{\frac{n+2s}{2}}}\frac{l^b}{|z|^a}1_{\{|z|\ge l^{\frac{1}{2s}}\}}\,dzdl\\&\lesssim\int_{t}^{\infty}l^{b-\frac{a}{2s}}\int_{0}^{\infty}
\frac{1}{(1+y)^{\frac{n+2s}{2}}}y^{\frac{n}{2}-1}\,dydl \lesssim t^{1+b-\frac{a}{2s}}.
\end{aligned}
\end{equation}
From \eqref{b.11}-\eqref{b.18} we prove the lemma.
\end{proof}

\end{document}